\documentclass{amsart}
\usepackage[utf8]{inputenc}\usepackage{amsmath}
\usepackage{amsfonts}
\usepackage{amssymb}
\usepackage[english]{babel}
\usepackage{amsmath,amsthm}
\usepackage{amsfonts}
\usepackage{fancyhdr}
\usepackage{setspace}
\usepackage{color}
\usepackage[percent]{overpic}
\usepackage{fix-cm}
\usepackage{graphicx}
\usepackage[all]{xy}
\usepackage{mathrsfs}
\usepackage{mathdots}
\usepackage{amsbsy}
\usepackage{leftidx}
\usepackage[hidelinks]{hyperref}
\usepackage{xcolor}
\usepackage{todonotes}
\usepackage{tikz-cd}
\usepackage[noadjust]{cite}
\usepackage[mark=o]{dynkin-diagrams}
\usepackage{array,multirow}

\hypersetup{
    linktocpage,
    colorlinks,
    linkcolor={magenta!100!black},
    citecolor={cyan!100!black},
    urlcolor={yellow!100!black}
}

\usepackage{tikz}

\theoremstyle{plain}
\newtheorem{teo}{Theorem}[section]
\newtheorem{lema}[teo]{Lemma}
\newtheorem{prop}[teo]{Proposition}

\newtheorem{cor}[teo]{Corollary}

\theoremstyle{definition}
\newtheorem{dfn}[teo]{Definition}
\newtheorem{rem}[teo]{Remark}

\theoremstyle{definition}

\theoremstyle{definition}

\theoremstyle{remark}

\theoremstyle{remark}
\newtheorem{ex}[teo]{Example}

\newcommand{\liep}{\mathfrak{p}}

\newcommand{\lieg}{\mathfrak{g}}
\newcommand{\liea}{\mathfrak{a}}

\newcommand{\g}{\mathsf{G}}
\newcommand{\p}{\mathsf{P}}
\newcommand{\ko}{\mathsf{K}}
\newcommand{\w}{\mathsf{W}}
\newcommand{\n}{\mathsf{N}}
\newcommand{\m}{\mathsf{M}}

\newcommand{\f}{\mathscr{F}}

\newcommand{\rr}{\mathbb{R}}

\newcommand{\zz}{\mathbb{Z}}

\newcommand{\pp}{\mathbb{P}}

\newcommand{\calhu}{\mathcal{H}^\upsilon(X)}
\newcommand{\callu}{\mathcal{L}^\upsilon(X)}
\newcommand{\calpu}{\mathcal{P}^\upsilon(X)}
\newcommand{\calbu}{\mathcal{B}^\upsilon(X)}

\newcommand{\hr}{\textnormal{HR}(\phi)}
\newcommand{\hru}{\textnormal{HR}^\upsilon(\phi)}

\newcommand{\phr}{\mathbb{P}\textnormal{HR}(\phi)}
\newcommand{\phru}{\mathbb{P}\textnormal{HR}^\upsilon(\phi)}

\newcommand{\intm}{\textbf{I}_m(\psi,\widehat{\psi})}

\newcommand{\rintm}{\textbf{J}_m(\psi,\widehat{\psi})}
\newcommand{\rintBM}{\textbf{J}_{m^{\textnormal{BM}}(\psi)}(\psi,\widehat{\psi})}
\newcommand{\pphi}{\mathscr{P}(\phi)}
\newcommand{\ppsi}{\mathscr{P}(\psi)}
\newcommand{\pwpsi}{\mathscr{P}(\widehat{\psi})}

\newcommand{\ephi}{\mathscr{E}(\phi)}
\newcommand{\epsi}{\mathscr{E}(\psi)}

\newcommand{\od}{\mathrm{d}}

\newcommand{\cone}{\mathscr{L}_\rho^\Theta}
\newcommand{\dcone}{(\mathscr{L}_\rho^\Theta)^*}

\newcommand{\cha}{\mathfrak{X}(\Gamma, \g)} %character variety
\newcommand{\ha}{\mathfrak{X}_{\Theta}(\Gamma, \g)} %\Theta Anosov representations
\newcommand{\haz}{\mathfrak{X}^{\tn{Z}}_{\Theta}(\Gamma, \g)}%\Theta Anosov Zariski dense
\newcommand{\Teich}{\tn{Teich}} %% Make all of this coherent

\newcommand{\teichrep}{\mathfrak{T}(S)}
\newcommand{\frakX}{\mathfrak X}%subset of character variety
\newcommand{\Hit}{\textnormal{Hit}}
\newcommand{\Ben}{\textnormal{Ben}}

\newcommand{\Pos}{\textnormal{Pos}}

\newcommand{\bc}{\begin{center}}
\newcommand{\ec}{\end{center}}

\newcommand{\bg}{\partial\Gamma}
\newcommand{\bgs}{\partial^{(2)}\Gamma}
\newcommand{\gh}{\Gamma_{\tn{H}}}

\newcommand{\gsp}{\Gamma_{\tn{SP}}}
\newcommand{\G}{\Gamma}

\newcommand{\PGL}{{\sf PGL}}
\newcommand{\PSL}{{\sf PSL}}

\newcommand{\PO}{{\sf PO}}

\newcommand{\R}{\mathbb R}
\newcommand{\tn}{\textnormal}

\newcommand{\A}{{\sf A}}

\newcommand{\D}{{\sf D}}

\numberwithin{equation}{section}

\begin{document}

\title[Asymmetric metrics for Anosov representations]{Thurston's asymmetric metrics for Anosov representations}
\author{Le\'on Carvajales, Xian Dai, Beatrice Pozzetti and Anna Wienhard}
\address{\newline Le\'on Carvajales \newline Universidad de la Rep\'ublica \newline Facultad de Ciencias Econ\'omicas y de Administraci\'on \newline Instituto de Estad\'istica \newline e-mail: leon.carvajales@fcea.edu.uy \newline \newline Xian Dai \newline Ruhr-Universit\"at Bochum \newline Fakult\"at f\"ur Mathematik \newline e-mail: xian.dai@ruhr-uni-bochum.de \newline \newline Beatrice Pozzetti \newline Ruprecht-Karls Universit\"at Heidelberg \newline Mathematisches Institut \newline e-mail: pozzetti@mathi.uni-heidelberg.de \newline \newline Anna Wienhard \newline Ruprecht-Karls Universit\"at Heidelberg \newline Mathematisches Institut \newline HITS gGmbH, Heidelberg Institute for Theoretical Studies \newline e-mail: wienhard@mathi.uni-heidelberg.de}
\thanks{This investigation is part of the project B5 in the CRC/TRR 191 Funded by the Deutsche Forschungsgemeinschaft (DFG, German Research Foundation) – Project-ID 281071066 – TRR 191. LC, BP and AW acknowledge funding by the DFG through the project 338644254 (SPP2026). LC acknowledges funding by Agencia Nacional de Investigación e Innovación - FCE\_3\_2020\_1\_162840. AW is supported by the European Research Council under ERC-Advanced Grant 101018839. This work is supported by the Deutsche Forschungsgemeinschaft under Germany's Excellence Strategy EXC-2181/1 - 390900948 (the Heidelberg STRUCTURES Cluster of Excellence).}

\maketitle

\begin{abstract}
We provide a good dynamical framework allowing to generalize Thurston's asymmetric metric and the associated Finsler norm from Teichmüller space to large classes of Anosov representations. In many cases, including the space of Hitchin representations, this gives a (possibly asymmetric) Finsler distance. In some cases we explicitly compute the associated Finsler norm.
\end{abstract}

\setcounter{tocdepth}{1}
\tableofcontents

\section{Introduction}

Let $S$ be a connected orientable surface without boundary, with finitely many punctures and negative Euler characteristic. The \textit{Teichm{\"u}ller space} $\Teich(S)$ of $S$ is the space of isotopy classes of complete, finite area hyperbolic structures on $S$. For a pair of points $g_1,g_2\in \Teich(S)$, Thurston \cite{ThurstonStretch} introduces the function $$d_{\tn{Th}}(g_1,g_2):=\log\sup_{c}  \bigg(\frac{L_{g_2}(c)}{L_{g_1}(c)}\bigg),$$ \noindent where the supremum is taken over all free isotopy classes $c$ of closed curves in $S$ and, for $g\in \Teich(S)$, the number $L_g(c)$ denotes the length of the unique geodesic in the class $c$, with respect to the metric $g$. In \cite[Theorem 3.1]{ThurstonStretch} Thurston shows that $d_{\tn{Th}}(\cdot,\cdot)$ defines an asymmetric distance on $\Teich(S)$, and investigates many properties of this metric. For instance, he shows (see \cite[Theorem 8.5]{ThurstonStretch}) that $d_{\tn{Th}}(g_1,g_2)$ coincides with the least possible Lipschitz constant of homeomorphisms from $(S,g_1)$ to $(S,g_2)$ isotopic to $\tn{id}_S$, and constructs families of geodesic rays for this metric, called \textit{stretch lines}.

Thurston also constructs a Finsler norm $\Vert\cdot\Vert_{\tn{Th}}$ on the tangent bundle of Teichm\"uller space: For $v\in T_g\Teich(S)$, he sets 
\begin{equation}\label{eq: finsler teich introd}
    \Vert v\Vert_{\tn{Th}}:=\displaystyle\sup_{c} \frac{ \od_g(L_\cdot (c))(v)}{ L_g(c)}.  
\end{equation} \noindent This is indeed a non-symmetric Finsler norm, namely it is non-negative, non degenerate, $(\rr_{\geq 0})$-homogeneous and satisfies the triangle inequality. Moreover, Thurston shows that the path metric on $\Teich(S)$ induced by this Finsler norm coincides with $d_{\tn{Th}}(\cdot,\cdot)$.

Assume now that $S$ is closed. Then $\Teich(S)$ identifies with a connected component $\teichrep$ of the character variety 
$$\mathfrak{X}(\pi_1(S),\mathsf{PSL}(2, \mathbb{R})):=\textnormal{Hom}(\pi_1(S), \mathsf{PSL}(2, \mathbb{R}))/\!\!/\mathsf{PSL}(2, \mathbb{R}).$$ \noindent For a conjugacy class $[\gamma]$ in $\pi_1(S)$ and a point $\rho\in\teichrep$, we set 
$$L_\rho^{2\lambda_1}([\gamma]):=2\lambda_1(\rho(\gamma)),$$ 
\noindent where $\lambda_1(\rho(\gamma))$ denotes the logarithm of the spectral radius of $\rho(\gamma)$. Identifying isotopy classes of closed curves in $S$ with conjugacy classes in $\pi_1(S)$, one deduces from Thurston's result that \begin{equation}\label{e.Th}d^{2\lambda_1}_{\tn{Th}}(\rho_1,\rho_2):=\sup_{[\gamma]\in[\pi_1(S)]} \log \bigg(\frac{L^{2\lambda_1}_{\rho_2}([\gamma])}{L^{2\lambda_1}_{\rho_1}([\gamma])}\bigg)
\end{equation} \noindent defines an asymmetric distance on $\teichrep$. Similarly, one gets an expression for the associated Finsler norm. The main goal of this note is to generalize this viewpoint, constructing asymmetric metrics and Finsler norms in other representation spaces that share many features with $\teichrep$, namely, spaces of \textit{Anosov} representations, with a particular attention to \textit{Hitchin}, \textit{Benoist} and \textit{positive} representations.

\subsection{Results}
For a finitely generated group $\Gamma$ and a semisimple Lie group $\sf G$ of non-compact type, we denote by $\cha$ the character variety
$$\cha:= \textnormal{Hom}(\Gamma, \sf G)/\!\!/\sf G.$$
We furthermore denote by $\liea^+$ a chosen Weyl chamber of $\sf G$, and by $\lambda:\sf G\to \liea^+$ the Jordan projection. A functional $\varphi\in\liea^*$ is \emph{positive on the limit cone} of a representation $\rho\in\cha$ if for all $\gamma\in\G$ of infinite order one has $\varphi(\lambda(\rho(\gamma)))\geq c\|\lambda(\rho(\gamma))\|$ for some $c>0$ and some norm on $\liea$. With this at hand, for any functional $\varphi\in\liea^*$ positive on the limit cone of  $\rho\in\cha$, we can consider its \emph{$\varphi$-marked length spectrum} 
$$ L^\varphi_{\rho}(\gamma):= \varphi(\lambda(\rho(\gamma))),
 $$ and its $\varphi$-\emph{entropy} 
$$h_{\rho}^\varphi:=\displaystyle\limsup_{t\to\infty}\frac{1}{t}\log\#\{[\gamma]\in[\Gamma]:L_\rho^\varphi(\gamma)\leq t\}\in[0,\infty]. $$ 

If $\frakX\subset \cha$ is a subset, let $\varphi\in\liea^*$ be a functional positive on the limit cone of  each representation $\rho\in\frakX$. Naively, one would like to define $d_{\tn{Th}}^{\varphi}: \frakX \times \frakX  \to \rr\cup\{\infty\}$ by 
\begin{equation}\label{e.ThurINTRO}
d_{\tn{Th}}^\varphi(\rho_1,\rho_2):=\log\left(\displaystyle\sup_{[\gamma]\in[\Gamma]}\frac{ L_{\rho_2}^\varphi(\gamma) }{L_{\rho_1}^\varphi(\gamma)}\right)
\end{equation} 
\noindent and prove that it defines an asymmetric metric for some specific choices of $\frakX$. However, in this general setting, there could exist pairs of representations so that the $\varphi$-length spectrum of $\rho_1$ is uniformly larger than the $\varphi$-length spectrum of $\rho_2$: with the above definition, in that situation we would have $d_{\tn{Th}}^\varphi(\rho_1,\rho_2)<0$ (see Remark \ref{rem: AvoidDomination} and references therein). To resolve this issue, we normalize the length ratio by the entropy: $$d_{\tn{Th}}^\varphi(\rho_1,\rho_2):=\log\left(\displaystyle\sup_{[\gamma]\in[\Gamma]}\frac{h_{\rho_2}^\varphi}{h_{\rho_1}^\varphi}\frac{ L_{\rho_2}^\varphi(\gamma) }{L_{\rho_1}^\varphi(\gamma)}\right)$$ \noindent (see Definition \ref{def: asymmetric distance anosov reps} for more details in the case when $\G$ has torsion). Observe that in the case when $\frakX$ is the Teichmüller space,  $h_{\rho}^{2\lambda_1}=1$, and thus this definition is compatible with the one given in Equation \eqref{e.Th}.

By construction $d_{\tn{Th}}^\varphi$ satisfies the triangular inequality. Our first result determines a setting in which such function is furthermore positive and separates points. 
For this we consider the definition of the space  of $\Theta$-\textit{Anosov} representations, an open subset of the character variety $\cha$ depending on a subset $\Theta$ of the set of simple roots $\Pi$ of $\sf G$ (we refer the reader to Section \ref{sec, AnosovReps} for the precise definition).  For any such set $\Theta$ we denote by 
$$\liea_\Theta:=\bigcap_{\alpha\in\Pi\setminus \Theta}\ker \alpha$$
and by $\liea_\Theta^*<\liea^*$ the set functionals invariant under the unique projection $p_\Theta:\liea\to\liea_\Theta$ 
invariant under the subgroup $\w_\Theta$ of the Weyl group of $\g$ fixing $\liea_\Theta$ pointwise.
\begin{teo}[See Theorems \ref{thm: dth for anosov} and \ref{thm:rigidity}]\label{thm:INTROZdense}
Assume that $\g$ is connected, real algebraic, simple and center free. Assume furthermore that $\frakX\subset \cha$ consists only of Zariski dense $\Theta$-Anosov representations. Let $\varphi\in\liea_\Theta^*$ be positive on the limit cone of each representation in $\frakX$, and suppose that an automorphism $\tau:\sf G\to\sf G$ leaving $\varphi$ invariant is necessarily inner. Then $d_{\tn{Th}}^\varphi(\cdot,\cdot)$ defines a (possibly asymmetric) metric on $\frakX$.
\end{teo}

The Thurston distance on the Teichm\"uller space of a closed surface is complete, however in general the distance $d_{\tn{Th}}^\varphi$ might be incomplete also due to the entropy renormalization. This is for example the case for the Teichm\"uller space of surfaces with boundary of variable length.
It would be interesting to investigate the relation between suitable metric completions and subsets of the length spectrum compactification, as introduced in \cite{Parreau}.

Provided we have a good understanding of all possible Zariski closures in a given subset $\frakX\subset\cha$, we can weaken the Zariski density assumption. This is for instance the case for the set of \emph{Benoist representations}. A Benoist representation is a representation $\rho:\G\to\mathsf{PGL}(d+1,\rr)$ that preserves and acts cocompactly on a strictly convex domain $\Omega_\rho\subset\pp(\rr^{d})$. We let $\Ben_d(\G)$ be the space of conjugacy classes of Benoist representations, which by work of Koszul \cite{Koszul} and Benoist \cite{BenoistDivIII} is a union of connected components of the character variety $\frak X(\G,\mathsf{PGL}(d+1,\rr))$. Benoist representations are  $\Theta$-Anosov for $\Theta=\{\alpha_1,\alpha_d\}$, see \cite{BenoistDivI} and \cite[Proposition 6.1]{GW}. In particular, the logarithm of the spectral radius $\lambda_1$ and the \textit{Hilbert length function} $\tn{H}:=\lambda_1-\lambda_{d+1}$ belong to $\liea_\Theta^*$. Here we recall that $\lambda_{d+1}(g)$ denotes the logarithm of the smallest eigenvalue of $g$. 

Since Benoist computed the possible Zariski closures of a Benoist representation \cite{BenoistAutomorphismes}, the argument of Theorem \ref{thm:INTROZdense} can be pushed further to show the following.

\begin{teo}[See Corollary \ref{cor: asymm for benoist hilbert} and Remark \ref{rem: other functionals in benoist components}]\label{thm:INTROBenoist} The following holds:
\begin{enumerate}
    \item 
The function $d_{\tn{Th}}^{\lambda_1}: \Ben_d(\G) \times \Ben_d(\G)  \to \rr$ given by $$d_{\tn{Th}}^{\lambda_1}(\rho,\widehat{\rho}):=\log\left(\displaystyle\sup_{[\gamma]\in[\Gamma]} \frac{h_{\widehat{\rho}}^{\lambda_1}}{h_{\rho}^{\lambda_1}} \frac{ L_{\widehat{\rho}}^{\lambda_1}(\gamma) }{L_{\rho}^{\lambda_1}(\gamma)}\right)$$ \noindent defines a (possibly asymmetric) distance on $\Ben_d(\G)$. 

\item The function $d_{\tn{Th}}^{\tn{H}}: \Ben_d(\G) \times \Ben_d(\G)  \to \rr$ given by $$d_{\tn{Th}}^{\tn{H}}(\rho,\widehat{\rho}):=\log\left(\displaystyle\sup_{[\gamma]\in[\Gamma]} \frac{h_{\widehat{\rho}}^{\tn{H}}}{h_{\rho}^{\tn{H}}} \frac{ L_{\widehat{\rho}}^{\tn{H}}(\gamma) }{L_{\rho}^{\tn{H}}(\gamma)}\right)$$ \noindent is non-negative, and one has $$d_{\tn{Th}}^{\tn{H}}(\rho,\widehat{\rho})=0 \Leftrightarrow \rho=\widehat{\rho} \tn{ or } \widehat{\rho}=\rho^\star,$$ \noindent where $\rho^\star$ is the \tn{contragredient} of $\rho$.
\end{enumerate}
\end{teo}

A similar result holds for a class of representations of fundamental groups of closed real hyperbolic manifolds into $\mathsf{PO}_0(2,q)$ called \textit{AdS-quasi-Fuchsian}. These were introduced by Mess \cite{Mess} and Barbot-M\'erigot \cite{Barbot,BM}. See Corollary \ref{cor: asymm for AdSQF hilbert}.

The renormalization by the entropy in Equation \eqref{e.ThurINTRO} while necessary to ensure positivity, might seem inconvenient: it may be difficult to obtain concrete control on the entropy, and thus the relation between such distance and the best Lipschitz constant of associated equivariant maps is lost. There are, however, natural classes of representations on which the entropy of some explicit functionals in the Levi-Anosov subspace $\liea_\Theta^*$ is constant. For instance, this is the case for the \textit{unstable Jacobian} $\tn{J}_{d-1}:=d\lambda_1+\lambda_{d+1}$ on Benoist components, thanks to work of Potrie-Sambarino \cite[Corollary 1.7]{PS}. In Corollary \ref{cor: unst Jac for benoist hilbert} we define the corresponding metric. Another important example is the case of \emph{Hitchin representations}, the representations in the connected component $\Hit(S,\g)$ of $\mathfrak X(\pi_1(S),\sf G)$, for a split real Lie group $\sf G$ and the fundamental group of a closed surface $S$, containing the composition of a lattice embedding $\pi_1(S)\to\PSL(2,\rr)$ and the principal embedding $\PSL(2,\rr)\to\sf G$ \cite{Lab, FG}.  Hitchin representations are Anosov with respect to the minimal parabolic \cite{FG,GLW}, so that $\liea_\Theta^*=\liea^*$ and the entropy with respect to all simple roots is constant on $\Hit(S,\g)$ and equal to one, when $\g$ is classical \cite{PS,PSW1}. All possible Zariski closures of $\mathsf{PSL}(d,\rr)$-Hitchin representations have been determined by Guichard \cite{Guichard}, and recently a written proof appeared in \cite{sambarino2020infinitesimal}. This result also covers $\mathsf{PSp}(2r,\rr)$ and $\mathsf{PSO}(p,p+1)$-Hitchin representations, but not the Hitchin component of $\mathsf{PSO}_0(p,p)$ (see Subsection \ref{subsec: rigidity hitchin} for details). As we explain in Subsection \ref{subsec: rigidity hitchin}, Sambarino's approach also works in that case. We deduce the following. 

\begin{teo}[See Corollary \ref{cor: asymm for hitchin roots}]\label{thm:INTROHitchin}
Let $\g$ be an adjoint, simple, real-split Lie group of classical type. Let $\alpha$ be any simple root of $\sf G$, with the exception of the roots listed in Table \ref{table:1}. Then the function $d_{\tn{Th}}^{\alpha}: \Hit(S,\g) \times \Hit(S,\g)  \to \rr$ given by $$d_{\tn{Th}}^{\alpha}(\rho,\widehat{\rho}):=\log\left(\displaystyle\sup_{[\gamma]\in[\Gamma]}  \frac{ L_{\widehat{\rho}}^{\alpha}(\gamma) }{L_{\rho}^{\alpha}(\gamma)}\right)$$ \noindent defines an asymmetric distance on $\Hit(S,\sf G)$. 
\end{teo}

\begin{table}[h!]

  \begin{center}
    \begin{tabular}{c|c|c|c} 
      Type & Group & Diagram &Bad roots  \\
      \hline
 
     $\A_{2n-1}$ & $\PSL_{2n}(\rr)$ & $\begin{dynkinDiagram}A{oo.*.oo}\draw[thick] (root 1) to [out=-45, in=-135] (root 5);\draw[thick] (root 2) to [out=-45, in=-135] (root 4);\end{dynkinDiagram}$&$\{\alpha_n\}$
      \\\hline
   
      \multirow{2}{*}{$\D_n$} & $\PO(n,n)$ $\forall n\geq5$ & $\begin{dynkinDiagram}D{**.**oo}\draw[thick] (root 5) to [out=-45, in=45] (root 6);\end{dynkinDiagram}$ &$\{\alpha_1,\ldots, \alpha_{n-2}\}$\\ 
       & $\PO(4,4)$ & $\begin{dynkinDiagram}D{****}
       \end{dynkinDiagram}$&$\{\alpha_1,\ldots, \alpha_{4}\}$\\
    \end{tabular}
    \caption{The  roots marked in black are fixed by a non-trivial automorphism, and are therefore not covered by Theorem \ref{thm:INTROHitchin}.}\label{table:1} 
  \end{center}
\end{table}

Also in this case, even for the bad roots we can understand precisely when two representations have distance zero. See Subsection \ref{s.Zdense} for further families of representations for which we can generalize Theorem \ref{thm:INTROHitchin}; this is notably the case for some connected components of \emph{$\Theta$-positive representations} of fundamental groups of surfaces in $\PO(p,p+1)$ \cite{GWTheta}, which are smooth and conjectured to only consist of Zariski dense representations \cite[Conjecture 1.7]{Collier}.

\medskip
As a second theme in the paper we give an explicit formula for the Finsler norm associated to the distance  on the set $\ha$ of $\Theta$-Anosov representations. More specifically, we introduce a function $\Vert\cdot\Vert_{\tn{Th}}^\varphi:T\ha\to \rr\cup\{\pm\infty\}$ which is defined as follows. For a given tangent vector $v\in T_\rho\ha$, we set $$\Vert v\Vert_{\tn{Th}}^\varphi:= \displaystyle\sup_{[\gamma]\in[\G]} \frac{\od_{\rho}(h_{\cdot}^\varphi)(v)L_\rho^\varphi(\gamma)+h_\rho^\varphi \od_\rho(L_\cdot^\varphi(\gamma))(v)}{h_\rho^\varphi L_\rho^\varphi(\gamma)}. $$ \noindent If $\rho\mapsto h_{\rho}^\varphi$ is constant, then this expression naturally generalizes Thurston's Finsler norm (\ref{eq: finsler teich introd}). We prove

\begin{prop}[See Corollary \ref{cor: link finsler and asymm for reps}]

Let $\{\rho_s\}_{s\in (-1,1)}\subset \ha$ be a real analytic family and set $\rho:=\rho_0$ and $ v:=\left.\frac{\od}{\od s}\right\vert_{s=0}\rho_s$. Then $s\mapsto d_{\tn{Th}}^\varphi(\rho,\rho_s)$ is differentiable at $s=0$ and $$\Vert v\Vert_{\tn{Th}}^\varphi=\left.\frac{\od}{\od s}\right\vert_{s=0}d_{\tn{Th}}^\varphi(\rho,\rho_s).$$
\end{prop}

It is natural to ask whether $\Vert\cdot\Vert_{\tn{Th}}^\varphi$ defines a Finsler norm. In this direction we show:

\begin{teo}[See Corollary \ref{cor: finsler for reps}]\label{thm: finsler INTRO}
Let $\rho\in\ha$ be a point admitting an analytic neighbourhood in $\ha$. Then the function $\Vert\cdot\Vert_{\tn{Th}}^\varphi:T_\rho\ha\to\rr\cup\{\pm\infty\}$ is real valued and non-negative. Furthermore, it is $(\rr_{>0})$-homogeneous, satisfies the triangle inequality and one has $\Vert v\Vert_{\tn{Th}}^\varphi=0$ if and only if \begin{equation}\label{eq: non deg finsler INTRO}
    d_\rho (L_\cdot^\varphi(\gamma))(v)=-\frac{\od_{\rho}(h_{\cdot}^\varphi)(v)}{h_\rho^\varphi}L_\rho^\varphi(\gamma)
\end{equation} \noindent for all $\gamma\in\G$. In particular, if the function $\widehat{\rho}\mapsto h_{\widehat{\rho}}^\varphi$ is constant, then $$\Vert v\Vert_{\tn{Th}}^\varphi=0\Leftrightarrow d_\rho (L_\cdot^\varphi(\gamma))(v)=0$$ \noindent for all $\gamma\in\G$.
\end{teo}

Condition (\ref{eq: non deg finsler INTRO}) has been studied by Bridgeman-Canary-Labourie-Sambarino \cite{BCLS,BCLSSIMPLEROOTS} in some situations. By applying their results we obtain:

\begin{cor}[See Corollaries \ref{cor: finsler norm hitchin first root} and \ref{cor: finsler norm hitchin spectral}]
The functions $\Vert\cdot\Vert^{\alpha_1}_{\tn{Th}}$ and $\Vert\cdot\Vert^{\lambda_1}_{\tn{Th}}$ define Finsler norms on $\Hit_d(S):=\Hit(S,\mathsf{PSL}(d,\rr))$.
\end{cor}

We don't know, in this general setting, if the length metric induced by the Finsler norm $\Vert\cdot\Vert^{\varphi}_{\tn{Th}}$ agrees with the distance $d_{\tn{Th}}^{\varphi}$: indeed it is not clear if the latter distance is geodesic.

Our final result is an application of Labourie-Wentworth's computation of the derivative of some length functions on $\Hit_d(S)$ along some special directions \cite{VariationAlongFuchsian}. By the work of Hitchin \cite{Hitchin_HitchinComponent}, fixing a Riemann surface structure $X_0$ on $S$, we can parametrize $\Hit_d(S)$ by a vector space of holomorphic differentials (of different degrees) over $X_0$. Given an holomorphic differential $q$  of degree $k$, we associate  to a ray $t\mapsto tq$ for $t\geq 0$ a family $\{\rho_t\}_{t\geq 0}$ of Hitchin representations by the above mentioned Hitchin's parametrization. We denote by $v(q)\in T_{X_0}\Hit_d(S)$ its tangent direction at $t=0$. The holomorphic differential $q$ also defines a function $\tn{Re}(q):T^1X_0\to\rr$. Details for this construction will be given in Subsection \ref{subsec: finsler for hitchin}.

\begin{teo}[See Proposition \ref{prop: finsler fuchsian locus hitchin}]\label{thm:INTROFinslerhitchin} There exist constants $C_1$ and $C_2$, only depending on $d$ and $k$, such that for every vector $v=v(q)\in T_{X_0} \Hit_d(S)$ as above, one has $$\Vert v(q)\Vert_{\tn{Th}}^{\lambda_1}=C_1\displaystyle\sup_{[\gamma]\in[\Gamma]} \int   \textnormal{Re}(q) \od \delta_\phi(a_{[\gamma]}) $$\noindent and $$\Vert v(q)\Vert_{\tn{Th}}^{\alpha_1}=C_2\displaystyle\sup_{[\gamma]\in[\Gamma]}  \int \textnormal{Re}(q) \od\delta_\phi(a_{[\gamma]}), $$ \noindent where $\phi$ denotes the geodesic flow of $X_0$, $a_{[\gamma]}\subset T^1X_0$ denotes the $\phi$-periodic orbit coresponding to $[\gamma]$, and $\delta_\phi(a_{[\gamma]})$ denotes the $\phi$-invariant Dirac probability measure supported on $a_{[\gamma]}$.
\end{teo}

\subsection{Outline of the proofs}\label{subsec: outlineINTRO}

The proofs of our main results follow closely the approach by Guillarmou-Knieper-Lefeuvre \cite{GeodesicStretchPressureMetric}, which is based on work of Knieper \cite{KnieperVolumeGrowth} and Bridgeman-Canary-Labourie-Sambarino \cite{BCLS}. In \cite{GeodesicStretchPressureMetric}, the authors work with the space $\mathfrak{M}$ of isometry classes of negatively curved, entropy one Riemannian metrics on a closed manifold $M$. For $g\in\mathfrak{M}$ and an isotopy class $c$ of closed curves in $M$, one may define $L_g(c)$  as we did when $g$ was a point in Teichm\"uller space. Guillarmou-Knieper-Lefeuvre define $$d_{\tn{Th}}(g_1,g_2):= \log \sup_{c} \frac{L_{g_2}(c)}{L_{g_1}(c)},$$ 
\noindent where the supremum is taken over all isotopy classes $c$ of closed curves in $M$. In \cite[Proposition 5.4]{GeodesicStretchPressureMetric} the authors show \begin{equation}\label{eq: ineq GKL}
    d_{\tn{Th}}(g_1,g_2)\geq 0
\end{equation} \noindent for all $g_1,g_2\in\mathfrak{M}$, and moreover \begin{equation}\label{eq: eq GKL}
    d_{\tn{Th}}(g_1,g_2)= 0 \Leftrightarrow L_{g_1}=L_{g_2}.
\end{equation}  Guillarmou-Lefeuvre's Local Length Spectrum Rigidity Theorem \cite[Theorem 1]{LocalLengthRigidity} (see also \cite[Theorem 1.1]{GeodesicStretchPressureMetric}) gives that Equation (\ref{eq: eq GKL}) is equivalent to $g_1=g_2$, provided that these two metrics are sufficiently regular and close enough in some appropriate topology. Hence, $d_{\tn{Th}}(\cdot,\cdot)$ defines an asymmetric metric on a neighbourhood of the diagonal of $\mathfrak{M}'\subset\mathfrak{M}$, where $\mathfrak{M}'$ is the subset of $\mathfrak{M}$ consisting of sufficiently regular metrics (see \cite{LocalLengthRigidity,GeodesicStretchPressureMetric} for details). Guillarmou-Knieper-Lefeuvre also construct an associated Finsler norm \cite[Lemma 5.6]{GeodesicStretchPressureMetric}.

Even though the Local Length Spectrum Rigidity Theorem is a geometric statement, the proofs of (\ref{eq: ineq GKL}) and (\ref{eq: eq GKL}) can be abstracted to a more  general dynamical framework inspired from \cite[Section 3]{BCLS}. We develop this general dynamical framework  in detail in Sections \ref{sec: thermodynamics} and \ref{sec: asymmetric metric and finsler norm for flows}, as well as the specific statements needed for the construction of an asymmetric distance and a Finsler norm in that setting. As we explain, these general constructions can then be applied not only to the space $\mathfrak{M}$ as in Guillarmou-Knieper-Lefeuvre, but also to other geometric settings, such as spaces of Anosov representations.  We expect that this can be applicable in many more geometric contexts.

The general dynamical framework in Guillarmou-Knieper-Lefeuvre's setting arises as follows:  Gromov observed that the geodesic flows of any two $g_1,g_2\in \mathfrak{M}$ are \textit{orbit equivalent} \cite{GromovGeodesicFlow}. Roughly speaking, this means that the two flows have the same orbits, travelled at possibly different ``speeds" (see Subsection \ref{subsec: orbit equiv} for details). The change of speed (or \textit{reparametrization}) is encoded by a positive H\"older continuous function $r=r_{g_1,g_2}$ on the unit tangent bundle $X:=T^1M$ of $M$. To be more precise, the function $r_{g_1,g_2}$ is only well defined up to an equivalence relation, called \textit{Liv\v{s}ic cohomology} (see Definition \ref{dfn: livsic}).  Thus, we  work in the general dynamical setting of studying the ``geometry" of the space $\mathcal{L}_1(X)$ of Liv\v{s}ic cohomology classes of entropy one H\"older functions on $X$ over the geodesic flow $\phi$ of $g_1$.

Since $\phi$ is an Anosov flow, one may study $\mathcal{L}_1(X)$ through the lens of \textit{Thermodynamic Formalism} (see Subsection \ref{subsec:coding and metric Anosov}). Crucial for us is the following rigidity result by Bridgeman-Canary-Labourie-Sambarino \cite[Proposition 3.8]{BCLS} (see Proposition \ref{prop: BCLS renorm int rigidity} below): there exists a distinguished $\phi$-invariant probability measure $m^{\tn{BM}}(\phi)$ so that \begin{equation}\label{eq: BCLS rigidity intro}
    \int r\od m^{\tn{BM}}(\phi) \geq 1
\end{equation} \noindent and equality holds if and only if $r$ is Liv\v{s}ic cohomologous to the constant function $1$, namely  the periods of periodic orbits of $\phi$ and the reparametrized flow by $r$ coincide. Thus \begin{equation}\label{eq: erg optim introd}
    \displaystyle\sup_{m}\int r\od m \geq 1,
\end{equation} \noindent where the supremum is taken over all $\phi$-invariant probability measures, and equality in the above formula holds if and only if $r$ is Liv\v{s}ic cohomologous to $1$. By Proposition \ref{prop:sup of periods and measures}, the quantity in (\ref{eq: erg optim introd}) coincides with the supremum of ratios of periods of periodic orbits for $\phi$ and the reparametrized flow by $r$.
These general dynamical considerations, when applied specifically to reparametrizing functions associated to $g_1,g_2\in\mathfrak{M}$, readily imply (\ref{eq: ineq GKL}) and (\ref{eq: eq GKL}).

Now as in \cite{BCLS} for their construction of a \textit{pressure metric} (see Subsections \ref{subsec: other related work} and \ref{subsec: comparison pressure norm} for a detailed comparison), the above general approach can also be applied to study spaces of Anosov representations. 
We  use  Sambarino's Reparametrizing Theorem \cite{Quantitative} (see Theorem \ref{thm: reparametrizing theorem} below) to map $\ha$  to a space of Liv\v{s}ic cohomology classes of H\"older functions  over  \textit{Gromov's geodesic flow} $\tn{U}\G$ of $\G$. 
More precisely, we  associate to each Anosov representation $\rho$ and each $\varphi\in\liea_\Theta^*$ a H\"older reparametrization of the geodesic flow $\tn{U}\G$ encoding the $\varphi$-spectral data of $\rho$. 
This procedure is  more involved than in the case of negatively curved metrics, not only because it depends on the additional choice of the functional $\varphi$, but also because the entropy of $\varphi$ is, in general, non-constant.
While, when working with the space $\mathfrak{M}$ one can bypass this problem by normalizing the metric, this is not a natural procedure in our setting, this is why the extra normalization appears in the expression for $d_{\tn{Th}}^\varphi(\cdot,\cdot)$ (see Remark \ref{rem: domination flows} for further comments on this point).
Nevertheless, Bridgmeman-Canary-Labourie-Sambarino's rigidity statement (\ref{eq: BCLS rigidity intro}) is adapted to the setting of arbitrary entropy and we deduce \begin{equation}\label{eq: ineq CDPW}
    d_{\tn{Th}}^\varphi(\rho_1,\rho_2)\geq 0
\end{equation} \noindent for all $\rho_1,\rho_2\in\ha$, and moreover \begin{equation}\label{eq: eq CDPW}
    d_{\tn{Th}}^\varphi(\rho_1,\rho_2)= 0 \Leftrightarrow h_{\rho_1}^\varphi L_{\rho_1}^\varphi=h_{\rho_2}^\varphi L_{\rho_2}^\varphi,
\end{equation} \noindent which are the exact analogues of Equations (\ref{eq: ineq GKL}) and (\ref{eq: eq GKL}). 

To finish the proof of Theorems \ref{thm:INTROZdense}, \ref{thm:INTROBenoist} and \ref{thm:INTROHitchin} we need to understand under which conditions one can guarantee \textit{Renormalized Length Spectrum Rigidity}, that is, under which conditions the equality $h_{\rho_1}^\varphi L_{\rho_1}^\varphi=h_{\rho_2}^\varphi L_{\rho_2}^\varphi$  implies that $\rho_1$ and $\rho_2$ are conjugate. As in the case of negatively curved metrics, where length spectrum rigidity is only known to hold locally, this typically requires to restrict to a subset of $\ha$. More precisely, we need to control the Zariski closure $\g_{\rho_i}$ of $\rho_i$, for $i=1,2$. Since central elements and compact factors are invisible to the Jordan projection, we must require that $\g_{\rho_i}$ is center free and without compact factors. Once this is assumed, and if we assume moreover that $\g_{\rho_i}$ is semisimple, renormalized length spectrum rigidity follows essentially from properties of Benoist's limit cone (see Theorem \ref{thm:rigidity} and \cite[Corollary 11.6]{BCLS}). In some special cases, such as Hitchin components and some components of Benoist and positive representations, these arguments can be pushed further to guarantee global rigidity (see Theorem \ref{Thm:rigitity length hitchin} and Section \ref{s.other}). 

We study the Finsler norm on $\ha$ following the same approach, namely, by finding a general dynamical construction inspired by \cite{GeodesicStretchPressureMetric}, and then pulling back this construction to spaces of Anosov representations. Observe, however, that in this case we need a more complicated expression than what's available in \cite{GeodesicStretchPressureMetric} because we cannot assume that the entropy is constant.

We may summarize the above discussion by saying that the results of this paper are obtained by adapting the corresponding constructions in \cite{GeodesicStretchPressureMetric} to the context of Anosov representations: we can rely on the Thermodynamical Formalism, on which part of the constructions in \cite{GeodesicStretchPressureMetric} are based, using the work of Sambarino \cite{Quantitative} and Bridgeman-Canary-Labourie-Sambarino \cite{BCLS}, and the local rigidity statement needed in \cite{GeodesicStretchPressureMetric} is replaced here by rigidity statements for Anosov representations from \cite{BCLS}. One of the strong points of our approach is to find a suitable general setup where both contexts can be encompassed, and which might prove useful for other geometric situations.

\subsection{Other related work}\label{subsec: other related work}

In \cite{BCLS,BCLSSIMPLEROOTS} the authors construct  $\mathsf{Out}(\G)$-invariant analytic Riemannian metrics on $\ha$: they deduce from the aforementioned rigidity result  that the Hessian of the renormalized intersection (\ref{eq: BCLS rigidity intro}) is a semidefinite non-negative form, called the \textit{pressure form}. This can be pulled back to spaces of Anosov representations, sometimes yielding a positive definite form \cite{BCLS,BCLSSIMPLEROOTS}. 
The construction of this paper is different: instead of integrating with respect to a given measure and taking a second derivative, we integrate with respect to all invariant measures (see Subsection \ref{subsec: comparison pressure norm} for more detailed comparisons).

The rigidity result in Equation (\ref{eq: BCLS rigidity intro}) was previously known to hold in other settings. When restricted to geodesic flows of closed hyperbolic surfaces, this is a reinterpretation of Bonahon's Rigidity Intersection Theorem \cite[p. 156]{BonahonCurrents} (see Appendix \ref{appendix: currents} for more details). More generally, that same result was known to hold for pairs of convex co-compact, rank one representations $\rho_1$ and $\rho_2$ of a word hyperbolic group $\G$: see Burger \cite[p. 219]{BurgerManhattan}. Burger's results readily imply that $$d_{\tn{Th}}(\rho_1,\rho_2):=\log\sup_{[\gamma]\in[\Gamma]} \frac{h_{\rho_2}}{h_{\rho_1}}\frac{ L_{\rho_2}(\gamma) }{L_{\rho_1}(\gamma)}$$ \noindent defines an asymmetric distance on a subset of the space of conjugacy classes of convex co-compact representations $\G\to\g$, where $\g$ has real rank one (note that in a rank one situation the choice of a functional $\varphi$ is irrelevant). Burger also relates the number \begin{equation}\label{eq: sup of ratios no entropy}
    \sup_{[\gamma]\in[\Gamma]} \frac{ L_{\rho_2}(\gamma) }{L_{\rho_1}(\gamma)}
\end{equation} \noindent with one of the asymptotic slopes of the corresponding \textit{Manhattan curve}: see \cite[Theorem 1]{BurgerManhattan}.  Gu\'eritaud-Kassel \cite[Proposition 1.13]{GK} extend Burger's asymmetric metric to some not necessarily convex co-compact representations into the isometry group of the real hyperbolic space. They also show that in some situations the value (\ref{eq: sup of ratios no entropy}) coincides with the best possible Lipschitz constant for maps between the two underlying real hyperbolic manifolds.

Our construction of the asymmetric metric is done on a very general dynamical setting, and pulled back to Anosov representations spaces through Sambarino's Reparametrizing Theorem. %We mention here that, f
For reparametrizations of the geodesic flow of a closed surface, a construction with similar flavor was introduced by Tholozan \cite[Theorem 1.31]{ThoHighest}. His construction leads to a symmetric distance, and it is described in terms of the projective geometry of some appropriate Banach space (see \cite{ThoHighest} and Remark \ref{rem: tholozan symmetric} for further details). It would be intriguing to understand the relation between Tholozan's construction and the approach we carry out here.

\subsection{Plan of the paper}

In Section \ref{sec: thermodynamics} we discuss the  dynamical setup, and in Section \ref{sec: asymmetric metric and finsler norm for flows} we construct the asymmetric metric and the corresponding Finsler norm in this general setting. In Section \ref{sec, AnosovReps} we recall the definition and main examples of interest of Anosov representations. In Section \ref{sec: anosov flows and reps} we recall Sambarino's Reparametrizing Theorem. In Section \ref{sec, generalizedThurston} we pull back the construction of Section \ref{sec: asymmetric metric and finsler norm for flows} to spaces of Anosov representations and also discuss the renormalized length spectrum rigidity in general. In Sections \ref{sec: hitchin} and \ref{s.other} we specify the discussion to Hitchin representations, as well as some components of Benoist and positive representations. In Appendix \ref{appendix: currents} we discuss in detail the link between the rigidity statement (\ref{eq: BCLS rigidity intro}) and Bonahon's Rigidity Intersection Theorem.

\subsection{Acknowledgements}

We are grateful to Gerhard Knieper, Rafael Potrie and Andr\'es Sambarino for several helpful discussions and comments.

\section{Thermodynamical formalism}\label{sec: thermodynamics}

We begin by recalling some important terminology and results about the dynamics of topological flows on compact metric spaces. In Subsection \ref{subsec: orbit equiv} we recall the notions of H\"older orbit equivalence and Liv\v{s}ic cohomology. In Subsection \ref{subsec: measures and pressure} we recall the important concept of pressure, and fix some terminology that will be used throughout the paper. In Subsection \ref{subsec:coding and metric Anosov} we recall the notion of \textit{Markov coding} of a topological flow, and state the main consequences of admitting such a coding. We also recall the notion of \textit{metric Anosov flows}, an important class of flows that admit Markov codings. Finally, in Subsection \ref{subsec: inter and renormalized intersection} we recall the notion of \textit{renormalized intersection}, which is central in our study of the asymmetric metric. The exposition follows closely Bridgeman-Canary-Labourie-Sambarino \cite[Section 3]{BCLS}.

\subsection{Topological flows, reparametrizations and (orbit) equivalence}\label{subsec: orbit equiv}

Let $\phi=(\phi_t: X\to X)$ be a H\"older continuous flow on a compact metric space $X$. In this paper we always assume that $\phi$ is \textit{topologically transitive}. This means that $\phi$ has a dense orbit. 

The choice of a continuous function $r:X\to\rr_{>0}$ induces a ``reparametrization" $\phi^r$ of the flow $\phi$. Informally, this is a flow with the same orbits than $\phi$, but travelled at a  different ``speed". To define this notion properly, we first let $\kappa_r:X\times\rr\to\rr$ be given by 
$$\kappa_r(x,t):=\displaystyle\int_{0}^tr(\phi_s(x))\od s.$$
\noindent  The function $\kappa_r(x,\cdot):\rr\to\rr$ is an increasing homeomorphism for all $x\in X$ and therefore admits an (increasing) inverse $\alpha_r(x,\cdot):\rr\to\rr$. That is, we have $$\kappa_r(x,\alpha_r(x,t))=\alpha_r(x,\kappa_r(x,t))=t$$
\noindent for all $x\in X$ and $t\in\rr$.

\begin{dfn}
The \textit{reparametrization} of $\phi$ by a continuous function $r:X\to\rr_{>0}$  is the flow $\phi^r=(\phi^r_t:X\to X)$ defined by the formula
$$\phi_t^r(x):=\phi_{\alpha_r(x,t)}(x)$$
\noindent for all $x\in X$ and $t\in\rr$. We say that $\phi^r$ is a \textit{H\"older reparametrization} of $\phi$ if $r$ is H\"older continuous. We let $\hr$ be the set of H\"older reparametrizations of $\phi$. 
\end{dfn}

The reader may wonder why we choose the function $\alpha_r$ to reparametrize, instead of directly considering the function $\kappa_r$. One reason is the following. Let $\psi\in\hr$, and denote by $r_{\phi,\psi}$ the corresponding reparametrizing function, i.e. $\psi=\phi^{r_{\phi,\psi}}$. Denote by $\mathcal{O}$ the set of periodic orbits of $\psi$ (note that this set is independent of the choice of $\psi$). Given $a\in\mathcal{O}$ we denote by $p_\psi(a)$ the period, according to the flow $\psi$, of the periodic orbit $a$. Then for every $x\in a$ one has the following equality $$\displaystyle\int_{0}^{p_{\phi}(a)}r_{\phi,\psi}(\phi_t(x))\od t=p_{\psi}(a).$$
\noindent Hence, by choosing the function $\alpha_{r_{\phi,\psi}}$ (instead of $\kappa_{r_{\phi,\psi}}$) we avoid a cumbersome formula involving the integral of $1/r_{\phi,\psi}$ when computing the periods of the new flow.

If we take another point $\widehat{\psi}\in\hr$, then $\widehat{\psi}$ is a reparametrization of $\psi$, that is, one has $\widehat{\psi}=\psi^{r_{\psi,\widehat{\psi}}}$ for some positive continuous function $r_{\psi,\widehat{\psi}}$. In fact, an explicit computation shows
\begin{equation}\label{eq: rep funtion fro psi to widehatpsi}
    r_{\psi,\widehat{\psi}}=\frac{r_{\phi,\widehat{\psi}}}{r_{\phi,\psi}}.
\end{equation} \noindent As above, for every $a\in\mathcal{O}$ and every $x\in a$ one has
\begin{equation}\label{eq: integral of reparametrizing over periodic orbit}
 \displaystyle\int_{0}^{p_{\psi}(a)}r_{\psi,\widehat{\psi}}(\psi_t(x))\od t=p_{\widehat{\psi}}(a).
\end{equation}

There are two notions of equivalence between topological flows that we now recall. A H\"older continuous flow $\phi'=(\phi_t':X'\to X')$ on a compact metric space $X'$ is said to be \textit{(H\"older) conjugate} to $\phi$ if there is a (H\"older) homeomorphism $h:X\to X'$ satisfying $$h\circ\phi_t=\phi'_t\circ h$$ \noindent for all $t\in\rr$. A weaker notion is that of orbit equivalence: the flow $\phi'=(\phi'_t:X'\to X')$ is said to be \textit{(H\"older) orbit equivalent} to $\phi$ if it is (H\"older) conjugate to a (H\"older) reparametrization of $\phi$. One can see that every flow in the orbit equivalence class of $\phi$ is topologically transitive.

To single out elements in $\hr$ which are conjugate to $\phi$, one introduces \textit{Liv\v{s}ic cohomology}. To motivate this notion, consider a H\"older continuous function $V:X\to\rr$ of class C$^1$ along $\phi$, and let $$r(x):=\left(\left.\dfrac{\od}{\od t}\right\vert_{t=0}V(\phi_t(x))\right)+1.$$
\noindent If $r$ is positive, then $\phi^r$ is conjugate to $\phi$. Explicitly, if one defines $h(x):=\phi_{V(x)}(x)$, then $$h\circ \phi_t^r=\phi_t\circ h$$ \noindent for all $t\in\rr$. 
\begin{dfn}\label{dfn: livsic}
Two H\"older continuous functions $f,g:X\to\rr$ are said to be  \textit{Liv\v{s}ic cohomologous} (with respect to $\phi$) if there is a H\"older continuous function $V:X\to\rr$ of class C$^1$ along the direction of $\phi$, so that for all $x\in X$ one has $$f(x)-g(x)=\left.\frac{\od}{\od t}\right\vert_{t=0}V(\phi_t(x)).$$
\noindent In that case we write $f\sim_\phi g$, and denote the Liv\v{s}ic cohomology class of $f$ with respect to $\phi$ by $[f]_\phi$. 
\end{dfn}

\subsection{Invariant measures, entropy and pressure}\label{subsec: measures and pressure}

For $\psi\in\hr$ we denote by $\ppsi$ the set of $\psi$-invariant probability measures on $X$. This is a convex compact metrizable space. We also let $\epsi\subset\ppsi$ be the subset consisting of \textit{ergodic} measures, that is, the subset of measures for which $\psi$-invariant measurable subsets have measure either equal to zero or one. The set $\epsi$ is the set of extremal points of $\ppsi$. 

By the Choquet Representation Theorem (see Walters \cite[p. 153]{WalterErgodicBook}), every element $m\in\ppsi$ admits an \textit{Ergodic Decomposition}. This means that there exists a unique probability measure $\tau_m$ on $\epsi$ %(with respect to the Borel sigma algebra) 
such that %$\tau_m(\epsi)=1$ and 
$$\int_X f(x) \mathrm{d}m(x)= \int_{\epsi} \bigg(\int_X f(x) \mathrm{d}\mu(x)\bigg)\mathrm{d}\tau_{m}(\mu)$$ \noindent holds for every continuous function $f$ on $X$.

The set of periodic orbits of $\psi$ embeds into $\ppsi$ as follows: for $a\in\mathcal{O}$, recall that $p_\psi(a)$ is the period of the periodic orbit $a$ according to the flow $\psi$. We denote by $\delta_{\psi}(a)\in\ppsi$ the \textit{Dirac mass} supported on $a$, that is the push-forward of the Lebesgue probability measure on $S^1\cong [0,1]/\sim$ (where $0\sim 1$) under the map $$S^1 \to X: t\mapsto \psi_{p_{\psi}(a)t}(x),$$
\noindent where $x$ is any point in $a$. Note that $\delta_{\psi}(a)\in\epsi$. Using Equation (\ref{eq: integral of reparametrizing over periodic orbit}), we conclude that for every $\widehat{\psi}\in\hr$ one has
\begin{equation}\label{eq: integral of reparametrizing over delta in periodic orbit}
    p_{\widehat{\psi}}(a)=p_{\psi}(a)\displaystyle\int_{X}r_{\psi,\widehat{\psi}}\od\delta_{\psi}(a).
\end{equation}

More generally, for $m\in\ppsi$, the map $m\mapsto\widehat{m}$ given by \begin{equation}\label{eq: iso between ppsi and pwpsi}
    \od \widehat{m}:=\frac{r_{\psi,\widehat{\psi}}\od m}{\int r_{\psi,\widehat{\psi}}\od m}
\end{equation} \noindent defines an isomorphism $\ppsi\cong\pwpsi$.

We now recall the notion of \textit{topological pressure}, which will be central for our purposes.

\begin{dfn}
Let $f: X\to \rr$ be a continuous function (or \textit{potential}). The \emph{topological pressure} (or \emph{pressure}) of $f$ is defined by 
\begin{equation}\label{eq: def pressure}
    \textbf{P}(\phi,f):=\sup\limits_{m\in\pphi}\left(h(\phi,m)+ \int_X f \od m\right),
\end{equation}
\noindent where $h(\phi,m)$ is the \textit{metric entropy} of $m$. 
\end{dfn}

The metric entropy (or \emph{measure theoretic entropy}) $h(\phi,m)$ is defined using $m$-measurable partition of $X$ and is a metric isomorphism invariant (see \cite[Chapter 4]{WalterErgodicBook}). When there is no risk of confusion we will omit the flow $\phi$ in the notation and simply write $\textbf{P}(f)=\textbf{P}(\phi,f)$. 

A special and important case is the pressure of the potential $f\equiv 0$, which is called the \textit{topological entropy} of $\phi$. It is denoted by $h_{\tn{top}}(\phi)$, or simply by $h_\phi$. The topological entropy is a topological invariant: conjugate flows have the same topological entropy. In contrast, the topological entropy is not invariant under reparametrizations. 

A measure $m \in \pphi$ realizing the supremum in Equation (\ref{eq: def pressure}) is called an \emph{equilibrium state} of $f$. An equilibrium state for $f\equiv 0$ is called a \emph{measure of maximal entropy} of $\phi$.

Liv\v{s}ic cohomologous functions share some common invariants defined in thermodynamical formalism.

\begin{rem}\label{LivsicCommon}
If $f: X\to \mathbb{R}$ and $g: X\to \mathbb{R}$ are Liv\v{s}ic cohomologous functions (w.r.t $\phi$), then $\textbf{P}(\phi,f)=\textbf{P}(\phi,g)$ and $m\in\pphi$ is an equilibrium state for $f$ if and only if it is an equilibrium state for $g$. Indeed, if $f\sim_\phi g$ and $m\in\pphi$ then
$$\displaystyle\int_Xf\mathrm{d}m=\displaystyle\int_Xg\mathrm{d}m.$$ \noindent This is a consequence of $\phi$-invariance of $m$ and the Mean Value Theorem for derivatives of real functions.
\end{rem}

The following is well-known and useful.

\begin{prop}[Bowen-Ruelle {\cite[Proposition 3.1]{BowenRuelle}, Sambarino \cite[Lemma 2.4]{Quantitative}}]\label{prop: pressureZero}
Let $\phi=(\phi_t:X\to X)$ be a H\"older continuous flow on a compact metric space $X$ and $r: X \to \mathbb{R}_{>0}$ be a H{\"o}lder continuous function. Then a real number $h$ satisfies 
$$\textnormal{\textbf{P}}(\phi, -hr)=0$$
\noindent if and only if $h=h_{\phi^{r}}$.
\end{prop}

\subsection{Symbolic coding and metric Anosov flows}\label{subsec:coding and metric Anosov}

We now specify an important class of topological flows for which pressure, equilibrium states and Liv\v{s}ic cohomology behave particularly well. The property we are interested in is the existence of a \textit{strong Markov coding} for the flow. Informally speaking, a Markov coding provides a way of modelling the flow by a suspension flow over a shift space. This allows us to obtain many properties about the dynamics of the flow, by studying the corresponding properties at the symbolic level. The reader can find a general introduction on how to model flows by Markov codings and suspension flows in Bowen \cite{Bowen-Symbolic} and Parry-Pollicott \cite[Appendix III]{ZetaFunction_Pollicott}. We give a cursory introduction of suspension flows and Markov partitions here.

Suppose $(\Sigma, \sigma_A)$ is a two-sided shift of finite type. Given a ``roof function" $r: \Sigma \to \mathbb{R}_{>0}$, the \emph{suspension flow} of $(\Sigma, \sigma_A)$ under $r$ is the quotient space
$$\Sigma_r:=\{(x,t)\in \Sigma\times\mathbb{R}: 0\leq t\leq r(x)\}/(x,r(x))\sim (\sigma_A(x), 0)$$
\noindent equipped with the natural flow $\sigma^r_{A,s}(x,t):=(x,t+s)$.

\begin{dfn}
A  \emph{Markov coding} for the flow $\phi=(\phi_t:X \to X)$ is a 4-tuple $(\Sigma, \sigma_A, \pi, r)$ where $(\Sigma,\sigma_A)$ is an irreducible two-sided subshift of finite type, the function $r:\Sigma \to \mathbb{R}_{>0}$ and the map $\pi:\Sigma_r \to X$ are continuous, and the following conditions hold:
\begin{itemize}
    \item The map $\pi$ is surjective and bounded-to-one.
    \item The map $\pi$ is injective on a set of full measure (for any ergodic measure of full support) and on a dense residual set.
    \item For all $t\in\rr$ one has $\pi\circ \sigma^r_{A,t}=\phi_t \circ \pi$.
\end{itemize}
If both $\pi$ and $r$ are H{\"o}lder continuous, we call the Markov coding a \emph{strong Markov coding}.
\end{dfn}

The proof of the following proposition can be found in Sambarino \cite[Lemma 2.9]{Quantitative}.

\begin{prop}\label{prop: coding for reparametrization}
Let $\phi=(\phi_t:X\to X)$ be a topological flow admitting a strong Markov coding. Then every flow in the H\"older orbit equivalence class of $\phi$ admits a strong Markov coding.
\end{prop}

Thanks to the previous proposition, if $\phi$ admits a strong Markov coding, then every element $\psi\in\hr$ also does. This has deep consequences for the dynamics of $\psi$ that we will discuss in this section. However, before doing that we will discuss an important class of topological flows that admit Markov codings, namely, \textit{metric Anosov} flows. This class is important to us because, as proved by Bridgeman-Canary-Labourie-Sambarino \cite[Sections 4 and 5]{BCLS}, every Anosov representation induces a \textit{geodesic flow} which is a topologically transitive and metric Anosov.

Among flows of class C$^1$ on compact manifolds, \textit{Anosov} flows provide an important class exhibiting many interesting dynamical properties. They were introduced by Anosov \cite{Anosov} in his study of the geodesic flow of closed negatively curved manifolds.  Anosov flows were generalized to \textit{Axiom A} flows by Smale \cite{SmaleDifferentiableDynamicalSystems}; we do not give full definitions here and refer the reader to Smale's original paper. An example of an Axiom A flow which is not Anosov is the geodesic flow of a noncompact convex cocompact real hyperbolic manifold, the restriction of the flow to the set of vectors tangent to geodesics in the convex hull of the limit set shares many dynamical properties with Anosov flows, even though this set is not a manifold. In some contexts (and particularly in the setting we are focusing on), C$^1$-regularity is too much to expect; \textit{Metric Anosov} flows form a class that further generalize Axiom A flows to the topological setting and still share many desirable properties with them. They were introduced by Pollicott \cite{PolMetricAnosov}, who also showed that these flows admit a Markov coding, generalizing the corresponding results for Axiom A flows obtained previously by Bowen \cite{Bowen-Symbolic}.

Let $\phi=(\phi_t: X\to X)$ be a continuous flow on a compact metric space $X$. For $\varepsilon>0$, we define the $\varepsilon$-\textit{local stable set} of $x$ by 
$$W_\varepsilon^s(x):=\{y\in X: d(\phi_t x, \phi_t y)\leq \varepsilon, \forall t\geq 0 \textnormal{ and }  d(\phi_t x, \phi_t y) \to 0 \textnormal{ as } t\to \infty\}$$
\noindent and the $\varepsilon$-\textit{local unstable set} of $x$ by
$$W_\varepsilon^u(x):=\{y\in X: d(\phi_{-t} x, \phi_{-t} y)\leq \varepsilon, \forall t\geq 0\textnormal{ and }  d(\phi_{-t} x, \phi_{-t} y) \to 0 \textnormal{ as } t\to \infty\}.$$

\begin{dfn}\label{dfn,metricAnosov} A topological flow $\phi=(\phi_t: X\to X)$ is  \textit{metric Anosov} if the following conditions hold: 
\begin{enumerate}
\item There exist positive constants $C, \lambda, \varepsilon$ such that 
    $$d(\phi_t(x), \phi_t(y))\leq C e^{-\lambda t}d(x,y) \text{   for all $y\in W^{s}_{\varepsilon}(x)$ and $t\geq 0$},$$
\noindent and    
 $$d(\phi_{-t}(x), \phi_{-t}(y))\leq C e^{-\lambda t}d(x,y) \text{   for all $y\in W^{u}_{\varepsilon}(x)$ and $t\geq 0$}.$$

\item There exists $\delta>0$ and a continuous function $v$ on the set 
$$X_\delta:=\{(x,y)\in X\times X: d(x,y)\leq \delta\}$$ such that for every $(x,y)\in X_\delta$, the number $v=v(x,y)$ is the unique value for which 
$W_{\varepsilon}^u(\phi_v x)\cap W_{\varepsilon}^s(y)$ is not empty  consists of a single point, denoted by $\langle x, y \rangle$.
\end{enumerate}
\end{dfn}

\begin{teo}[Pollicott \cite{PolMetricAnosov}]\label{teo: pollicott coding metric anosov}
A topologically transitive metric Anosov flow on a compact metric space admits a Markov coding.
\end{teo}

For the rest of the section, we fix a topologically transitive flow $\phi=(\phi_t:X\to X)$ admitting a strong Markov coding. In this case the entropy of $\phi$ agrees with the exponential growth rate of periodic orbits:
\begin{equation}\label{eq: entropy}
    h_\phi=\displaystyle\lim_{t\to\infty}\frac{1}{t}\log\#\{a\in\mathcal{O}: p_\phi(a)\leq t\}.
\end{equation}
\noindent Moreover this number is positive and finite (see Bowen \cite{PeriodicOrbits-Bowen} and Pollicott \cite{PolMetricAnosov}).

Another useful consequence of the existence of a Markov coding is the  density of $\mathcal{O}$ in $\ephi$. Combined with the Ergodic Decomposition (c.f. Subsection \ref{subsec: measures and pressure}), it provides a nice way of relating invariant measures and periodic orbits.

\begin{teo}\label{teo: periodic orbits dense}
Let $\phi=(\phi_t:X\to X)$ be a topologically transitive flow admitting a strong Markov coding. Then for every  measure $m\in\ephi$ there is a sequence of periodic orbits $\{a_j\}\subset\mathcal{O}$ such that, as $j\to\infty$, $$\delta_\phi(a_j)\to m$$ \noindent in the weak-$\star$ topology.
\end{teo}

\begin{proof}
This is well known in hyperbolic dynamics (see e.g. Sigmund \cite[Theorem 1]{Sigmund_InvariantMeasures} when $\phi$ is Axiom A). We comment briefly on the ingredients of the proof, since we haven't found an explicit reference in our specific setting. 

By Pollicott \cite[p.195]{PolMetricAnosov} there is a $\sigma_A$-invariant ergodic measure $\mu$ on $\Sigma$ so that $m=\pi_*(\widehat{\mu})$, where $\widehat{\mu}$ is the probability measure on $\Sigma_r$ induced by the measure on $\Sigma\times\rr$ given by $$\frac{\mu\otimes \od t}{\int r\od \mu}.$$ \noindent Hence, it suffices to prove that $\mu$ can be approximated by periodic orbits of $\sigma_A$. This is a consequence of two dynamical properties of $\sigma_A$, called \textit{expansiveness} and the \textit{pseudo-orbit tracing property} (see e.g. \cite[Definition 3.2.11]{KH} and \cite[Theorem 1]{Walters}). Indeed, provided these properties Sigmund's argument \cite[Theorem 1]{SigmundHomeos} can be carried out in the present framework.
\end{proof}

With respect to equilibrium states we have the following theorem.

\begin{teo}[Bowen-Ruelle \cite{BowenRuelle}, Pollicott \cite{PolMetricAnosov}, Parry-Pollicott {\cite[Proposition 3.6]{ZetaFunction_Pollicott}}]\label{teo: CohomologousEquilibrium} 
Let $\phi=(\phi_t:X\to X)$ be a topologically transitive flow admitting a strong Markov coding. For every H{\"o}lder continuous function $f:X\to\rr$, there exists a unique equilibrium state $m_f(\phi)$ for $f$ with respect to $\phi$. Furthermore, the equilibrium state is ergodic. Finally, if $g:X\to\rr$ is H\"older continuous and $m_f(\phi)=m_g(\phi)$, then there exists a constant function $c$ so that $f-g\sim_\phi c$.
\end{teo}

The equilibrium state for $f\equiv 0$ is called the \textit{Bowen-Margulis measure} of $\phi$, and denoted by $m^{\tn{BM}}(\phi)$. For Anosov flows, the existence of this measure was proved by Margulis in his PhD Thesis \cite{MarThesis}. Uniqueness was originally conjectured by Bowen \cite{Bowen-Symbolic} and this justifies the name. In a more geometric context, e.g. for the geodesic flow of a convex cocompact real hyperbolic manifold, Sullivan \cite{SulDensity} gave a description of this measure using Patterson-Sullivan theory. Because of this, the measure of maximal entropy in those contexts is sometimes called the \textit{Bowen-Margulis-Sullivan measure}.

If $f\sim_\phi g$ then the integrals of $f$ and $g$ over every periodic orbit coincide. In the present setting we also have a converse statement.

\begin{teo}  [Liv\v{s}ic {\cite{Livsic}}] \label{teo: livsic} Let $\phi=(\phi_t:X\to X)$ be a topologically transitive flow admitting a strong Markov coding. Suppose that  $f$  and  $g$ are two H{\"o}lder continuous functions such that for all $a\in\mathcal{O}$ and all $x\in a$ one has $$\displaystyle\int_0^{p_\phi(a)}f(\phi_t(x))\mathrm{d}t=\displaystyle\int_0^{p_\phi(a)}g(\phi_t(x))\mathrm{d}t.$$ \noindent Then $f\sim_\phi g$.
\end{teo}

A proof of Liv\v{s}ic's Theorem \ref{teo: livsic} can be found in \cite[Theorem 4.3]{Walkden}: even though it is stated for C$^1$ hyperbolic flows, the proof only uses the existence of the Markov partition.

The final property of metric Anosov flows we will need is convexity of the pressure function, and a characterization of its first derivative in terms of equilibrium states.  Let $S$ be a C$^k$ (resp. smooth, analytic) manifold. A family of functions $\{f_s: X\to  \mathbb{R}\}_{s\in S}$ is said to be a C\emph{$^k$ (resp. smooth, analytic) family}, if for all $x\in X$, the function $s\mapsto f_s(x)$ is C$^k$ (resp. smooth, analytic).

\begin{prop}[Parry-Pollicott {\cite[Propositions 4.7, 4.10 and 4.12]{ZetaFunction_Pollicott}}]\label{prop: firstDevPressure}
Let $\phi=(\phi_t:X\to X)$ be a topologically transitive flow admitting a strong Markov coding. Then:

\begin{enumerate}
    \item For every pair of H\"older continuous functions $f,g:X\to\rr$, the function $$s\mapsto\mathbf{P}(\phi,f+sg)$$ \noindent is convex. Furthermore, it is strictly convex if $g$ is not Liv\v{s}ic cohomologous (w.r.t. $\phi$) to a constant function.
    \item Let $\{f_s\}_{s\in(-1,1)}$ be a \tn{C}$^k$ (resp. smooth, analytic) family of $\upsilon$-H{\"o}lder continuous functions on $X$. Then $s\mapsto \mathbf{P}(\phi,f_s)$ is a \tn{C}$^k$ (resp, smooth, analytic) function, and $$\frac{\od \tn{\textbf{P}}(\phi,f_s)}{\od s}\bigg|_{s=0}= \int_X \left(\frac{\od f_s}{\od  s}\bigg|_{s=0}\right) \mathrm{d}m_{f_0},$$ \noindent where $m_{f_0}=m_{f_0}(\phi)$ is the equilibrium state of $f_0$ (w.r.t $\phi$). 
\end{enumerate}
\end{prop}

\subsection{Intersection and renormalized intersection}\label{subsec: inter and renormalized intersection}
Intersection and renormalized intersection provide a way of ``measuring the difference" between two points in $\hr$. The notion of intersection was introduced by Thurston in the context of Teichm\"uller space (see Wolpert \cite{Wolpert}), and then reinterpreted by Bonahon \cite{BonahonCurrents} (see also Appendix \ref{appendix: currents}). Burger \cite{BurgerManhattan} generalized this notion to pairs of convex cocompact representations into Lie groups of real rank equal to one, and noticed a rigid inequality for this number after renormalizing by entropy. Bridgeman-Canary-Labourie-Sambarino \cite[Section 3.4]{BCLS} further generalized this (renormalized) intersection in the abstract dynamical setting we are focusing on. We will use these notions to study the asymmetric distance and Finsler norm in $\hr$ in Section \ref{sec: asymmetric metric and finsler norm for flows}.

\begin{dfn}
Let $\psi,\widehat{\psi}\in\hr$. For $m\in\ppsi$, the $m$-\textit{intersection number} between $\psi,\widehat{\psi}\in \hr$ is defined by $$\intm:=\displaystyle\int_{X}r_{\psi,\widehat{\psi}}\od m,$$
where the positive continuous function $r_{\psi,\widehat{\psi}}$ is given by Equation (\ref{eq: rep funtion fro psi to widehatpsi}).
\end{dfn}

Recall that $\phi$ is a topologically transitive flow admitting a strong Markov coding. Intersection numbers and ratios of periods are linked as follows.

\begin{prop} \label{prop:sup of periods and measures}
For every $\psi,\widehat{\psi}\in\hr$ the following equality holds
$$ \displaystyle\sup_{a\in\mathcal{O}}\frac{p_{\widehat{\psi}}(a)}{p_{\psi}(a)}=\displaystyle\sup_{m\in\ppsi} \mathbf{I}_m(\psi,\widehat{\psi}).$$
\end{prop}

\begin{proof}

The proof follows closely Guillarmou-Knieper-Lefeuvre \cite[Lemma 4.10]{GeodesicStretchPressureMetric}. We include it for completeness.

First of all we observe that
\begin{equation}\label{eq: sup intersections on ergodic}
  \displaystyle\sup_{m\in\ppsi}\intm=\displaystyle\sup_{m\in\epsi}\intm.  
\end{equation}
\noindent Indeed, let $m_0\in\ppsi$ be such that $$\displaystyle\sup_{m\in\ppsi}\intm=\mathbf{I}_{m_0}(\psi,\widehat{\psi}).$$ \noindent By Ergodic Decomposition (c.f. Subsection \ref{subsec: measures and pressure}) we have

 \begin{align*}
 \mathbf{I}_{m_0}(\psi,\widehat{\psi}) &= \displaystyle\int_{\epsi}\left(\displaystyle\int_X r_{\psi,\widehat{\psi}}(x)\mathrm{d}\mu(x)\right)\mathrm{d}\tau_{m_0}(\mu)\\
 &\leq \displaystyle\sup_{m\in\epsi}\intm \times \displaystyle\int_{\epsi}\mathrm{d}\tau_{m_0}(\mu)\\ 
 &= \displaystyle\sup_{m\in\epsi}\intm.
 \end{align*}
\noindent The reverse inequality being trivial, this proves Equality (\ref{eq: sup intersections on ergodic}).

We now prove $$ \displaystyle\sup_{a\in\mathcal{O}}\frac{p_{\widehat{\psi}}(a)}{p_{\psi}(a)}\leq \displaystyle\sup_{m\in\epsi} \intm.$$
\noindent To do that, take a sequence $a_j\in\mathcal{O}$ such that
$$ \displaystyle\sup_{a\in\mathcal{O}}\frac{p_{\widehat{\psi}}(a)}{p_{\psi}(a)}= \displaystyle\lim_{j\rightarrow\infty}\frac{p_{\widehat{\psi}}(a_j)}{p_{\psi}(a_j)}.$$
\noindent Since $\epsi$ is compact we may assume $\delta_{\psi}(a_j) \rightarrow m$ for some $m\in\epsi$. By Equation (\ref{eq: integral of reparametrizing over delta in periodic orbit}) we have
$$\displaystyle\sup_{a\in\mathcal{O}}\frac{p_{\widehat{\psi}}(a)}{p_{\psi}(a)} =\displaystyle\lim_{j\rightarrow\infty}\displaystyle\int_{X} r_{\psi,\widehat{\psi}}\od\delta_{\psi}(a_j)=\displaystyle\int_{X} r_{\psi,\widehat{\psi}}\od m\leq \displaystyle\sup_{m\in\epsi} \intm.$$

To finish the proof, it remains to show $$ \displaystyle\sup_{a\in\mathcal{O}}\frac{p_{\widehat{\psi}}(a)}{p_{\psi}(a)}\geq \displaystyle\sup_{m\in\epsi} \intm.$$
\noindent By Theorem \ref{teo: periodic orbits dense}, given $m\in\epsi$ we may find a sequence $a_j\in\mathcal{O}$ such that $\delta_{\psi}(a_j)\rightarrow m $. Proceeding as above we have
$$ \displaystyle\sup_{a\in\mathcal{O}}\frac{p_{\widehat{\psi}}(a)}{p_{\psi}(a)} \geq \displaystyle\lim_{j\to\infty}\displaystyle\int_{X} r_{\psi,\widehat{\psi}}\od\delta_{\psi}(a_j)=\displaystyle\int_{X} r_{\psi,\widehat{\psi}}\od m =\intm.$$
\noindent The result follows taking supremum over all $m\in\epsi$.
\end{proof}

The supremum $$\displaystyle\sup_{m\in\ppsi} \mathbf{I}_m(\psi,\widehat{\psi})=\displaystyle\sup_{m\in\ppsi} \int r_{\psi,\widehat{\psi}}\od m$$ \noindent is a well studied quantity in dynamics. Indeed, this number and the measure(s) attaining the $\sup$ is the subject of study of \textit{Ergodic Optimization}. A general belief in this area is that ``typically" among sufficiently regular functions, the maximizing measure is unique, and supported on a periodic orbit. See Jenkinson \cite{jenkinson} and references therein for a nice survey. However, for the geometric applications we have in mind these types of generic results are not enough. In the specific case of reparametrizing functions arising from points in the Teichm\"uller space of a closed surface, Thurston gives a description of the measures realizing the $\sup$ above: these are always (partially) supported on a topological lamination on the surface, and this lamination is typically a simple closed geodesic (see \cite[p.4 and Section 10]{ThurstonStretch} for details).

The function $m\mapsto\intm$ is continuous with respect to the weak-$\star$ topology on $\ppsi$. Since $\ppsi$ is compact, Proposition \ref{prop:sup of periods and measures} implies \begin{equation}\label{eq: sup of periods is finite}
    \displaystyle\sup_{a\in\mathcal{O}}\frac{p_{\widehat{\psi}}(a)}{p_{\psi}(a)}<\infty.
\end{equation}

\begin{rem}\label{rem: domination flows}
Thanks to the above remark one may try to use directly the $\log$ of the number in (\ref{eq: sup of periods is finite}) to produce a metric on $\hr$. However, the following problem arises. For a constant function $r=c>1$, we have $$\log\left(\displaystyle\sup_{a\in\mathcal{O}}\frac{p_{\phi}(a)}{p_{\phi^r}(a)}\right)=\log\left(\frac{1}{c}\right)<0.$$
\noindent Hence, the quantity in Equation \eqref{eq: sup of periods is finite} %$$\log\left(\displaystyle\sup_{a\in\mathcal{O}}\frac{p_{\cdot}(a)}{p_{\cdot}(a)}\right)$$
%\noindent 
cannot define a distance in $\hr$. This problem also arises in the  geometric setting we will focus on (c.f. Remark \ref{rem: AvoidDomination}). 
\end{rem}

A way of resolving the above issue, natural from the viewpoint of dynamical systems, is to normalize by the entropy. Together with Proposition \ref{prop:sup of periods and measures}, this motivates the following definition.

\begin{dfn}
Let $\psi,\widehat{\psi}\in\hr$ and $m\in\ppsi$. The $m$-\textit{renormalized intersection} between $\psi$ and $\widehat{\psi}$ is 
$$\rintm:=\frac{h_{\widehat{\psi}}}{h_{\psi}}\intm.$$
\end{dfn}

Considering renormalized intersection fixes the above issue:
\begin{prop}[Bridgeman-Canary-Labourie-Sambarino {\cite[Proposition 3.8]{BCLS}}]\label{prop: BCLS renorm int rigidity}
For every $\psi,\widehat{\psi}\in\hr$ one has $$\mathbf{J}_{m^{\tn{BM}}(\psi)}(\psi,\widehat{\psi})\geq 1.$$
\noindent Moreover, equality holds if and only if $(h_{\widehat{\psi}}r_{\phi,\widehat{\psi}})\sim_\phi( h_{\psi}r_{\phi,\psi})$.
\end{prop}

\begin{proof}
By Equation (\ref{eq: rep funtion fro psi to widehatpsi}) we have $$\mathbf{J}_{m^{\tn{BM}}(\psi)}(\psi,\widehat{\psi})= \frac{h_{\widehat{\psi}}}{h_\psi}  \displaystyle\int \left(\frac{r_{\phi,\widehat{\psi}}}{r_{\phi,\psi}}\right)\od m^{\tn{BM}}(\psi).$$ \noindent Now the statement becomes precisely that of \cite[Proposition 3.8]{BCLS}.
\end{proof}

\section{Asymmetric metric and Finsler norm for flows}\label{sec: asymmetric metric and finsler norm for flows}

As always we assume that $\phi$ is a topologically transitive flow admitting a strong Markov coding. 
We want to use the formula 
$$\log\left(\displaystyle\sup_{a\in\mathcal{O}}\frac{h_{\widehat{\psi}}}{h_\psi}\frac{p_{\widehat{\psi}}(a)}{ p_\psi(a)}\right)=\log\left(\frac{h_{\widehat{\psi}}}{h_\psi}\displaystyle\sup_{a\in\mathcal{O}}\frac{p_{\widehat{\psi}}(a)}{ p_\psi(a)}\right)$$ \noindent to define a distance on a suitable quotient of $\hr$. We begin understanding which pairs are at distance zero:

\begin{lema}\label{lem: equivalence rel in hr}
For $\psi$ and $\widehat{\psi}$  in $\hr$ the following are equivalent:
\begin{enumerate}
    \item For every $a\in\mathcal{O}$,  $h_{\widehat{\psi}}p_{\widehat{\psi}}(a)=h_\psi p_\psi(a)$.
    \item $(h_{\widehat{\psi}}r_{\phi,\widehat{\psi}}) \sim_\phi  (h_\psi r_{\phi,\psi})$.
    \item  $r_{\psi,\widehat{\psi}}\sim_\psi h_\psi/h_{\widehat{\psi}}$.
    \item There exists a constant function $c$ so that $r_{\psi,\widehat{\psi}}\sim_\psi c$.
\end{enumerate}
\end{lema}

\begin{proof}

Since $\psi$ and $\widehat{\psi}$ are topologically transitive and admit a strong Markov coding (c.f. Proposition \ref{prop: coding for reparametrization}), all results from Section \ref{sec: thermodynamics} apply. In particular, the equivalence between (3) and (4) follows from Equation (\ref{eq: entropy}).

The implications (2)$\Rightarrow$(1) and (3)$\Rightarrow$(1) are straightforward. The implications (1)$\Rightarrow$(2) and  (1)$\Rightarrow$(3) hold thanks to Liv\v{s}ic's Theorem \ref{teo: livsic} (applied to $\phi$ and $\psi$ respectively).
\end{proof}

We say that $\psi$ and $\widehat{\psi}$ in $\hr$ are \textit{projectively equivalent} (and denote $\psi\sim\widehat{\psi}$) if any of the equivalent conditions of Lemma \ref{lem: equivalence rel in hr} hold. We denote by $\phr$ the quotient space under this relation, and denote by $[\psi]\in\phr$ the equivalence class of $\psi$.

\subsection{Asymmetric metric on $\phr$}\label{subsec: asymm metric flows}

Define $d_{\tn{\tn{Th}}}: \phr\times\phr  \to \rr$ by 
$$d_{\tn{Th}}([\psi],[\widehat{\psi}]):=\log\left(\displaystyle\sup_{a\in\mathcal{O}}\frac{h_{\widehat{\psi}}}{h_{\psi}} \frac{p_{\widehat{\psi}}(a)}{p_{\psi}(a)}\right),$$
\noindent where $\psi$ and $\widehat{\psi}$ are representatives of $[\psi]$ and $[\widehat{\psi}]$ respectively. Lemma \ref{lem: equivalence rel in hr} guarantees that $d_{\tn{Th}}$ is well-defined, as it does not depend on the choice of these representatives.

\begin{teo}\label{teo: asymmetric distance flows}
The function $d_{\tn{Th}}$ defines a (possibly asymmetric) distance on $\phr$. 
\end{teo}

By ``possibly asymmetric" we mean that there is no reason to expect that the equality $d_{\tn{Th}}([\psi],[\widehat{\psi}])=d_{\tn{Th}}([\widehat{\psi}],[\psi])$  holds for all pairs $[\psi],[\widehat{\psi}]\in\phr$. In fact, in some specific situations it is possible to show that $d_{\tn{Th}}(\cdot,\cdot)$ is indeed asymmetric (c.f. Remark \ref{rem: asymmetric}).

\begin{proof}[Proof of Theorem \ref{teo: asymmetric distance flows}]

Let $[\psi],[\widehat{\psi}]\in\phr$ and pick representatives $\psi,\widehat{\psi}\in\hr$. By Proposition \ref{prop:sup of periods and measures} we have
$$d_{\tn{Th}}([\psi],[\widehat{\psi}])=\log\left(\displaystyle\sup_{m\in \ppsi} \rintm\right).$$
\noindent Proposition \ref{prop: BCLS renorm int rigidity} implies
$$\displaystyle\sup_{m\in \ppsi} \rintm\geq \rintBM\geq 1,$$
\noindent and therefore $d_{\tn{Th}}([\psi],[\widehat{\psi}])\geq 0$. Moreover, if $d_{\tn{Th}}([\psi],[\widehat{\psi}])= 0$, then Proposition \ref{prop: BCLS renorm int rigidity} implies $(h_{\widehat{\psi}}r_{\phi,\widehat{\psi}}) \sim_\phi (h_{\psi}r_{\phi,\psi})$, which by Lemma \ref{lem: equivalence rel in hr} means $[\psi]=[\widehat{\psi}]$. Since the triangle inequality for $d_{\tn{Th}}(\cdot,\cdot)$ is easily verified, the proof is complete. 
\end{proof}

\begin{rem}\label{rem: tholozan symmetric}

When $\phi$ is a (not necessarily H\"older) continuous parametrization of the geodesic flow of a closed orientable surface of genus $g\geq 2$, Tholozan \cite{ThoHighest} defined a symmetric distance in $\phr$ which has similar flavor to our $d_{\tn{Th}}(\cdot,\cdot)$. More precisely, he works in the space of (not necessarily H\"older) continuous reparametrizations of $\phi$ and considers an appropriate equivalence relation on this space, which restricts to $\sim$ in the H\"older setting. Tholozan proves that the quotient space under this equivalence relation sits as an open, weakly proper, convex domain in the projective space of some Banach space. Hence, it carries a natural \textit{Hilbert metric} (see \cite[Proposition 1.29]{ThoHighest} for details). In \cite[Theorem 1.31]{ThoHighest}, he gives an expression for this Hilbert metric which is a symmetrized version of $d_{\tn{Th}}(\cdot,\cdot)$.

\end{rem}

\subsection{Finsler norm}\label{subsection, FinslerNorm}
We now define a Finsler norm $\Vert\cdot\Vert_{\tn{Th}}$ on the ``tangent space" $T_{[\psi]}\phr$ of every $[\psi]\in\phr$, and provide a link with the asymmetric distance $d_{\tn{Th}}(\cdot,\cdot)$ (Proposition \ref{prop: FinslerNorm} below). Recall that a \textit{Finsler norm} on a vector space $V$ is a function $\Vert\cdot\Vert:V\to\rr$ such that for all $v,w\in V$ and all $a\geq 0$ one has:

\begin{itemize}
    \item $\Vert v\Vert \geq 0$, with equality if and only if $v=0$,
    \item $\Vert a v\Vert=a\Vert v\Vert$, and
    \item $\Vert v+w\Vert\leq \Vert v\Vert+\Vert w\Vert$.
\end{itemize}

Before starting we need to make sense of the ``tangent space" $T_{[\psi]}\phr$ (c.f. also \cite[Subsection 3.5.2]{BCLS}). To do this, we express our space of reparametrizations as a level set of the pressure function, and apply Proposition \ref{prop: firstDevPressure} and the Implicit Function Theorem in Banach spaces \cite{ImplicitFunctiontheorem}. We need to be careful though, because the space of H\"older continuous functions on $X$ is not closed in the topology of uniform convergence. To fix this issue, we will fix a H\"older exponent $\upsilon$ and work restricted to the space $\calhu$ of $\upsilon$-H\"older functions. In the geometric applications we have in mind, namely for spaces of Anosov representations, this is not a strong assumption as discussed in \cite[Section 6]{BCLS} (see also Subsection \ref{subsec: finsler for anosov reps} below).

Fix $\upsilon>0$ and endow $\calhu$ with the Banach norm $$\Vert f\Vert_\upsilon:= \Vert f\Vert_\infty+\sup_{x\neq y}\frac{\vert f(x)-f(y)\vert}{d(x,y)^\upsilon},$$ \noindent where $\Vert\cdot\Vert_\infty$ denotes the uniform norm. Let $\calbu\subset\calhu$ be the space of $\phi$-Liv\v{s}ic \textit{coboundaries}, that is, the set of $\upsilon$-H\"older functions on $X$ which are $\phi$-Liv\v{s}ic cohomologus to zero. By Liv\v{s}ic's Theorem \ref{teo: livsic}, $\calbu$ is a closed (vector) subspace of $\calhu$. We  endow the quotient space $\callu:=\calhu/\calbu$ of Liv\v{s}ic cohomology classes in $\calhu$ with the norm $$ [f]_{\phi} \mapsto\inf_{u\in [f]_{\phi}} \Vert u \Vert_{\upsilon},$$ \noindent which by abuse of notations will also be denoted by $\Vert \cdot\Vert_\upsilon$. Note that $(\calhu,\Vert\cdot\Vert_\upsilon)$ is a Banach space.

Let $\hru$ be the set of reparametrizations $\psi\in\hr$ so that $r_{\phi,\psi}\in\calhu$,  and $\phru$ be its projection to $\phr$. Let $[\psi]\in\phru$ be any point and take a representative $\psi\in\hru$ satisfying $h_\psi=1$. By Proposition \ref{prop: pressureZero} we have $$\mathbf{P}(\phi,-r_{\phi,\psi})=0.$$ \noindent Moreover, if $\widehat{\psi}\in[\psi]$ is another representative satisfying $h_{\widehat{\psi}}=1$, Lemma \ref{lem: equivalence rel in hr} states that $r_{\phi,\widehat{\psi}}\sim_\phi r_{\phi,\psi}$. We then have an injective map from $\phru$ to the space $$\calpu:=\left\{[r]_\phi\in\callu: \mathbf{P}(\phi,-r)=0\right\}.$$ \noindent Hence, $\phru$ identifies with the open subset of $\calpu$ consisting of Liv\v{s}ic cohomology classes of pressure zero, strictly positive, $\upsilon$-H\"older continuous functions on $X$. In view of this discussion, throughout this section all representatives $\psi$ of points $[\psi]$ in $\phru$ are assumed to satisfy $h_\psi=1$.

From now on we simply denote $[r]_\phi$ by $[r]$, omitting the underlying flow $\phi$. By Proposition \ref{prop: firstDevPressure}, for any positive $g\in\calhu$ one has $$\od_{[r]}\mathbf{P}(\phi,\cdot)([g])>0.$$ \noindent That same proposition and the Implicit Function Theorem in Banach spaces imply that the tangent space to $\calpu$ at $[r]$ is given by $$T_{[r]}\calpu=\left\{[g]\in\callu:  \int_X g \mathrm{d}m_{-r}=0 \right\},$$ \noindent where $m_{-r}=m_{-r} (\phi)$ denotes the equilibrium state of $-r$ (w.r.t. $\phi$). Since $\phru$ sits as an open subset of $\calpu$, it is natural to define the \textit{tangent space} to $\phru$ at $[\psi]$ by $$T_{[\psi]}\phru:=T_{[r_{\phi,\psi}]}\calpu.$$

We are now ready to define our Finsler norm. 

\begin{dfn} \label{def:Finslernorm}
 Let $[g]$ be a vector in $T_{[\psi]}\phru$. We define
$$\Vert [g] \Vert_{\tn{Th}}:=\displaystyle\sup_{m\in\pphi}\frac{\int g\od m}{\int r_{\phi,\psi}\od m}.$$ 
\end{dfn}
\noindent Note that this is well-defined, i.e. it does not depend on the choice of the representatives $g$ and $r_{\phi,\psi}$ in the respective $\phi$-Liv\v{s}ic cohomology classes (c.f. Remark \ref{LivsicCommon}). Furthermore, by Equation (\ref{eq: iso between ppsi and pwpsi}) we have the following more succinct expression: \begin{equation}\label{eq: finsler succint}
    \Vert[g]\Vert_{\tn{Th}}=\sup_{m\in\ppsi}\int \left(\frac{g}{r_{\phi,\psi}}\right)\od m.
\end{equation}

By definition of the tangent space,  $\Vert[g]\Vert_{\tn{Th}}\geq 0$ Moreover, $(\rr_{>0})$-homogeneity and the triangle inequality are easily verified. Hence, the following shows that $\Vert\cdot\Vert_{\tn{Th}}$ is a Finsler norm.

\begin{lema}\label{lem: norm is non degenerate}
Let $[g]\in T_{[\psi]}\phru$ be such that $\Vert [g]\Vert_{\tn{Th}}=0$. Then $[g]=0$.
\end{lema}

\begin{proof}
To prove the lemma it suffices to show that $g$ is Liv\v{s}ic cohomologous (w.r.t. $\phi$) to a constant function $c$. Indeed, if this is the case, then by Remark \ref{LivsicCommon} we have $$c=\displaystyle\int c\od m_{-r_{\phi,\psi}}=\displaystyle\int g\od m_{-r_{\phi,\psi}}=0.$$ \noindent Hence $[g]=0$ as desired. 

Let us assume by contradiction that $g$ is not Liv\v{s}ic cohomologous to a constant. By Proposition \ref{prop: firstDevPressure} the function $s\mapsto\mathbf{P}(\phi,-r_{\phi,\psi}+sg)$ is then strictly convex and $$\left.\frac{\od}{\od s}\right\vert_{s=0}\mathbf{P}(\phi,-r_{\phi,\psi}+sg)=\displaystyle\int g\od m_{-r_{\phi,\psi}}=0.$$ \noindent Strict convexity implies then $$\mathbf{P}(\phi,-r_{\phi,\psi}+g)>\mathbf{P}(\phi,-r_{\phi,\psi})=0.$$ 

On the other hand, we show that $\Vert[g]\Vert_{\tn{Th}}=0$ implies $\mathbf{P}(\phi,-r_{\phi,\psi}+g)\leq 0$, giving the desired contradiction. Indeed, note that  $$\mathbf{P}(\phi,-r_{\phi,\psi}+g)\leq \displaystyle\sup_{m\in\pphi}\left( h(\phi,m)-\int  r_{\phi,\psi}\od m\right)+\displaystyle\sup_{m\in\pphi} \int  g\od m.$$ \noindent Since $\Vert [g]\Vert_{\tn{Th}}=0$ and $r_{\phi,\psi}$ is positive, we have $$
\displaystyle\sup_{m\in\pphi}\int g\od m\leq 0,$$ \noindent and therefore $$\mathbf{P}(\phi,-r_{\phi,\psi}+g)\leq \displaystyle\sup_{m\in\pphi}\left( h(\phi,m)-\int  r_{\phi,\psi}\od m\right)=\mathbf{P}(\phi,-r_{\phi,\psi})=0.$$

\end{proof}

We now  link the Finsler norm $\Vert\cdot\Vert_{\tn{Th}}$ and the asymmetric distance $d_{\tn{Th}}(\cdot,\cdot)$. A path $\{[\psi^s]\}_{s\in(-1,1)}\subset\phru$ is  \textit{analytic} (resp. C$^k$, \textit{smooth}) if there is an analytic (resp. C$^k$, smooth) path $\{\widetilde{g}_s\}_{s\in(-1,1)}\subset\calhu$ of strictly positive functions so that $\left[\phi^{\widetilde{g}_s}\right]=[\psi^s]$ for all $s\in(-1,1)$.

Pick a path $\{[\psi^s]\}_{s\in(-1,1)}\subset\phru$ of class C$^1$ and let $\{\widetilde{g}_s\}_{s\in(-1,1)}\subset\calhu$ be as above. By Bridgeman-Canary-Labourie-Sambarino \cite[Proposition 3.12]{BCLS}, the function $s\mapsto h_{\phi^{\widetilde{g}_s}}$ is of class C$^1$. Hence, $s\mapsto g_s:=h_{\phi^{\widetilde{g}_s}}\widetilde{g}_s$ is also C$^1$. Furthermore, we have $$\left[\phi^{g_s}\right]=\left[ \phi^{\widetilde{g}_s} \right]=[\psi^s]$$ \noindent for all $s$, and therefore we may choose $\psi^s=\phi^{g_s}$. By construction we have $h_{\psi^s}=1$, that is, $\mathbf{P}(\phi,-g_s)=0$ for all $s\in(-1,1)$ (Proposition \ref{prop: pressureZero}). If we denote $\dot{g}_0:= \left.\frac{\od}{\od s}\right\vert_{s=0} g_{s}$, we have $$\left[\dot{g}_0\right]=\left.\frac{\od}{\od s}\right\vert_{s=0}\left[g_s\right],$$ \noindent and Proposition \ref{prop: firstDevPressure} gives $$0=\displaystyle\int\left(-\dot{g}_0\right)\od m_{-g_0},$$ \noindent where $m_{-g_0}=m_{-g_0}(\phi)$ is the equilibrium state of $-g_0$ (w.r.t. $\phi$). That is, setting $\psi:=\psi^0$ we have $[\dot{g}_0]\in T_{[\psi]}\phru$.

\begin{prop}\label{prop: FinslerNorm}
With the notations above, the function $s\mapsto d_{\tn{Th}}([\psi],[\psi^s])$ is differentiable at $s=0$. Furthermore, one has $$ \left\Vert \left[\dot{g}_0\right]\right\Vert_{\tn{Th}}=\left.\frac{\od }{\od s}\right\vert_{s=0}d_{\tn{Th}}([\psi],[\psi^s]).$$
\end{prop}

\begin{proof}
Compare Guillarmou-Knieper-Lefeuvre \cite[Lemma 5.6]{GeodesicStretchPressureMetric}. Let $$r_s:=\frac{g_s}{r_{\phi,\psi}}=\frac{g_s}{g_0},$$ \noindent which is the reparametrizing function from $\psi$ to $\psi^s$. Note that $$\dot{r}_0:=\left.\frac{\od}{\od s}\right\vert_{s=0} r_{s}=\frac{\dot{g}_0}{r_{\phi,\psi}},$$ and by Equation (\ref{eq: finsler succint}) we have \begin{equation}\label{eq: finsler and dth}
   \left\Vert \left[\dot{g}_0\right]\right\Vert_{\tn{Th}}=\displaystyle\sup_{m\in\ppsi}\int \dot{r}_0\od m.
\end{equation}

On the other hand, let $u(s):=e^{d_{\tn{Th}}([\psi],[\psi^s])}$. Notice $h_{\psi^s}\equiv 1$.  By Theorem \ref{teo: periodic orbits dense}, periodic orbit measures are dense in the space of invariant probability measures.  We therefore have,   $$u(s)=\displaystyle\sup_{m\in\ppsi}\int r_s\od m.$$
\noindent It suffices to show that $u$ is differentiable at $s=0$ and $u'(0)=\Vert[\dot{g}_0]\Vert_{\tn{Th}}$. 
Since $r_0\equiv 1$, we have $$\frac{u(s)-u(0)}{s}=\frac{\displaystyle\sup_{m\in\ppsi}\int r_s\od m-\displaystyle\sup_{m\in\ppsi}\int 1 \od m}{s}=\displaystyle\sup_{m\in\ppsi}\int \left(\frac{r_s-1}{s}\right)\od m,$$ \noindent and thanks to Equation (\ref{eq: finsler and dth}) we need to show $$\displaystyle\lim_{s\to 0}\left(\displaystyle\sup_{m\in\ppsi}\int \left(\frac{r_s-1}{s}\right)\od m\right)= \displaystyle\sup_{m\in\ppsi}\int \dot{r}_0\od m.$$

Fix some $\varepsilon>0$. The Mean Value Theorem implies that $\frac{r_s-1}{s}$ converges uniformly to $\dot{r}_0$ as $s\to 0$. There exists then $\delta>0$ so that, for all $0<\vert s \vert<\delta$ one has $$\sup\limits_{x\in X}\left\vert\frac{r_s(x)-1}{s}- \dot{r}_0(x)\right\vert<\varepsilon.$$ \noindent Fix any $s$ so that $0<\vert s \vert<\delta$. For every $m\in\ppsi$ we have $$\left\vert\int \frac{r_s-1}{s} \od m - \int \dot{r}_0 \od m\right\vert \leq \sup\limits_{x\in X}\left\vert\frac{r_s(x)-1}{s}- \dot{r}_0(x)\right\vert <\varepsilon.$$ \noindent Therefore $$\int \dot{r}_0 \od m-\varepsilon <\int \frac{r_s-1}{s} \od m<\int \dot{r}_0 \od m+\varepsilon,$$ \noindent for all $m\in\ppsi$. Taking supremum over all $m\in\ppsi$ the result follows.

\end{proof}

\begin{rem}\label{rem: first derivative of rint}
\begin{enumerate}
\item Keeping the notations from above, Proposition \ref{prop: FinslerNorm} can be restated as $$\left\Vert \left[\left.\frac{\od}{\od s}\right\vert_{s=0}g_s\right]\right\Vert_{\tn{Th}}=\left.\frac{\od}{\od s}\right\vert_{s=0}\left(\displaystyle\sup_{m\in\ppsi}\mathbf{J}_m(\psi,\psi^s)\right).$$ \noindent We will come back to this equality in Subsection \ref{subsec: comparison pressure norm}, comparing our viewpoint with previous work of Bridgeman-Canary-Labourie-Sambarino \cite{BCLS}.

\item Notice that although $\Vert\cdot\Vert_{\tn{Th}}$ is a Finsler norm induced from the asymmetric distance $d_{\tn{Th}}(\cdot,\cdot)$, it is not clear whether $d_{\tn{Th}}(\cdot,\cdot)$ is the length distance induced from $\Vert\cdot\Vert_{\tn{Th}}$. In the context of Teichm\"uller space (c.f. Remark \ref{rem: asymmetric}), Thurston \cite{ThurstonStretch} shows that $d_{\tn{Th}}(\cdot,\cdot)$ coincides with the length distance induced by the Finsler norm.

\item  The Finsler norm $\Vert\cdot\Vert_{\tn{Th}}$ is, in general, not induced by an inner product. Indeed, in some concrete examples (c.f. Remark \ref{rem: asymmetric}) one may find tangent vectors $[g]$ for which $$\Vert [g]\Vert_{\tn{Th}}\neq\Vert -[g]\Vert_{\tn{Th}}.$$

\end{enumerate}
\end{rem}

\subsection{Comparison with pressure norm}\label{subsec: comparison pressure norm}
Thurston also introduced a Riemannian metric on the Teichm\"uller space of a closed surface $S$, which agrees with the Weil-Petersson metric (see Wolpert \cite{Wolpert}). McMullen \cite{McMullen} reinterpreted this construction using Thermodynamical Formalism, and Bridgeman-Canary-Labourie-Sambarino \cite{BCLS} took inspiration from this to produce a Euclidean norm $\Vert\cdot\Vert_\mathbf{P}$ on $T_{[\psi]}\phru$. We now briefly recall the construction of \cite{BCLS} and point out the difference with our approach.

Let $[\psi]\in\phru$ and $[g]\in T_{[\psi]}\phru$ be a tangent vector. Thanks to Proposition \ref{prop: firstDevPressure}, one has $\left.\frac{\od^2}{\od s^2}\right\vert_{s=0}\mathbf{P}(-r_{\phi,\psi}+sg)\geq 0$. Hence, one may define $$\Vert[g]\Vert_\mathbf{P}:=\sqrt{\frac{\left.\frac{\od^2}{\od s^2}\right\vert_{s=0}\mathbf{P}(-r_{\phi,\psi}+sg)}{\int r_{\phi,\psi}\od m_{-r_{\phi,\psi}}}}.$$ \noindent Work of Ruelle and Parry-Pollicott implies that $\Vert\cdot\Vert_{\mathbf{P}}$ is a norm\footnote{In particular one has to show that $\Vert[g]\Vert_\mathbf{P}=0$ if and only if $[g]=0$.} on $T_{[\psi]}\phru$, called the \textit{pressure norm}. Moreover, this norm is induced from an inner product, and in fact one has $$\Vert[g]\Vert_\mathbf{P}^2=\frac{\displaystyle\lim_{T\to\infty}\frac{1}{T}\int\left(\displaystyle\int_0^T g(\phi_s(x))\od s\right)^2\od m_{-r_{\phi,\psi}}(x)}{\int r_{\phi,\psi}\od m_{-r_{\phi,\psi}}}.$$ \noindent See \cite[Subsection 3.5.1]{BCLS} for details.

As noticed in \cite[Subsection 3.5.2]{BCLS} the pressure norm is related to the $m^{\tn{BM}}(\psi)$-renormalized intersection. Indeed, consider the function $\mathbf{J}_{[\psi]}(\cdot)$ on $\phru$ given by $$\mathbf{J}_{[\psi]}([\widehat{\psi}]):=\rintBM,$$ \noindent where $\psi$ (resp. $\widehat{\psi}$) is a representative of $[\psi]$ (resp. $[\widehat{\psi}]$). One may check that this is a well-defined function, as it does not depend on the choice of these representatives.  Furthermore, by Proposition \ref{prop: BCLS renorm int rigidity} this function has a minimum at $[\psi]$ and therefore its Hessian at $[\psi]$ defines a non-negative symmetric bilinear form on $T_{[\psi]}\phru$. In fact, if we let $\{g_s\}_{s\in(-1,1)}$ be a smooth path as in Proposition \ref{prop: FinslerNorm}, then one has $$\left\Vert\left[\left.\frac{\od}{\od s}\right\vert_{s=0}g_s\right]\right\Vert_\mathbf{P}^2=\left.\frac{\od^2}{\od s^2}\right\vert_{s=0}\mathbf{J}_{[\psi]}([\psi^s]).$$ \noindent See \cite[Proposition 3.11]{BCLS} for details.

Hence, the second derivative of the $m^{\tn{BM}}(\psi)$-renormalized intersection defines an inner product on $T_{[\psi]}\phru$. In contrast, our viewpoint is different: rather than taking a second derivative of the renormalized intersection with respect to a given measure, we take the supremum of renormalized intersections over all measures, and then take a first derivative  (c.f. Remark \ref{rem: first derivative of rint}).

\section{Anosov representations} \label{sec, AnosovReps}

Anosov representations were introduced by Labourie \cite{Lab} for fundamental groups of negatively curved manifolds, and then extended by Guichard-W. \cite{GW} to general word hyperbolic groups. They provide a stable class of discrete representations with finite kernel into semisimple Lie groups, that share many features with holonomies of convex cocompact hyperbolic manifolds. We will briefly recall this notion in Subsection \ref{subsec: anosov reps}, after fixing some notations and terminology in Subsection \ref{subsec: structure lie groups}. In Subsection \ref{subsec: examples anosov reps} we discuss examples. For a more complete account on the state of the art of the field, see e.g. \cite{KasICM,PozBourbaki,WieICM} and references therein.

\subsection{Structure of semisimple Lie groups}\label{subsec: structure lie groups}
Standard references for this part are the books of Knapp \cite{Kna} and Helgason \cite{Hel}.

Let $\g$ be a connected real semisimple algebraic group of non compact type with Lie algebra $\lieg$. Let $\ko$ be a maximal compact subgroup of $\g$ and $\tau$ be the corresponding Cartan involution of $\mathfrak{g}$. Let
$$\mathfrak{p}:=\{v\in\mathfrak{g}: \tau v=-v\}.$$
\noindent We fix a Cartan subspace $\mathfrak{a}\subset\mathfrak{p}$ and let $\m$ be the centralizer of $\liea$ in $\ko$. 

A natural dynamical system one may look at when studying a discrete subgroup $\Delta<\g$, is the right action of $\liea$ on $\Delta\backslash\g/\m$. When $\g$ has real rank equal to one, this action is conjugate to the action of the geodesic flow of the underlying negatively curved manifold. However, in general it may be hard to study the action $\liea\curvearrowright \Delta\backslash\g/\m$. In many situations (including the setting we are aiming for), it proves useful to consider a ``more hyperbolic" dynamical system, namely, the action of the center of the Levi group associated to a parallel set. We now fix the terminology needed to define this dynamical system.

Denote by $\Sigma$ the set of \textit{roots} of $\mathfrak{a}$ in $\lieg$, that is, the set of functionals $\alpha\in\liea^*\setminus\{0\}$ for which the \textit{root space}
$$\mathfrak{g}_{\alpha}:=\{Y\in \mathfrak{g}: [X,Y] =\alpha(X)Y \tn{ for all } X\in\mathfrak{a}\}$$
\noindent is non zero. Fix a positive system $\Sigma^{+}\subset\Sigma$ associated to a closed Weyl chamber $\mathfrak{a}^+\subset\mathfrak{a}$. The set of simple roots for $\Sigma^{+}$ is denoted by $\Pi$.

\begin{ex}\label{ex: roots}
Suppose $\g=\mathsf{PSL}(V)$, where $V$ is a real (resp. complex) vector space of dimension $d\geq 2$. The Lie algebra of $\g$ is the space of traceless linear operators in $V$. Hence every element of $\lieg$ acts on $V$. A maximal compact subgroup is the subgroup of orthogonal (resp. unitary) matrices with respect to an inner (resp. Hermitian inner) product $o$ in $V$. A Cartan subspace $\liea\subset\liep$ is the subalgebra of matrices which are diagonal on a given projective basis $\mathcal{E}$ of $V$ orthogonal with respect to $o$. The choice of a closed Weyl chamber $\liea^+\subset\liea$ corresponds to the choice of a total order $\{\ell_1,\dots,\ell_d\}$ on $\mathcal{E}$. Explicitly, if $\lambda_j(X)$ denotes the eigenvalue of $X\in\liea$ on the eigenline $\ell_j$, the Weyl chamber $\liea^+$ is given by the set of matrices $X\in\liea$ for which $$\lambda_1(X)\geq\dots\geq\lambda_d(X).$$ \noindent For $i\neq j$ we let $\alpha_{i,j}(X):=\lambda_i(X)-\lambda_j(X)$. Then $$\Sigma=\{\alpha_{i,j}:i\neq j\} \tn{ and } \Sigma^+=\{\alpha_{i,j}:i< j\}.$$ \noindent The set of simple roots is $$\Pi=\{\alpha_{i,i+1}:i=1,\dots,d-1\}.$$
\noindent Sometimes we will write the elements of $\Pi$ simply by $\alpha_i:=\alpha_{i,i+1}$.
\end{ex}

Let $\w$ be the \textit{Weyl group} of $\Sigma$. We realize it as $$\w\cong\mathsf{N}_{\ko}(\liea)/\m,$$
\noindent where $\mathsf{N}_{\ko}(\liea)$ is the normalizer of $\liea$ in $\ko$. The group $\w$ acts simply transitively on the set of Weyl chambers in $\liea$, thus there exists a unique element $w_0\in\w$ taking $\mathfrak{a}^{+}$ to $-\mathfrak{a}^{+}$. The \textit{opposition involution} associated to $\mathfrak{a}^+$ is $\iota:=-w_0$.

We will furthermore need the structure of parabolic subgroups of $\g$. Fix a non empty subset $\Theta \subset \Pi$. Consider the subalgebras
$$\mathfrak{p}_{\Theta}:= \lieg_0 \oplus\bigoplus_{\alpha \in\Sigma^{+}} \mathfrak{g}_{\alpha}\bigoplus_{\alpha\in\langle\Pi-\Theta\rangle}\mathfrak{g}_{-\alpha}$$
\noindent and
$$\overline{\mathfrak{p}_{\Theta}}:= \lieg_0 \oplus\bigoplus_{\alpha \in\Sigma^{+}} \mathfrak{g}_{-\alpha}\bigoplus_{\alpha\in\langle\Pi-\Theta\rangle}\mathfrak{g}_{\alpha},$$
\noindent where $\langle\Pi-\Theta\rangle$ denotes the set of positive roots generated by roots in $\Pi-\Theta$. We let $\p_\Theta$ and $\overline{\p}_\Theta$ be the corresponding subgroups of $\g$. Every parabolic subgroup of $\g$ is conjugate to a unique $\p_{\Theta}$, for some $\Theta \subset \Pi$. Note that $\overline{\p}_\Theta$ is conjugate to $\p_{\iota(\Theta)}$, where $$\iota(\Theta):=\{\alpha\circ\iota: \alpha\in\Theta\}.$$
\noindent The parabolic subgroup $\overline{\p}_\Theta$ is  \textit{opposite} to $\p_\Theta$. 

Let $$\f_\Theta:=\g/\p_\Theta \textnormal{ and }\overline{\f}_\Theta:=\g/\overline{\p}_\Theta$$
\noindent be the corresponding \textit{flag manifolds} of $\g$. Two flags $\xi\in\f_\Theta$ and $\overline{\xi}\in\overline{\f}_\Theta$ are  \textit{transverse} if $(\overline{\xi},\xi)$ belongs to $\f^{(2)}_\Theta$, the unique open orbit of the action of $\g$ on $\overline{\f}_\Theta\times\f_\Theta$. We also let $\f:=\f_\Pi$ and $\f^{(2)}:=\f_\Pi^{(2)}$.

\begin{ex}\label{ex: flags in psl}
Let $\g$ be as in Example \ref{ex: roots}. 
When $d=2$ there is only one flag manifold, which identifies with $\mathbb{P}(\rr^2)$ (the projective space of $\rr^2$).
When $d=3$ there are three flag manifolds, namely $\mathbb{P}(\rr^3)$, $\mathbb{P}((\rr^3)^*)$, and $$\mathscr{F}=\{(\xi^1,\xi^2)\in \mathbb{P}(\rr^3)\times\mathbb{P}((\rr^3)^*): \xi^1\subset\xi^2 \},$$ \noindent where in the above formula we have implicitly identified $\mathbb{P}((\rr^3)^*)$ with the Grassmannian of $2$-dimensional subspaces of $\rr^3$.

More generally, for arbitrary $d\geq 2$ the choice of $\Theta$ is equivalent to the choice of a subset $\{1\leq i_1<\dots<i_p\leq d-1\}$, for some $1\leq p\leq d-1$. Then $\mathscr{F}_\Theta$ identifies with the space of \textit{partial flags} indexed by $\Theta$, that is, the space of sequences $\xi$ of the form $(\xi^{i_1}\subset\dots\subset\xi^{i_p})$, where $\xi^{i_j}$ is a linear subspace of $V$ of dimension $i_j$, for all $j=1,\dots,p$. Furthermore, one has $\iota(\Theta)=\{1\leq d-i_p<\dots<d-i_1\leq d-1\}$. A flag $\overline{\xi}\in\overline{\mathscr{F}}_\Theta$ is transverse to $\xi\in\mathscr{F}_\Theta$ if and only if for all $j=1,\dots ,p$ the sum $\overline{\xi}^{d-i_j}+\xi^{i_j}$ is direct. 
\end{ex}

A point in $(\overline{\xi},\xi)\in\f^{(2)}_\Theta$ determines a \textit{parallel set} of the Riemannian symmetric space $X_\g$ of $\g$. It is the union of all parametrized flat subspaces $f$ of $X_\g$ so that the flag associated to $f(\liea^+)$ (resp. $f(-\liea^+)$) belongs to the fiber over $\xi$ (resp. $\overline{\xi}$), for the fibration $\f\to\f_\Theta$ (resp. $\f\to\overline{\f}_\Theta$). When the real rank of $\g$ is equal to $1$, this is just a geodesic of $X_\g$. When $\Theta=\Pi$, it is a maximal flat subspace of $X_\g$. Any parallel set is identified with the Riemannian symmetric space of the Levi subgroup $L_\Theta = \p_\Theta\cap\overline{\p}_\Theta$, a reductive subgroup of $\g$. 

Let $$\mathfrak{a}_{\Theta} := \bigcap_{\alpha \in \Pi -\Theta} \text{ker } \alpha$$ \noindent be the Lie algebra of the center of $L_\Theta =\p_{\Theta}\cap \overline{\p}_{\Theta}$ (in particular, $\liea_\Pi=\liea$). There is a unique projection $p_{\Theta}: \mathfrak{a} \to \mathfrak{a}_{\Theta}$ invariant under the group 
$$\w_{\Theta}:=\{w\in \w: w\vert_{\mathfrak{a}_{\Theta}}=\operatorname{id}_{\mathfrak{a}_{\Theta}}\}.$$ \noindent The dual space $\liea_\Theta^*$ identifies naturally with $\{\varphi\in\liea^*: \varphi\circ p_\Theta=\varphi\}$. We will use this identification throughout the paper.
 
Consider the space $\f_\Theta^{(2)}\times\liea_\Theta$, endowed with the action of $\liea_\Theta$ by translations on the last coordinate. This action commutes with a natural action of $\g$ that we now describe, and the quotient dynamics is the ``more hyperbolic" dynamical system we have referred to at the beginning of this subsection.

Let $\n$ be the \textit{unipotent radical} of $\p=\p_\Pi$, i.e. the connected subgroup of $\g$ associated to the Lie algebra $\sum_{\alpha\in\Sigma^+}\lieg_\alpha$.  The \textit{Iwasawa Decomposition} is
$$\g=\ko\exp(\liea)\n.$$
\noindent   In particular, $\mathscr{F}\cong\ko/\m$ and for $\xi\in\f$ we may find $k\in\ko$ such that $k\m=\xi$. Quint \cite{QuintCocylce} defines a map $\sigma: \g\times \mathscr{F} \to \mathfrak{a}$ by the formula
$$gk=l \exp(\sigma(g,k\m))n,$$
\noindent where $n\in\n$ and $l\in\ko$. Quint \cite[Lemme 6.11]{QuintCocylce} also shows that $p_{\Theta}\circ\sigma:\g \times \f \to \liea_{\Theta}$ factors through a map $\sigma_{\Theta}: \g \times \f_{\Theta} \to \liea_{\Theta}$. For every $g, h\in \g$ and $\xi\in \mathscr{F}_{\Theta}$ one has
$$\sigma_{\Theta}(gh,\xi)=\sigma_{\Theta}(g,h\cdot\xi)+\sigma_{\Theta}(h,\xi).$$
\noindent  The map $\sigma_\Theta$ is called the $\Theta$-\textit{Busemann-Iwasawa cocycle} of $\g$. Observe that the action of $\liea_\Theta$ on $\f_\Theta^{(2)}\times\liea_\Theta$ commutes with the action of $\g$ given by $$g\cdot(\overline{\xi},\xi,X):=(g\cdot \overline{\xi},g\cdot \xi,X-\sigma_\Theta(g,\xi)).$$

\begin{rem}
The Busemann-Iwasawa cocycle of $\g$ is a vector valued version of the \textit{Busemann function} of the Riemannian symmetric space $X_\g$ of $\g$. Indeed, when $\g$ has real rank equal to one, then $\mathscr{F}$ identifies with the visual boundary $\partial X_\g$ of $X_\g$. Let $o\in X_\g$ be the point fixed by $\ko$. After identifying $\liea$ with $\rr$ suitably, one has $$\sigma(g,\xi)=b_\xi(o,g^{-1}\cdot o),$$
\noindent where $b_\cdot(\cdot,\cdot):\partial X_\g\times X_\g\times X_\g\to\rr$ is the Busemann function. A similar interpretation holds in higher rank (c.f. \cite[Lemme 6.6]{QuintCocylce}). 
\end{rem}

In Section \ref{sec: anosov flows and reps} we will consider a flow space which is even better behaved than the action of $\liea_\Theta$ associated to a parallel set. It will be induced by the choice of a functional in $\liea_\Theta^*$. Natural generators of $\liea_\Theta^*$ are the \textit{fundamental weights} associated to $\Theta$, whose definition we now recall.

Denote by $(\cdot,\cdot)$ the inner product on $\liea^*$ dual to the Killing form of $\lieg$. For $\varphi,\psi\in\liea^*$ set $$\langle\varphi,\psi\rangle:=2\frac{(\varphi,\psi)}{(\psi,\psi)}.$$
\noindent Given $\alpha\in\Pi$, the corresponding \textit{fundamental weight} is the functional $\omega_\alpha\in\liea^*$ defined by the formulas $ \langle\omega_\alpha,\beta\rangle=\delta_{\alpha\beta}$ for $\beta\in\Pi$. One has \begin{equation}\label{eq: fund weight and ptheta}
\omega_\alpha\circ p_\Theta =\omega_\alpha    
\end{equation}\noindent for all $\alpha\in\Theta$ (c.f. Quint \cite[Lemme II.2.1]{QuiDivergence}). In particular, we have $\omega_\alpha\in\liea_\Theta^*$.

Fundamental weights are related to a special set of linear representations of $\g$ introduced by Tits \cite{Tits}. If $\Lambda:\g\to\mathsf{PGL}(V)$ is an irreducible representation, a functional $\chi\in\liea^*$ is a \textit{weight} of $\Lambda$ if the \textit{weight space} $$V_\chi:=\{v\in V: \Lambda(\exp(X))\cdot v=e^{\chi(X)}v, \tn{ for all } X\in\liea\}$$ \noindent is non zero. Tits \cite{Tits} shows that there exists a unique weight $\chi_\Lambda$ which is maximal with respect to the order given by $\chi\geq \chi'$ if $\chi-\chi'$ is a linear combination of simple roots with non-negative coefficients. The functional $\chi_\Lambda$ is called the \textit{highest weight} of $\Lambda$ and the representation is  \textit{proximal} if the associated weight space $V_{\chi_\Lambda}$ is one dimensional. The next proposition is useful.

\begin{prop}[Tits \cite{Tits}]\label{prop: tits}
For every $\alpha\in\Pi$ there exists a finite dimensional real vector space $V_\alpha$ and a proximal irreducible representation $\Lambda_\alpha:\g\to\mathsf{PGL}(V_\alpha)$ such that the highest weight $\chi_\alpha=\chi_{\Lambda_\alpha}$ is of the form $k_\alpha\omega_\alpha$, for some integer $k_\alpha\geq 1$.
\end{prop}

We fix from now on a set of representations $\{\Lambda_\alpha\}_{\alpha\in\Pi}$ as in Proposition \ref{prop: tits}. Observe that for all $\alpha\in\Theta$ we have\begin{equation}\label{eq: highest weight and ptheta}
\chi_\alpha\circ p_\Theta =\chi_\alpha,    
\end{equation} \noindent and therefore $\chi_\alpha$ belongs to $\liea_\Theta^*$.

We conclude recalling the definitions of Cartan and Jordan projections of $\g$ for later use. The \textit{Cartan projection} of $g\in\g$ is the unique element $\mu(g)\in\mathfrak{a}^+$ satisfying $$g\in\ko\exp(\mu(g))\ko.$$
\noindent The \textit{Jordan projection} of $g$ is defined by $$\lambda(g):=\displaystyle\lim_{n\to\infty}\frac{\mu(g^n)}{n}.$$
\noindent One may show that for all $\alpha\in\Pi$ and all $g\in\g$ one has \begin{equation}\label{eq: spectral radious and fund weight}
    \lambda_1(\Lambda_\alpha(g))=\chi_\alpha(\lambda(g))=k_\alpha\omega_\alpha(\lambda(g)),
\end{equation} \noindent where $\lambda_1(\Lambda_\alpha(g))$ denotes the logarithm of the modulus of the highest eigenvalue of $\Lambda_\alpha(g)$.

We denote $$\mu_{\Theta}:=p_{\Theta}\circ \mu \textnormal{ and } \lambda_{\Theta}:=p_{\Theta}\circ \lambda.$$

\subsection{Anosov representations and their length functions}\label{subsec: anosov reps}

We now define Anosov representations and their corresponding length functions and entropies. The definition that we present here is not the original definition, but an equivalent one established in \cite{KLPanosovcharacterizations,GGKW,BPS}.

Let $\Gamma$ be a finitely generated group and $\vert\cdot\vert$ be the word length associated to a finite generating set (that we fix from now on).

\begin{dfn}\label{def: anosov rep}
Let $\Theta\subset\Pi$ be a non empty set. A representation $\rho: \Gamma \to \g$ is  $\p_\Theta$-\textit{Anosov} (or $\Theta$-\textit{Anosov}) if there exist positive constants $C$ and $c$ such that for all $\alpha\in\Theta$ one has
$$\alpha(\mu(\rho(\gamma)))\geq C\vert\gamma\vert-c.$$
\noindent When $\Theta=\Pi$ and $\g$ is split, $\rho$ is sometimes called \textit{Borel-Anosov}. When $\g=\mathsf{PSL}(V)$ with $V$ as in Example \ref{ex: roots}, $\{\alpha_1\}$-Anosov representations are also called \textit{projective Anosov}.
\end{dfn}

An immediate consequence of Definition \ref{def: anosov rep} is that Anosov representations are quasi-isometric embeddings from $\Gamma$ to $\g$. In particular, they are discrete and have finite kernels. A deeper consequence is a theorem by Kapovich-Leeb-Porti \cite[Theorem 1.4]{KLPmorse} (see also \cite[Section 3]{BPS}): if $\rho:\Gamma\to\g$ is $\Theta$-Anosov then $\Gamma$ is word hyperbolic. Throughout the paper we shall assume that $\Gamma$ is non elementary and denote by $\partial\Gamma$ its Gromov boundary. We also let $\bgs$ be the space of ordered pairs of different points in $\bg$. Every infinite order element $\gamma\in\Gamma$ has a unique attracting (resp. repelling) fixed point in $\bg$, denoted by $\gamma_+$ (resp. $\gamma_-$). We let $\gh\subset\Gamma$ be the subset consisting of infinite order elements. The conjugacy class of $\gamma\in\Gamma$ is denoted by $[\gamma]$, and the set of conjugacy classes of elements of $\Gamma$ (resp. $\gh$) will be denoted by $[\Gamma]$ (resp. $[\gh]$).

A central feature of $\Theta$-Anosov representations is that they admit \textit{limit maps}. By definition, these are H\"older continuous, $\rho$-equivariant, dynamics preserving maps
$$\xi_\rho: \partial \Gamma \to \mathscr{F}_\Theta \textnormal{ and } \overline{\xi}_\rho: \partial \Gamma \to \overline{\mathscr{F}}_\Theta,$$
\noindent which are moreover \textit{transverse}, that is, for every $x\neq y$ in $\partial\Gamma$ one has
$$ (\overline{\xi}_\rho(x),\xi_\rho(y))\in\f^{(2)}_\Theta. $$
\noindent The limit maps exist and are unique (see \cite{BPS,GGKW,KLPanosovcharacterizations} for details).

\begin{ex}\label{ex: limit map in grassmanian}
Let $\g$ be as in Example \ref{ex: roots} and $\Theta=\{1\leq i_1<\dots<i_p\leq d-1\}$ for some $1\leq p\leq d-1$ (c.f. Example \ref{ex: flags in psl}). For $j=1,\dots ,p$, we let $$\xi_\rho^{i_j}:\bg\to\mathbb{G}_{i_j}(V)$$ \noindent be the $i_j$-coordinate of $\xi_\rho$ into the Grassmannian $\mathbb{G}_{i_j}(V)$ of $i_j$-dimensional subspaces of $V$.
\end{ex}

The set of $\Theta$-Anosov representations from $\Gamma$ to $\g$ is an open subset of the space of all representations $\Gamma\to\g$. This is a consequence of the original definition  \cite{Lab,GW}. Indeed, the original definition requires \textit{a priori} the word hyperbolicity of $\Gamma$ and the existence of the limit maps, with them one constructs a flow space which, by definition, satisfies certain form of uniform hyperbolicity. General results in hyperbolic dynamics give that this is an open condition.

Projective Anosov representations are very general:

\begin{prop}[Guichard-W. {\cite[Proposition 4.3]{GW}}]\label{prop: anosov and tits}
Let $\rho:\Gamma\to\g$ be  $\Theta$-Anosov. Then for every $\alpha\in\Theta$ the representation $\Lambda_\alpha\circ\rho:\Gamma\to\mathsf{PGL}(V_\alpha)$ is projective Anosov.
\end{prop}

We denote by $\ha$ the space of conjugacy classes of $\p_\Theta$-Anosov representations from $\Gamma$ to $\g$. Length functions and entropies are important invariants to study this space. By work of Sambarino \cite{HyperconvexRepsExponentialGrowth,Quantitative} that we will recall in Section \ref{sec: anosov flows and reps}, they provide a way of associating to each $\rho\in\ha$ certain flow space as in Sections \ref{sec: thermodynamics} and \ref{sec: asymmetric metric and finsler norm for flows}, and therefore one may use the Thermodynamical Formalism to study $\ha$. To define length functions and entropies properly we need to recall the definition of a fundamental object, introduced by Benoist \cite{Benoist_AsymtoticLinearGroups} for general discrete subgroups of $\g$.

\begin{dfn}
The $\Theta$-\textit{limit cone} of $\rho\in \ha$ is the smallest closed cone $\cone\subset\mathfrak a_\Theta ^+$ containing the set $\{\lambda_{\Theta}(\rho (\gamma)): \gamma\in \Gamma\}$. The \textit{limit cone} $\mathscr{L}_\rho$ of $\rho$ is the $\Pi$-limit cone.
\end{dfn}

In the above definition we abuse notations, because $\rho$ is a conjugacy class of representations. However, it is clear that the $\Theta$-limit cone is independent of the choice of a representative in this conjugacy class. 

Under the assumption that $\rho$ is Zariski dense, Benoist \cite{Benoist_AsymtoticLinearGroups} showed that $\mathscr{L}_\rho$ is a convex cone with non empty interior\footnote{In fact, Benoist shows this result for any Zariski dense discrete subgroup of $\g$.}. Since $p_\Theta$ is a surjective linear map, the same properties hold for the $\Theta$-limit cone.

Let $$\dcone:=\{\varphi \in \mathfrak{a}_{\Theta}^{*}: \varphi|_{\cone}\geq 0\}$$
\noindent be the \textit{dual cone}. We denote by $\text{int} (\dcone)$ the interior of $\dcone$, that is, the set of functionals in $\liea_\Theta^*$ which are positive on $\cone\setminus\{0\}$.

Fix a functional $$\varphi \in \bigcap_{{\rho}\in\ha} \text{int} (\dcone).$$
\noindent The above intersection is non empty. For example, it contains $\lambda_1$ and more generally $\omega_\alpha$ for all $\alpha\in\Pi$.

\begin{dfn}
The $\varphi$-\textit{marked length spectrum} (or simply $\varphi$-\textit{length spectrum}) of $\rho\in\ha$ is the function $L^\varphi_{\rho}:\Gamma \to \rr_{\geq 0} $ given by
    \begin{align*}
    L^\varphi_{\rho}&(\gamma):= \varphi(\lambda_{\Theta}(\rho(\gamma))).
    \end{align*}
    \end{dfn}
\noindent Observe that for a $\Theta$-Anosov representation $\rho$,  $L_\rho^\varphi(\gamma)>0$ if and only if $\gamma\in\gh$ (that is, if it has infinite order). Furthermore the $\varphi$-length spectrum is invariant under conjugation in $\Gamma$ and therefore descends to a function $[\Gamma]\to\rr_{\geq 0}$. We will often abuse notations and denote this function by $L^\varphi_{\rho}$ as well.

\begin{dfn}
The $\varphi$-\textit{entropy} of $\rho$ is defined by $$h_\rho^\varphi:=\displaystyle\limsup_{t\to\infty}\frac{1}{t}\log\#\{[\gamma]\in[\Gamma]:L_\rho^\varphi(\gamma)\leq t\}\in[0,\infty].$$
\end{dfn}

The $\varphi$-entropy of $\rho$ was introduced by Sambarino \cite{HyperconvexRepsExponentialGrowth,Quantitative}, who showed that this quantity is defined by a true limit, is positive, finite, and coincides with the topological entropy of a suitable flow associated to $\rho$ and $\varphi$. We will briefly recall these results and facts in Section \ref{sec: anosov flows and reps}. 

\begin{ex}
Here is a concrete set of length spectra that will be of interest (the corresponding entropies are named accordingly). Let $\g=\mathsf{PSL}(V)$ with $V$ as in Example \ref{ex: roots}:

\begin{itemize}\label{list, Lengths}
    \item If $\rho:\Gamma\to\g$ is $\Theta$-Anosov and $\alpha_i\in\Theta$ belongs to $\liea_\Theta^*$ (this is always the case if $\Theta=\Pi$), then $L_\rho^{\alpha_i}$ is called the $i^{\tn{\tn{th}}}$-\textit{simple root length spectrum} of $\rho$.
    \item If $\rho:\Gamma\to\g$ is projective Anosov, then $L_\rho^{\alpha_{1,d}}$ is called the \textit{Hilbert length spectrum} of $\rho$. We denote it by $L_\rho^\mathrm{H}$.
    \item If $\rho:\Gamma\to\g$ is projective Anosov, then $L_\rho^{\lambda_1}$ is called the \textit{spectral radius length spectrum} of $\rho$.
\end{itemize}

\end{ex}

\subsection{Examples of Anosov representations}\label{subsec: examples anosov reps}

Schottky type constructions as in Benoist \cite{BenPropres} provide basic examples of $\Theta$-Anosov representations of free groups. In this subsection we give a list of other examples that will be of interest to us.

\begin{ex}[Teichm\"uller space] 
Let $S$ be a connected, closed, orientable surface of genus $\geq 2$ and $\Gamma=\pi_1(S)$ be its fundamental group (in short, $\Gamma$ is a \textit{surface group}). The \textit{Teichm\"uller space} of $S$ is the space of isotopy classes of Riemannian metrics on $S$ of constant curvature equal to $-1$. Throughout the paper we identify this space with a connected component $\teichrep$ of the space of $\mathsf{PSL}(2,\rr)$-conjugacy classes of faithful and discrete representations $\Gamma\to\mathsf{PSL}(2,\rr)$. By the \v{S}varc-Milnor Lemma (see \cite[Proposition 19 of Ch. 3]{GdlH}), representations in $\teichrep$ are Anosov. 
\end{ex}

\begin{ex}[Hitchin representations]\label{ex: hitchin and positive}
An important class of Anosov representations is given by Hitchin representations. For every split real Lie group $\sf G$, we denote by $\tau:\PSL(2,\R)\to\sf G$ the \emph{principal embedding} \cite{Kos}, which is well defined up to conjugation. In the case of $\mathsf{G}=\PSL(d,\R)$, $\tau$ gives the unique irreducible linear representation of $\PSL(2,\R)$. It was proven by Labourie \cite{Lab} and Fock-Goncharov \cite{FG} that, given the holonomy $\rho_h:\G\to\PSL(2,\R)$ of any chosen hyperbolization $h$ of $S$, the entire connected component of $\tau\circ\rho_h$ consists of Borel-Anosov representations. This component is usually referred to as the \emph{Hitchin component}. An element in it is called a (conjugacy class of)  \emph{Hitchin representation}. We will denote by $\Hit_d(S)$ (resp. $\Hit(S,\sf G)$) the Hitchin component  of $\G$ in $\PSL(d,\R)$ (resp. in $\sf G$).  Any Hitchin-representation is Borel-Anosov, i.e. it is Anosov with respect to any subset of $\Pi$. It was proven in \cite{PS,PSW1} that the entropy of each simple root is constant and equal to one on each Hitchin component, when $\sf G$ is not of exceptional type. 
\end{ex}
\begin{ex}[$\Theta$-positive representations]\label{ex:positive}
A general framework encompassing all cases of connected components of character varieties of fundamental groups of surfaces only consisting of Anosov representations was proposed by Guichard-W. \cite{GWTheta}, see also \cite{GLW}. They introduce the class of $\Theta$-positive representations, which includes, apart from Hitchin components, \emph{maximal representations} in Hermitian Lie groups, as well as the conncected components of representation in the  ${\sf PO}_0(p,q)$-character variety and some components in the character varieties of the four exceptional Lie groups with restricted root system of type $\mathsf{F}_4$. While Hitchin representations are Borel-Anosov, the other representations are, in general, only Anosov with respect to a proper subset $\Theta<\Pi$, which consists of a single root in the case of maximal representations, and has $p-1$ elements in the case of ${\sf PO}_0(p,q)$-positive representations. 
It was proven in \cite{PSW2} that for  maximal and $\Theta$-positive representations in ${\sf PO}_0(p,q)$ the entropy with respect to any root in $\Theta$ is equal to one.
\end{ex}

\begin{ex}[Hyperconvex representations]\label{ex: hyperconvex}
Another important class of Anosov representations are \emph{$(1,1,p)$-hyperconvex representations} studied in \cite{PSW1}. These are representations $\rho:\G\to\PGL(d,\R)$ that are $\{\alpha_1,\alpha_p\}$-Anosov, and satisfy the additional transversality property that for all triples of pairwise distinct points $x,y,z\in\partial\G$, the sum $\xi_\rho^1(x)+\xi_\rho^1(y)+\xi_\rho^{d-p}(z)$ is direct. If $\G$ is a cocompact lattice in ${\sf PO}(1,p)$, so that $\partial\G=\mathbb S^{p-1}$, it follows from \cite{PSW1} that $\xi_\rho^1(\partial\G)$ is a C$^1$-submanifold of $\pp(\rr^d)$. Furthermore it was proven in \cite{PSW2} that for these representations, which sometimes admit non-trivial deformations, the entropy for the functional $p\omega_{\alpha_1}-\omega_{\alpha_p}$ is constant and equal to 1. Important examples of this class are the groups $\G$ dividing a properly convex domain in $\pp(\rr^d)$ studied by Benoist \cite{BenoistDivII,BenoistDivI,BenoistDivIII,BenoistDivIV}. These are $(1,1,d-1)$-hyperconvex, and were already studied by Potrie-Sambarino \cite{PS}.
\end{ex}

\begin{ex}[AdS-quasi-Fuchsian representations]\label{ex: AdSquasi-fuchsian}

Let $q\geq 2$ and $\Gamma$ be the fundamental group of a closed $q$-dimensional manifold. A representation $\rho:\Gamma\to\mathsf{PO}(2,q)$ is said to be \textit{AdS-quasi-Fuchsian} if it is faithful, discrete and preserves an acausal topological $(q-1)$-sphere on the boundary of the anti-de Sitter space $\mathbb{A}\tn{d}\mathbb{S}^{1,q}$. Recall that $\mathbb{A}\tn{d}\mathbb{S}^{1,q}$ is defined as the set of negative lines for the underlying quadratic form $\langle\cdot,\cdot\rangle_{2,q}$, and its boundary is the space $\partial\mathbb{A}\tn{d}\mathbb{S}^{1,q}$ of isotropic lines. A subset of $\partial\mathbb{A}\tn{d}\mathbb{S}^{1,q}$ is said to be \textit{acausal} if it lifts to a cone in $\rr^{2+q}\setminus\{0\}$ in which all $\langle\cdot,\cdot\rangle_{2,q}$-products of non collinear vectors are negative. The fundamental example of an AdS-quasi-Fuchsian representation is given by \textit{AdS-Fuchsian} representations, i.e. representations of the form $$\Gamma\to\mathsf{PO}(1,q)\to\mathsf{PO}(2,q),$$ \noindent where the first map is the holonomy of a closed real hyperbolic manifold, and the second arrow is the standard embedding stabilizing a negative line in $\rr^{2+q}$.

AdS-quasi-Fuchsian representations were introduced in seminal work by Mess \cite{Mess} for $q=2$, and then generalized by Barbot-M\'erigot and Barbot \cite{BM,Barbot} for $q>2$. They are $\{\alpha_1\}$-Anosov representations, where $\alpha_1$ is the simple root in $\mathsf{PO}(2,q)$ corresponding to the stabilizer of an isotropic line (see \cite{BM}). Furthermore, the space of AdS-quasi-Fuchsian representations is a union of connected components of the representation space (see \cite{Barbot}). AdS-quasi-Fuchsian representations were generalized to $\mathbb{H}^{p-1,q}$-\textit{convex-cocompact} representations by Danciger-Gu\'eritaud-Kassel \cite{DGKHpqCC}.

\end{ex}

\section{Flows associated to Anosov representations}\label{sec: anosov flows and reps}

We now recall Sambarino's Reparametrizing Theorem \cite{HyperconvexRepsExponentialGrowth,Quantitative}. This result associates to each $\rho\in\ha$ and each $\varphi\in\tn{int}((\cone)^*)$ a topological flow on a compact space, recording the data of the $\varphi$-length spectrum of $\rho$, and admitting a strong Markov coding. Through the Thermodynamical Formalism, this provides a powerful tool to study the representation $\rho$ and the space $\ha$ of $\p_\Theta$-Anosov representations.

Sambarino deals originally with Anosov representations of the fundamental group of a closed negatively curved manifold. In that case he uses the geodesic flow of the manifold (which is Anosov) as a ``reference" flow, and from $\rho$ and $\varphi$ builds a H\"older reparametrization of that flow encoding the periods $L_\rho^\varphi(\gamma)=\varphi(\lambda_\Theta(\rho(\gamma)))$. In the present framework, we are dealing with more general word hyperbolic groups. Nevertheless, his result is known to still hold: one may replace the reference geodesic flow of the manifold by the \textit{Gromov-Mineyev geodesic flow} of $\Gamma$. This is a topologically transitive H\"older continuous flow on a compact metric space $\tn{U}\Gamma$, well defined up to H\"older orbit equivalence. It was introduced by Gromov \cite{Gro} (see also Mineyev \cite{Min} for details). To define this flow space one considers a proper and cocompact action of $\Gamma$ on $\bgs\times\rr$, extending the natural action of $\Gamma$ on $\bgs$. The space $\bgs\times\rr$ equipped with this action will be denoted by $\widetilde{\tn{U}\Gamma}$, and we refer to this action as the $\Gamma$-\textit{action} on $\bgs\times\rr$. In the sequel we will consider many different actions of $\Gamma$ on $\bgs\times\rr$, depending on various choices, and this justifies this specific terminology and notation. 

The $\Gamma$-action commutes with the $\rr$-action given by $$t:(x,y,s)\mapsto(x,y,s+t).$$\noindent We let $\phi=(\phi_t:\tn{U}\Gamma\to\tn{U}\Gamma)$ be the quotient \textit{Gromov-Mineyev geodesic} flow. Central in all what follows is a result by Bridgeman-Canary-Labourie-Sambarino \cite[Sections 4 \& 5]{BCLS}, stating that in the present setting $\phi$ is metric Anosov, and one has the following (see also \cite{CLT}).

\begin{teo}[\cite{PeriodicOrbitsFlows, Bowen-Symbolic,PolMetricAnosov,BCLS}]\label{thm: gromov flow coding}
Let $\Gamma$ be a word hyperbolic group admitting an Anosov representation. Then $\phi$ admits a strong Markov coding.
\end{teo}

\subsection{The Reparametrizing Theorem}\label{subsec: reparam thm}
Provided Theorem \ref{thm: gromov flow coding}, Sambarino's Re\-pa\-ra\-me\-trizing Theorem carry on to this more general setting, as summarized in detail in \cite{SambarinoDichotomy}. More precisely, Sambarino shows that to define a H\"older re\-pa\-ra\-me\-tri\-zation of $\phi$ it suffices to consider a \textit{H\"older cocycle} over $\Gamma$ with non-negative \textit{periods} and finite \textit{entropy}. We do not give full definitions here and refer the reader to \cite[Sections 3.1 and 3.2]{SambarinoDichotomy} for details, but let us now recall how this construction works specifically for the $\varphi$-\textit{Busemann-Iwasawa cocycle} of $\rho$ (also called the $\varphi$-\textit{refraction cocycle} of $\rho$ in \cite[Definition 3.5.1]{SambarinoDichotomy}).

Let $\rho\in\ha$ and consider the pullback $\beta^{\rho}_{\Theta}: \Gamma \times \partial \Gamma \to \liea_{\Theta}$ of the Busemann-Iwasawa cocycle of $\g$ through the representation $\rho$, that is, 
$$\beta^{\rho}_{\Theta}(\gamma,x):=\sigma_{\Theta}(\rho(\gamma), \xi_\rho(x)).$$ \noindent The group $\Gamma$ acts on $\bgs\times\rr$ by $$\gamma\cdot (x,y,s):= (\gamma\cdot  x, \gamma\cdot  y, s- \varphi\circ\beta^{\rho}_{\Theta}(\gamma,y)) .$$ \noindent The space $\bgs\times\rr$ equipped with this action will be denoted by $\widetilde{\tn{U}\Gamma}^{\rho,\varphi}$ and we refer to this action as the $(\rho,\varphi)$-\textit{refraction action} (or simply the $(\rho,\varphi)$-\textit{action}). We let $\tn{U}\Gamma^{\rho,\varphi}$ be the quotient space. The $(\rho,\varphi)$-action commutes with the $\rr$-action given by $$t:(x,y,s)\mapsto(x,y,s-t).$$ \noindent We let $\phi
^{\rho,\varphi}=(\phi
^{\rho,\varphi}_t:\tn{U}\Gamma^{\rho,\varphi}\to\tn{U}\Gamma^{\rho,\varphi})$ be the quotient flow, called the $(\rho,\varphi)$-\textit{refraction} flow. As shown by Sambarino, to prove that $\phi^{\rho,\varphi}$ is H\"older orbit equivalent to $\phi$ one needs to analyse the \textit{periods} and \textit{entropy} of the $(\rho,\varphi)$-refraction cocycle. Let us now recall these notions. 

For every $\gamma\in\gh$ one has $\beta^{\rho}_{\Theta}(\gamma,\gamma_{+})=\lambda_{\Theta}(\rho \gamma)$ (c.f. \cite[Lemma 7.5]{Quantitative}). In particular, the \textit{period} $\varphi(\beta^{\rho}_{\Theta}(\gamma,\gamma_{+}))=L_\rho^\varphi(\gamma)$ of $\gamma\in\gh$ is positive. In \cite[Section 3.2]{SambarinoDichotomy}, the \textit{entropy} of $\varphi\circ\beta_\Theta^\rho$ is defined by $$\displaystyle\limsup_{t\to\infty}\frac{1}{t}\log\#\{[\gamma]\in[\gh]:\varphi(\beta^{\rho}_{\Theta}(\gamma,\gamma_{+}))\leq t\}\in[0,\infty].$$ \noindent Note that the definition of this entropy differs from the $\varphi$-entropy of $\rho$ by the fact that here we are only considering conjugacy classes of infinite order elements in $\Gamma$, while for $h_\rho^\varphi$ we also allow conjugacy classes represented by finite order elements. However, the two numbers coincide: a theorem by Bogopolskii-Gerasimov \cite{BogoGera} (see also Brady \cite{Brady}), states that there exists a positive $K_\Gamma$ such that every finite subgroup of $\Gamma$ has at most $K_\Gamma$ elements. In particular, there are only finitely many conjugacy classes of finite order elements in $\Gamma$ and therefore 
\begin{equation}\label{eq: entropy of rho and entropy of cocycle}
  h_\rho^\varphi=\displaystyle\limsup_{t\to\infty}\frac{1}{t}\log\#\{[\gamma]\in[\gh]:\varphi(\beta^{\rho}_{\Theta}(\gamma,\gamma_{+}))\leq t\}\in[0,\infty].  
\end{equation} \noindent Moreover, the $\varphi$-entropy is positive and finite. Indeed, let $\alpha\in\Theta$ and consider the function $\pp(\cone)\to\rr_{>0}$ given by $$\rr v\mapsto \frac{\varphi(v)}{\chi_\alpha(v)},$$ \noindent where $v\neq 0$ is any vector representing the line $\rr v$. Since $\pp(\cone)$ is compact, we find a constant $c>1$ so that $$c^{-1}\leq \frac{L_{\rho}^\varphi(\gamma)}{\chi_\alpha(\lambda(\rho(\gamma)))}\leq c$$ \noindent for all $\gamma\in\gh$. Applying Equation (\ref{eq: spectral radious and fund weight}) we conclude $$c^{-1}\leq \frac{L_{\rho}^\varphi(\gamma)}{\lambda_1(\Lambda_\alpha(\rho(\gamma)))}\leq c$$ \noindent for all $\gamma\in\gh$. Thanks to Proposition \ref{prop: anosov and tits}, to show $0<h_\rho^\varphi<\infty$ it suffices to show that the spectral radius entropy of a projective Anosov representation is positive and finite. On the one hand, finiteness follows by an easy geometric argument (see \cite[Lemma 5.1.2]{SambarinoDichotomy}). Positiveness though follows from dynamical reasons: the spectral radius entropy coincides with the topological entropy of the \textit{geodesic flow} of $\rho$, introduced in \cite[Section 4]{BCLS}. Since the latter flow is metric Anosov, we know by Subsection \ref{subsec:coding and metric Anosov} that its topological entropy is positive (see \cite[Theorem 5.1.3]{SambarinoDichotomy} for details).

We have checked the hypothesis on periods and entropy needed to have Sambarino's Reparametrizing Theorem.

\begin{teo}[see {\cite[Corollary 5.3.3]{SambarinoDichotomy}}]\label{thm: reparametrizing theorem}
Let $\rho\in\ha$ and $\varphi\in\tn{int}((\cone)^*)$. Then there exists an equivariant H\"older homeomorphism $$\tilde{\nu}^{\rho,\varphi}:\widetilde{\tn{U}\Gamma}\to\widetilde{\tn{U}\Gamma}^{\rho,\varphi},$$ \noindent such that for all $(x,y)\in\bgs$ there exists an increasing homeomorphism $\tilde{h}_{(x,y)}^{\rho,\varphi}:\rr\to\rr$ satisfying \begin{equation}\label{eq: rep thm}
    \tilde{\nu}^{\rho,\varphi}(x,y,s)=(x,y,\tilde{h}_{(x,y)}^{\rho,\varphi}(s))
\end{equation} \noindent for all $s\in\rr$. In particular, the $(\rho,\varphi)$-refraction action is proper and cocompact. Moreover, if we let $\nu^{\rho,\varphi}:\tn{U}\Gamma\to\tn{U}\Gamma^{\rho,\varphi}$ be the map induced by $\tilde{\nu}^{\rho,\varphi}$, then the flow $$(\nu^{\rho,\varphi})^{-1}\circ\phi^{\rho,\varphi}\circ\nu^{\rho,\varphi}$$ \noindent is a H\"older reparametrization of $\phi$.
\end{teo}

Define $\mathtt{R}_\varphi:\ha\to\phr$ by $$\mathtt{R}_\varphi(\rho):=[(\nu^{\rho,\varphi})^{-1}\circ\phi^{\rho,\varphi}\circ\nu^{\rho,\varphi}].$$ 
\noindent The map $\mathtt{R}_\varphi$ is well defined because the map $\nu^{\rho,\varphi}$, while not canonical, is well defined up to Liv\v{s}ic equivalence. We will use $\mathtt{R}_\varphi$ together with the work in Sections \ref{sec: thermodynamics} and \ref{sec: asymmetric metric and finsler norm for flows} to define and study an asymmetric metric on a suitable quotient of $\ha$:  $\mathtt{R}_\varphi$ might not be injective.
To this aim we will relate, in Section \ref{sec, generalizedThurston}, the $\varphi$-length spectrum (resp. $\varphi$-entropy) of $\rho$ with the periods of periodic orbits (resp. topological entropy) of $\phi^{\rho,\varphi}$. We conclude this section discussing the equality: $$h_\rho^\varphi=h_{\tn{top}}(\phi^{\rho,\varphi}).$$ \noindent 
When $\Gamma$ is torsion free this follows directly from \cite[Theorem 3.2.2]{SambarinoDichotomy}; we include in the next subsection a proof allowing for finite order elements in $\Gamma$.

\subsection{Strongly primitive elements, periodic orbits and entropy}\label{subsec: strongly primitive}

The \textit{axis} of an element $\gamma\in\gh$ is  $A_\gamma:=(\gamma_-,\gamma_+)\times\rr\subset\bgs\times\rr$. The element $\gamma$ acts via $(\rho,\varphi)$  on $A_\gamma$ as translation by $-\varphi(\lambda_\Theta(\rho(\gamma)))=-L_\rho^\varphi(\gamma)$. The axis $A_\gamma$ descends to a periodic orbit $a_\rho^\varphi(\gamma)=a_\rho^\varphi([\gamma])$ of $\phi
^{\rho,\varphi}$:  conjugate elements in $\Gamma$ determine the same periodic orbit. We let   $\mathcal{O}^{\rho,\varphi}$ be the set of periodic orbits of $\phi^{\rho,\varphi}$. The period $p_{\phi^{\rho,\varphi}}(a_\rho^\varphi(\gamma))$ of $a_\rho^\varphi(\gamma)$ divides the number $L_\rho^\varphi(\gamma)$, and we say that $\gamma$ is \textit{strongly primitive} (w.r.t the pair $(\rho,\varphi)$) if this period is precisely $L_\rho^\varphi(\gamma)$. Denote by $\gsp\subset\gh$ the set of strongly primitive elements. A priori, this set depends on the $(\rho,\varphi)$-action. However, we will show in Lemma \ref{lem: strongly primitive for other reparametrization} that this is not the case. 

\begin{rem}\label{rem: primitive and s primitive}
When $\Gamma$ is torsion free, strongly primitive elements coincide with \textit{primitive} elements of $\Gamma$, that is, elements that cannot be written as a power of another element. In that case, there is a one to one correspondence between periodic orbits of $\phi^{\rho,\varphi}$ and conjugacy classes of primitive elements in $\Gamma$. However, if $\Gamma$ contains finite order elements this correspondence no longer holds (see e.g. Blayac \cite[Section 3.4]{BlayacThesis} for a detailed discussion).
\end{rem}

The discussion above yields a well defined map
\begin{equation}\label{eq: projection gh to periodic orbits}
    [\gh]\to\mathcal{O}^{\rho,\varphi}\times(\zz_{>0}): [\gamma]\mapsto (a_\rho^\varphi(\gamma),n_\rho^\varphi(\gamma)),
\end{equation} \noindent where $n_\rho^\varphi(\gamma)=n_\rho^\varphi([\gamma])$ is determined by the equality $$L_{\rho}^\varphi(\gamma)=n_\rho^\varphi(\gamma)p_{\phi^{\rho,\varphi}}(a_\rho^\varphi(\gamma)).$$

To prove the equality $h_\rho^\varphi=h_{\tn{top}}(\phi^{\rho,\varphi})$ we first show the following technical lemma (recall that $K_\Gamma>0$ is the constant given by Bogopolskii-Gerasimov's Theorem \cite{BogoGera}).

\begin{lema}\label{lem: fibers of projection gh to periodic}
The fibers of the map \textnormal{(\ref{eq: projection gh to periodic orbits})} have at most $K_\Gamma$ elements. 
\end{lema}

\begin{proof}

Take $(a,n)\in \mathcal{O}^{\rho,\varphi}\times(\zz_{>0})$ and fix $\gamma_0\in\gsp$ such that $a_\rho^\varphi(\gamma_0)=a$. Let $H(\gamma_0)$ be the set of elements in $\gh$ that act trivially on $A_{\gamma_0}$. Since the $(\rho,\varphi)$-action is proper, the subgroup $H(\gamma_0)$ is finite and therefore $\# H(\gamma_0)\leq K_\Gamma$. We conclude observing that the fiber over $(a,n)$ is contained in $$\left\{[\gamma_0^n\eta]: \eta\in H(\gamma_0)\right\}.$$
\end{proof}

\begin{cor}\label{cor: varhpi entropy is entropy of traslation flow}
Let $\rho\in\ha$ and $\varphi\in\tn{int}((\cone)^*)$. Then the $\varphi$-entropy of $\rho$ coincides with the topological entropy of the refraction flow $\phi^{\rho,\varphi}$.
\end{cor}

\begin{proof}
The inequality $ h_{\tn{top}}(\phi_\rho^\varphi)\leq h_\rho^\varphi$ is easily seen. To show the reverse inequality, recall from Equation (\ref{eq: entropy of rho and entropy of cocycle}) that $$h_\rho^\varphi=\displaystyle\limsup_{t\to\infty}\frac{1}{t}\log\#\{[\gamma]\in[\gh]:L_\rho^\varphi(\gamma)\leq t\}.$$
\noindent Lemma \ref{lem: fibers of projection gh to periodic} implies then $$h_\rho^\varphi\leq\displaystyle\limsup_{t\to\infty}\frac{1}{t}\log\#\{(a,n)\in\mathcal{O}^{\rho,\varphi}\times(\zz_{>0}): np_{\phi^{\rho,\varphi}}(a)\leq t\}.$$
\noindent If we let 
$$k:=\displaystyle\min_{a\in\mathcal{O}^{\rho,\varphi}}p_{\phi^{\rho,\varphi}}(a)>0,$$
\noindent we have $$\#\{(a,n)\in\mathcal{O}^{\rho,\varphi}\times(\zz_{>0}): np_{\phi^{\rho,\varphi}}(a)\leq t\}\leq \frac{t}{k}\times\#\{a\in\mathcal{O}^{\rho,\varphi}: p_{\phi^{\rho,\varphi}}(a)\leq t\}.$$
\noindent Equation (\ref{eq: entropy}) implies the desired inequality. 
\end{proof}

\section{Thurston's metric and Finsler norm for Anosov representations}\label{sec, generalizedThurston}

Fix a functional $$\varphi \in \bigcap_{{\rho}\in\ha} \text{int} (\dcone).$$ \noindent Recall from Section \ref{sec: anosov flows and reps} that this induces a map $$\mathtt{R}_\varphi:\ha\to\phr,$$ \noindent where $\phi$ is a H\"older parametrization of the Gromov-Mineyev geodesic flow of $\Gamma$. In view of the contents of Section \ref{sec: asymmetric metric and finsler norm for flows} (and thanks to Theorem \ref{thm: gromov flow coding}), it is natural to try to ``pull back" the asymmetric metric on $\phr$ to $\ha$ under this map. This motivates the following definition.

\begin{dfn}\label{def: asymmetric distance anosov reps}
 Define $d_{\tn{Th}}^{\varphi}: \ha \times \ha  \to \rr\cup\{\infty\}$ by\footnote{When $\gamma\notin\gh$ one has $L_{\rho}^\varphi(\gamma)=0=L_{\widehat{\rho}}^\varphi(\gamma)$. In the above definition it is understood that in that case we set $$\frac{ L_{\widehat{\rho}}^\varphi(\gamma)}{ L_{\rho}^\varphi(\gamma)}=0.$$} $$d_{\tn{Th}}^\varphi(\rho,\widehat{\rho}):=\log\left(\displaystyle\sup_{[\gamma]\in[\Gamma]}\frac{h_{\widehat{\rho}}^\varphi}{h_{\rho}^\varphi}\frac{ L_{\widehat{\rho}}^\varphi(\gamma) }{L_{\rho}^\varphi(\gamma)}\right).$$
\end{dfn}

The main theorem of this section is the following.

\begin{teo}\label{thm: dth for anosov}
The function $d_{\tn{Th}}^\varphi(\cdot,\cdot)$ is real valued, non-negative, and satisfies the triangle inequality. Furthermore $$d_{\tn{Th}}^\varphi(\rho,\widehat{\rho})=0 \Leftrightarrow h_{\rho}^\varphi L_\rho^\varphi=h_{\widehat{\rho}}^\varphi L_{\widehat{\rho}}^\varphi.$$
\end{teo}

We deduce Theorem \ref{thm: dth for anosov} from Theorem \ref{teo: asymmetric distance flows}: in Corollary \ref{cor: dist for reps coincides with distance for flows} we show that for all $\rho,\widehat{\rho}\in\ha$,  $$d_{\tn{Th}}^\varphi(\rho,\widehat{\rho})=d_{\tn{Th}}(\mathtt{R}_\varphi(\rho),\mathtt{R}_\varphi(\widehat{\rho}))$$ and in Corollary \ref{cor: kernel or rvarphi} we prove that  $\mathtt{R}_\varphi(\rho)=\mathtt{R}_\varphi(\widehat{\rho})$ if and only if $h_{\rho}^\varphi L_\rho^\varphi=h_{\widehat{\rho}}^\varphi L_{\widehat{\rho}}^\varphi$.  Both Corollaries \ref{cor: dist for reps coincides with distance for flows} and \ref{cor: kernel or rvarphi} are straightforward when $\Gamma$ is torsion free (see Remark \ref{rem: primitive and s primitive}). We explain the details in Subsection \ref{subsec: proof of dth for representations}  allowing for finite order elements in $\Gamma$. In Subsection \ref{subsec: renorm length rigidity} we discuss general conditions that guarantee renormalized length spectrum rigidity. As a consequence, we will have an asymmetric metric defined in interesting subsets of $\ha$ (under some assumptions on $\g$). More examples will be discussed in Sections \ref{sec: hitchin} and \ref{s.other}. In Subsection \ref{subsec: finsler for anosov reps} we use the map $\mathtt{R}_\varphi$ to pull back the Finsler norm of $\phr$ to $\ha$.

\subsection{Proof of Theorem \ref{thm: dth for anosov}}\label{subsec: proof of dth for representations}

Let $\rho\in\ha$. Recall from Subsection \ref{subsec: strongly primitive} that $\gamma\in\gh$ is strongly primitive (w.r.t $(\rho,\varphi)$) if the $(\rho,\varphi)$-action of $\gamma$ on the axis $A_\gamma$ is a translation by the period of the corresponding periodic orbit of $\phi^{\rho,\varphi}$. The following technical lemma implies in particular that this notion is independent of $\rho$ (recall the notation introduced in Equation (\ref{eq: projection gh to periodic orbits})). We note that this holds in the more general setting of H\"older reparametrizations of the Gromov geodesic flow (see also Remark \ref{rem: sp for geodesic flow} below).

\begin{lema}\label{lem: strongly primitive for other reparametrization}

Let $\rho$ and $\widehat{\rho}$ in $\ha$, then for every $\gamma\in\gh$ one has $$n_\rho^\varphi(\gamma)=n_{\widehat{\rho}}^\varphi(\gamma).$$ \noindent In particular, $\gamma$ is strongly primitive for the $(\rho,\varphi)$-action if and only if it is strongly primitive for the $(\widehat{\rho},\varphi)$-action.

\end{lema}

\begin{proof}

To ease notations we let $n:=n_\rho^\varphi(\gamma)$ and $\widehat{n}:=n_{\widehat{\rho}}^\varphi(\gamma)$. Suppose by contradiction that $n\neq \widehat{n}$, say $n<\widehat{n}$.

Let $a=a_\rho^\varphi(\gamma)$ (resp. $\widehat{a}=a_{\widehat{\rho}}^\varphi(\gamma)$) be the periodic orbit of $\phi^{\rho,\varphi}$ (resp. $\phi^{\widehat{\rho},\varphi}$) associated to $[\gamma]$. Fix a strongly primitive $\gamma_0$ (resp. $\widehat{\gamma}_0$) representing $a$ (resp. $\widehat{a}$) for the $(\rho,\varphi)$-action (resp. $(\widehat{\rho},\varphi)$-action). By definition of $n$ and $\widehat{n}$ we have \begin{equation}\label{eq: in lemma n equals widehat n}
    L_\rho^\varphi(\gamma)=nL_\rho^\varphi(\gamma_0) \tn{ and } L_{\widehat{\rho}}^\varphi(\gamma)=\widehat{n}L_{\widehat{\rho}}^\varphi(\widehat{\gamma}_0).
\end{equation} \noindent We may assume furthermore that $(\gamma_0)_\pm=(\widehat{\gamma}_0)_\pm$.

On the other hand, by Theorem \ref{thm: reparametrizing theorem} there exists an equivariant H\"older homeomorphism $$\nu:\widetilde{\tn{U}\Gamma}^{\rho,\varphi}\to\widetilde{\tn{U}\Gamma}^{\widehat{\rho},\varphi},$$ \noindent such that for all $(x,y)\in\bgs$ there exists an increasing homeomorphism $h_{(x,y)}:\rr\to\rr$ satisfying $$\nu(x,y,s)=(x,y,h_{(x,y)}(s)).$$ \noindent Hence, for all $\eta\in\Gamma$ and all $(x,y,s)\in\widetilde{\tn{U}\Gamma}^{\rho,\varphi}$ one has
$$h_{(\eta\cdot x,\eta\cdot y)}(s-\varphi\circ\beta_\Theta^\rho(\eta,y))=h_{( x, y)}(s)-\varphi\circ\beta_\Theta^{\widehat{\rho}}(\eta,y).$$
\noindent In particular, Equation (\ref{eq: in lemma n equals widehat n}) gives
$$h_{((\gamma_0)_-,(\gamma_0)_+)}(s-nL_{\rho}^\varphi(\gamma_0))=h_{((\gamma_0)_-,(\gamma_0)_+)}(s-L_{\rho}^\varphi(\gamma))=h_{((\gamma_0)_-,(\gamma_0)_+)}(s)-L_{\widehat{\rho}}^\varphi(\gamma),$$ \noindent and therefore $$h_{((\gamma_0)_-,(\gamma_0)_+)}(s-nL_{\rho}^\varphi(\gamma_0))=h_{((\gamma_0)_-,(\gamma_0)_+)}(s)-\widehat{n}L_{\widehat{\rho}}^\varphi(\widehat{\gamma}_0).$$ \noindent Hence $$h_{((\gamma_0)_-,(\gamma_0)_+)}(s-nL_{\rho}^\varphi(\gamma_0))=h_{((\gamma_0)_-,(\gamma_0)_+)}(s)-L_{\widehat{\rho}}^\varphi(\widehat{\gamma}_0^{\widehat{n}})=h_{((\gamma_0)_-,(\gamma_0)_+)}(s-L_{\rho}^\varphi(\widehat{\gamma}_0^{\widehat{n}})).$$ \noindent We then conclude $$h_{((\gamma_0)_-,(\gamma_0)_+)}(s-nL_{\rho}^\varphi(\gamma_0))=h_{((\gamma_0)_-,(\gamma_0)_+)}(s-\widehat{n}L_{\rho}^\varphi(\widehat{\gamma}_0)).$$ \noindent This implies $$nL_{\rho}^\varphi(\gamma_0)=\widehat{n}L_{\rho}^\varphi(\widehat{\gamma}_0)>nL_{\rho}^\varphi(\widehat{\gamma}_0).$$ \noindent This is a contradiction because $\gamma_0$ was assumed to be strongly primitive for the $(\rho,\varphi)$-action.
\end{proof}

\begin{cor}\label{cor: dist for reps coincides with distance for flows}
For every $\rho$ and $\widehat{\rho}$ in $\ha$ one has $$d_{\tn{Th}}^\varphi(\rho,\widehat{\rho})=d_{\tn{Th}}(\mathtt{R}_\varphi(\rho),\mathtt{R}_\varphi(\widehat{\rho})).$$
\end{cor}

\begin{proof}
By Corollary \ref{cor: varhpi entropy is entropy of traslation flow} we have $$d_{\tn{Th}}^\varphi(\rho,\widehat{\rho})=\log\left(\displaystyle\sup_{[\gamma]\in[\Gamma]}\frac{h_{\tn{top}}(\phi^{\widehat{\rho},\varphi})}{h_{\tn{top}}(\phi^{\rho,\varphi})}\frac{ L_{\widehat{\rho}}^\varphi(\gamma) }{L_{\rho}^\varphi(\gamma)}\right).$$ \noindent Equation (\ref{eq: projection gh to periodic orbits}) gives then $$d_{\tn{Th}}^\varphi(\rho,\widehat{\rho})=\log\left(\displaystyle\sup_{[\gamma]\in[\Gamma]}\frac{h_{\tn{top}}(\phi^{\widehat{\rho},\varphi})}{h_{\tn{top}}(\phi^{\rho,\varphi})}\frac{ n_{\widehat{\rho}}^\varphi(\gamma)}{n_{\rho}^\varphi(\gamma)}\frac{p_{\phi^{\widehat{\rho},\varphi}}(a_{\widehat{\rho}}^\varphi(\gamma))}{p_{\phi^{\rho,\varphi}}(a_{\rho}^\varphi(\gamma))}\right).$$ \noindent By Lemma \ref{lem: strongly primitive for other reparametrization} we have $$d_{\tn{Th}}^\varphi(\rho,\widehat{\rho})=\log\left(\displaystyle\sup_{[\gamma]\in[\Gamma]}\frac{h_{\tn{top}}(\phi^{\widehat{\rho},\varphi})}{h_{\tn{top}}(\phi^{\rho,\varphi})}\frac{ p_{\phi^{\widehat{\rho},\varphi}}(a_{\widehat{\rho}}^\varphi(\gamma))}{p_{\phi^{\rho,\varphi}}(a_{\rho}^\varphi(\gamma))}\right).$$ \noindent This finishes the proof.

\end{proof}

\begin{rem}\label{rem: AvoidDomination}
There are geometric settings in which the renormalization by entropy in the definition of the asymmetric metric is essential (see also Section \ref{subsec: inter and renormalized intersection}). For instance, Tholozan \cite[Theorem B]{ThoEntropy} shows that there exist pairs $\rho$ and $j$ in $\Hit_3(S)$ for which there is a $c>1$ so that \begin{equation}\label{eq: domination}
  L^\mathrm{H}_{\rho}(\gamma) \geq c L^\mathrm{H}_{j}(\gamma)  
\end{equation} \noindent for all $\gamma\in \pi_1(S)$ (recall the notation introduced in Example \ref{list, Lengths}). Hence $$\log\left(\displaystyle\sup_{[\gamma]\in[\pi_1(S)]}\frac{L^\mathrm{H}_j(\gamma)}{L_\rho^\mathrm{H}(\gamma)}\right)\leq\log\left(\frac{1}{c}\right)<0.$$ 

On the other hand some length functions on some spaces of Anosov representations have constant entropies (c.f. Subsection \ref{subsec: examples anosov reps}). In these situations, renormalizing by entropy is not needed.
\end{rem}

We now compute the set of points which are identified under the map $\mathtt{R}_\varphi$, finishing the proof of Theorem \ref{thm: dth for anosov}.

\begin{cor}\label{cor: kernel or rvarphi}
Let $\rho$ and $\widehat{\rho}$ be two points in $\ha$. Then $$\mathtt{R}_\varphi(\rho)=\mathtt{R}_\varphi(\widehat{\rho})\Leftrightarrow h_{\rho}^\varphi L_\rho^\varphi=h_{\widehat{\rho}}^\varphi L_{\widehat{\rho}}^\varphi.$$
\end{cor}

\begin{proof}
 By definition of $\phr$ and Corollary \ref{cor: varhpi entropy is entropy of traslation flow} we have $$\mathtt{R}_\varphi(\rho)=\mathtt{R}_\varphi(\widehat{\rho})\Leftrightarrow  h_\rho^\varphi p_{\phi^{\rho,\varphi}}(a_\rho^\varphi(\gamma))=h_{\widehat{\rho}}^\varphi p_{\phi^{\widehat{\rho},\varphi}}(a_{\widehat{\rho}}^\varphi(\gamma))$$ \noindent for all $\gamma\in\gh$. Thanks to Lemma \ref{lem: strongly primitive for other reparametrization} this is equivalent to $$h_\rho^\varphi  n_{\rho}^\varphi(\gamma) p_{\phi^{\rho,\varphi}}(a_\rho^\varphi(\gamma))=h_{\widehat{\rho}}^\varphi n_{\widehat{\rho}}^\varphi(\gamma) p_{\phi^{\widehat{\rho},\varphi}}(a_{\widehat{\rho}}^\varphi(\gamma))$$ \noindent for all $\gamma\in\gh$. Since for all $\gamma\in\gh$ we have $$ n_{\rho}^\varphi(\gamma) p_{\phi^{\rho,\varphi}}(a_\rho^\varphi(\gamma))=L_\rho^\varphi(\gamma) \tn{ and } n_{\widehat{\rho}}^\varphi(\gamma) p_{\phi^{\widehat{\rho},\varphi}}(a_{\widehat{\rho}}^\varphi(\gamma))=L_{\widehat{\rho}}^\varphi(\gamma),$$ \noindent the proof is finished.

\end{proof}

To finish this subsection we record the following technical remark for future use.

\begin{rem}\label{rem: sp for geodesic flow}
One may define the notion of strongly primitive elements for the action $\Gamma \curvearrowright \widetilde{\tn{U}\Gamma}$, in a way analogue to the definition for the action $\Gamma \curvearrowright \widetilde{\tn{U}\Gamma}^{\rho,\varphi}$. As in Lemma \ref{lem: strongly primitive for other reparametrization}, one shows that $\gamma$ is strongly primitive for $\Gamma \curvearrowright \widetilde{\tn{U}\Gamma}$ if and only if it is strongly primitive for the $(\rho,\varphi)$-action, for some (any) $\rho\in\ha$. 

On the other hand, if we let $\mathcal{O}$ be the set of periodic orbits of $\phi$, we may take for each $a\in\mathcal{O}$ a strongly primitive representative $\gamma_a\in\gsp$. We see that $$a\mapsto [A_{\gamma_a}]$$ \noindent defines a one to one correspondence between $\mathcal{O}$ and $\mathcal{O}^{\rho,\varphi}$ for all $\rho\in\ha$, where $[A_{\gamma_a}]$ is the image of the axis $A_{\gamma_a}$ under the quotient map $\widetilde{\tn{U}\Gamma}^{\rho,\varphi}\to\tn{U}\Gamma^{\rho,\varphi} $. A set $\{\gamma_a\}_{a\in\mathcal{O}}$ of strongly primitive elements representing each periodic orbit will be fixed from now on.
\end{rem}

\subsection{Renormalized length spectrum rigidity}\label{subsec: renorm length rigidity}

Recall that $\g$ is a connected semisimple real algebraic group of non-compact type. In this subsection we discuss necessary conditions that two $\Theta$-Anosov representations with the same renormalized length spectra must satisfy.

For a Lie group $\g_1$ we denote by $(\g_1)_0$ the connected component, in the Hausdorff topology, containing the identity. If $\sigma:\g_1\to\g_2$ is a Lie group isomorphism, we denote, with a slight abuse of notation, by $\sigma:\mathfrak a_{\g_1}^+\to \mathfrak a_{\g_2}^+$ the induced linear isomorphism between Weyl chambers. Furthermore, if $\g_1<\g$ is a Lie group inclusion, we denote by $\pi_{\g_1}:\mathfrak a_{\g_1}^+\to \mathfrak a_{\g}^+$ the induced piecewise linear map.

We will need the following fairly general classical rigidity result, which is an application of Benoist \cite[Theorem 1]{Benoist_AsymtoticLinearGroups}. See for instance \cite[Corollary 11.6]{BCLS}, Burger \cite{BurgerManhattan} and Dal'bo-Kim \cite{Criterion_Zariki}.

\begin{teo}\label{thm:rigidity}
Let $\rho$ and $\widehat{\rho}$ be two $\Theta$-Anosov representations into $\g$. Denote by $\g_{\rho}$ (resp. $\g_{\widehat{\rho}}$) the Zariski closure of $\rho(\Gamma)$ (resp. $\widehat{\rho}(\Gamma)$). Assume that $\g_{\rho}$ and $\g_{\widehat{\rho}}$ are simple, real algebraic and center-free. Assume furthermore $\rho(\Gamma)\subset (\g_\rho)_0$ and $\widehat{\rho}(\Gamma)\subset (\g_{\widehat{\rho}})_0$. Then if the equality $h^\varphi_{\rho}L^\varphi_{\rho}=h^\varphi_{{\widehat{\rho}}}L^\varphi_{\widehat{\rho}}$ holds, there exists an isomorphism $\sigma:(\g_{\rho})_0\to(\g_{\widehat{\rho}})_0$ such that $\sigma\circ \rho=\widehat{\rho}$. It furthermore holds $\varphi\circ \pi_{\g_{\widehat{\rho}}}\circ \sigma=\varphi\circ \pi_{\g_\rho}$.
\end{teo}

Denote by $\haz\subset\ha$ the subset consisting of Zariski dense representations. 

\begin{cor}\label{cor: distance in zariski dense components}
Assume that $\g$ is simple, center-free, and for every non-inner automorphism $\sigma$ of $\g$ one has $\varphi\circ \sigma \neq \varphi$. Then $d_{\tn{Th}}^\varphi(\cdot,\cdot)$ defines a (possibly asymmetric) metric on $\haz$.
\end{cor}

\begin{rem}
The group $\g$ needs to be center-free in Theorem \ref{thm:rigidity} and Corollary \ref{cor: distance in zariski dense components}:  the Jordan and Cartan projections of $\g$ factor through the adjoint form of $\g$, thus any two representations differing by a central character will have the same renormalized length spectrum, and thus distance zero. 
\end{rem}

\subsection{Finsler norm for Anosov representations}\label{subsec: finsler for anosov reps}

Bridgeman-Canary-Labourie-Sambarino \cite{BCLS,BCLSSIMPLEROOTS} used the map $\mathtt{R}_\varphi$ to pull-back the pressure norm on $\phru$ to produce a pressure metric on $\ha$ (for some choices of $\varphi$). We now imitate this procedure  working with the Finsler norm defined in Subsection \ref{subsection, FinslerNorm}.

A family of representations $\{\rho_z:\Gamma\to\g\}_{z\in D}$ parametrized by a real analytic disk $D$ is  \textit{real analytic} if for all $\gamma\in\Gamma$ the map $z\mapsto\rho_z(\gamma)$ is real analytic. We fix a real analytic neighbourhood of $\rho\in\ha$ and a real analytic family $\{\rho_z\}_{z\in D}\subset \ha$, parametrized by some real analytic disk $D$ around $0$, so that $\rho_0=\rho$ and $\cup_{z\in D}\rho_z$ coincides with this neighbourhood. By abuse of notation we will sometimes identify the neighbourhood with $D$ itself.

\begin{dfn}\label{dfn: finsler reps}
Given a tangent vector $v\in T_\rho\ha$ we set $$\Vert v\Vert_{\tn{Th}}^\varphi:= \displaystyle\sup_{[\gamma]\in[\gh]} \frac{\od_{\rho}(h_{\cdot}^\varphi)(v)L_\rho^\varphi(\gamma)+h_\rho^\varphi \od_\rho(L_\cdot^\varphi(\gamma))(v)}{h_\rho^\varphi L_\rho^\varphi(\gamma)}, $$ \noindent where $\od_{\rho}(h_{\cdot}^\varphi)$ (resp. $\od_{\rho}(L_{\cdot}^\varphi(\gamma))$) is the derivative of $\widehat{\rho}\mapsto h_{\widehat{\rho}}^\varphi$ (resp. $\widehat{\rho}\mapsto L_{\widehat{\rho}}^\varphi(\gamma)$) at $\rho$. In particular, if $\widehat{\rho}\mapsto h_{\widehat{\rho}}^\varphi$ is constant one has
\begin{equation}\label{eq: finsler reps constant entropy}  \Vert v\Vert_{\tn{Th}}^\varphi= \displaystyle\sup_{[\gamma]\in[\gh]} \frac{ \od_\rho(L_\cdot^\varphi(\gamma))(v)}{ L_\rho^\varphi(\gamma)}.\end{equation}

\end{dfn}

\begin{rem}\label{rem: def finsler anosov}
\begin{enumerate}
    \item Recall that by \cite[Section 8]{BCLS}, entropy varies in an analytic way over $\ha$. In particular, $h_\cdot^\varphi$ is differentiable.
    \item Equation (\ref{eq: finsler reps constant entropy}) generalizes Thurston's Finsler norm on Teichm\"uller space \cite[p.20]{ThurstonStretch}.
\end{enumerate}
\end{rem}

We want conditions  guaranteeing that $\Vert\cdot\Vert_{\tn{Th}}^\varphi$ defines a Finsler norm on $T_\rho\ha$; a priori it is not even clear that $\Vert\cdot\Vert_{\tn{Th}}^\varphi$ is real valued and non-negative. To  link  $\Vert\cdot\Vert_{\tn{Th}}^\varphi$ and the Finsler norm of Subsection \ref{subsection, FinslerNorm}, we need the following proposition. We  fixed a set of strongly primitive elements $\{\gamma_a\}$ representing each periodic orbit $a\in\mathcal{O}$ in Remark \ref{rem: sp for geodesic flow}.

\begin{prop}[{\cite[Proposition 6.2]{BCLS}, {\cite[Proposition 6.1]{BCLSSIMPLEROOTS}}}]\label{prop: Rvarhpi analytic}
Let $\{\rho_z\}_{z\in D}$ be a real analytic family of $\Theta$-Anosov representations. Then up to  restricting $D$ to a smaller disk around $0$, there exists $\upsilon>0$ and a real analytic family $\{\widetilde{g}_z:\tn{U}\Gamma\to \rr_{>0}\}_{z\in D}\subset\mathcal{H}^\upsilon(\tn{U}\Gamma)$ so that for all $z\in D$, all $a\in\mathcal{O}$ and all $x\in a$ one has $$\displaystyle\int_0^{p_\phi(a)}\widetilde{g}_z(\phi_s(x))\od s=L_{\rho_z}^\varphi(\gamma_a).$$ \noindent In particular, the map $D\to \phru$ given by $z\mapsto \mathtt{R}_\varphi(\rho_z)=\left[\phi^{\widetilde{g}_z}\right]$ is real analytic.
\end{prop}

\begin{proof}
The argument follows \cite[Proposition 6.1]{BCLSSIMPLEROOTS}. Since $\{\omega_\alpha\}_{\alpha\in\Theta}$ span $\liea_\Theta^*$, there exist real numbers $a_\alpha$ so that $\varphi=\sum_{\alpha\in\Theta}a_\alpha\omega_\alpha$.  \cite[Proposition 6.2]{BCLS} gives the result for projective Anosov representations and the spectral radius length function, thus the proof of \cite[Proposition 6.1]{BCLSSIMPLEROOTS} applies  (c.f. Proposition \ref{prop: anosov and tits} and Equation (\ref{eq: spectral radious and fund weight})).
\end{proof}

Fix a real analytic family $\{\widetilde{g}_z\}$ as in Proposition \ref{prop: Rvarhpi analytic}. By \cite[Proposition 3.12]{BCLS} the function $z\mapsto h_{\phi^{\widetilde{g}_z}}$ is real analytic. By Corollary \ref{cor: varhpi entropy is entropy of traslation flow} we get that $z\mapsto h_{\rho_z}^\varphi$ is real analytic, as claimed in Remark \ref{rem: def finsler anosov}.

Proposition \ref{prop: Rvarhpi analytic}  bridges between $\Vert\cdot\Vert_{\tn{Th}}^\varphi$ and the Finsler norm on $\phru$, as we now explain. First, observe that in Definition \ref{dfn: finsler reps} it suffices to consider only strongly primitive elements when taking the sup, that is: $$\Vert v\Vert_{\tn{Th}}^\varphi= \displaystyle\sup_{[\gamma]\in[\gsp]} \frac{\od_{\rho}(h_{\cdot}^\varphi)(v)L_\rho^\varphi(\gamma)+h_\rho^\varphi \od_\rho(L_\cdot^\varphi(\gamma))(v)}{h_\rho^\varphi L_\rho^\varphi(\gamma)}.$$ 
\noindent Indeed the function $\widehat{\rho}\mapsto n_{\widehat{\rho}}^\varphi(\gamma)$ is constant for all $\gamma\in\gh$ (Lemma \ref{lem: strongly primitive for other reparametrization}), and Remark \ref{rem: sp for geodesic flow} gives \begin{equation}\label{eq: finsler anosov rep with strongly primitive}
    \Vert v\Vert_{\tn{Th}}^\varphi= \displaystyle\sup_{a\in\mathcal{O}} \frac{\od_{\rho}(h_{\cdot}^\varphi)(v)L_\rho^\varphi(\gamma_a)+h_\rho^\varphi \od_\rho(L_\cdot^\varphi(\gamma_a))(v)}{h_\rho^\varphi L_\rho^\varphi(\gamma_a)}.
\end{equation} \noindent Recalling the notations from Subsection \ref{subsection, FinslerNorm} we have the following.

\begin{lema}\label{lem: finsler reps and flows}
Let $\{\rho_z\}_{z\in D}\subset \ha$ be a real analytic family parametrizing an open neighbourhood around $\rho=\rho_0$. Fix an analytic path $z:(-1,1)\to D$ so that $z(0)=0$ and set $\rho_s:=\rho_{z(s)}$ and $ v:=\left.\frac{\od}{\od s}\right\vert_{s=0}\rho_s$. Let also $h_s:=h_{\rho_{s}}^\varphi$ and $g_s:=h_s\widetilde{g}_{z(s)}$. Then $$\Vert v\Vert_{\tn{Th}}^\varphi=\Vert [\dot{g}_0]\Vert_{\tn{Th}}.$$ \noindent 
\end{lema}

In the above statement, by construction, the Liv\v{s}ic cohomology class $[\dot{g}_0]=[\dot{g}_0]_\phi$ belongs to the tangent space $T_{[\phi^{g_0}]}\phru$.

\begin{proof}[Proof of Lemma \ref{lem: finsler reps and flows}]
Combining Equations (\ref{eq: finsler anosov rep with strongly primitive}) and (\ref{eq: integral of reparametrizing over delta in periodic orbit}), and Proposition \ref{prop: Rvarhpi analytic} we have $$\Vert v\Vert_{\tn{Th}}^\varphi=\displaystyle\sup_{a\in\mathcal{O}} \left.\frac{\od}{\od s}\right\vert_{s=0} \frac{h_sL_{\rho_s}^\varphi(\gamma_a) }{h_\rho^\varphi L_\rho^\varphi(\gamma_a)}=\displaystyle\sup_{a\in\mathcal{O}} \left.\frac{\od}{\od s}\right\vert_{s=0} \frac{h_s}{h_0}\frac{  \int \widetilde{g_s}\od \delta_\phi(a) }{ \int \widetilde{g_0}\od \delta_\phi(a)}.$$ \noindent Hence $$\Vert v\Vert_{\tn{Th}}^\varphi=\displaystyle\sup_{a\in\mathcal{O}} \left.\frac{\od}{\od s}\right\vert_{s=0} \frac{ \int g_s\od \delta_\phi(a) }{ \int g_0\od \delta_\phi(a)}=\displaystyle\sup_{a\in\mathcal{O}}  \frac{ \int \dot{g}_0\od \delta_\phi(a) }{\int g_0\od \delta_\phi(a)}.$$ \noindent By Theorem \ref{teo: periodic orbits dense} we get $$\Vert v\Vert_{\tn{Th}}^\varphi=\displaystyle\sup_{m\in\pphi}  \frac{ \int \dot{g}_0\od m }{\int g_0\od m}.$$ \noindent This finishes the proof.

\end{proof}

From Propositions \ref{prop: FinslerNorm} and \ref{prop: Rvarhpi analytic}, and Corollary \ref{cor: dist for reps coincides with distance for flows} we obtain the following.

\begin{cor}\label{cor: link finsler and asymm for reps}
Keep the notations from Lemma \ref{lem: finsler reps and flows}. Then $s\mapsto d_{\tn{Th}}^\varphi(\rho,\rho_s)$ is differentiable at $s=0$ and $$\Vert v\Vert_{\tn{Th}}^\varphi=\left.\frac{\od}{\od s}\right\vert_{s=0}d_{\tn{Th}}^\varphi(\rho,\rho_s).$$
\end{cor}

We now turn to the study of conditions guaranteeing that $\Vert\cdot\Vert_{\tn{Th}}^\varphi$ defines a Finsler norm.

\begin{cor}\label{cor: finsler for reps}
Let $\rho\in\ha$ be a point admitting an analytic neighbourhood in $\ha$. Then function $\Vert\cdot\Vert_{\tn{Th}}^\varphi:T_\rho\ha\to\rr\cup\{\pm\infty\}$ is real valued and non-negative. Furthermore, it is $(\rr_{>0})$-homogeneous, satisfies the triangle inequality and $\Vert v\Vert_{\tn{Th}}^\varphi=0$ if and only if \begin{equation}\label{eq: cor finsler resp}
    d_\rho (L_\cdot^\varphi(\gamma))(v)=-\frac{\od_{\rho}(h_{\cdot}^\varphi)(v)}{h_\rho^\varphi}L_\rho^\varphi(\gamma)
\end{equation} \noindent for all $\gamma\in\gh$. In particular, if the function $\rho\mapsto h_{\rho}^\varphi$ is constant, then $$\Vert v\Vert_{\tn{Th}}^\varphi=0\Leftrightarrow d_\rho (L_\cdot^\varphi(\gamma))(v)=0$$ \noindent for all $\gamma\in\gh$.
\end{cor}

\begin{proof}
By Lemma \ref{lem: norm is non degenerate} and Lemma \ref{lem: finsler reps and flows}, the function $\Vert\cdot\Vert_{\tn{Th}}^\varphi$ is real valued, non-negative, $(\rr_{>0})$-homogeneous and satisfies the triangle inequality. Furthermore, keeping the notation from Lemma \ref{lem: finsler reps and flows}, if $\Vert v\Vert_{\tn{Th}}^\varphi=0$ then $ \dot{g}_0\sim_\phi 0$ and this condition is equivalent to $$0=\displaystyle\int_0^{p_\phi(a)} \dot{g}_0(\phi_t(x))\od t=\left.\frac{\od}{\od s}\right\vert_{s=0}\displaystyle\int_0^{p_\phi(a)} g_s(\phi_t(x))\od t$$ \noindent for all $a\in\mathcal{O}$ and $x\in a$. Hence $$0=\left.\frac{\od}{\od s}\right\vert_{s=0} h_s\displaystyle\int_0^{p_\phi(a)} \widetilde{g}_s(\phi_t(x))\od t=\left.\frac{\od}{\od s}\right\vert_{s=0} h_s L_{\rho_s}^\varphi(\gamma_a).$$ \noindent Thus $$d_\rho (L_\cdot^\varphi(\gamma_a))(v)=-\frac{\od_{\rho}(h_{\cdot}^\varphi)(v)}{h_\rho^\varphi}L_\rho^\varphi(\gamma_a)$$ \noindent for all $a\in\mathcal{O}$. Now by Lemma \ref{lem: strongly primitive for other reparametrization} for every $\gamma\in\gh$ there is some $n\geq 1$ and $a\in\mathcal{O}$ so that $L_{\rho}^\varphi(\gamma)=nL_{\rho}^\varphi(\gamma_a)$ for all $\rho\in\ha$. This finishes the proof.
\end{proof}

In view of Corollary \ref{cor: finsler for reps}, to show that $\Vert\cdot\Vert_{\tn{Th}}^\varphi$ is a Finsler norm, one needs to guarantee that condition (\ref{eq: cor finsler resp}) implies $v=0$. These type of questions have been addressed by Bridgeman-Canary-Labourie-Sambarino \cite{BCLS,BCLSSIMPLEROOTS} in some situations. Rather than discussing these results here, we will recall them in the next sections, when needed.

\section{Hitchin representations}\label{sec: hitchin}

In this section we focus on Hitchin representations. The Zariski closures of $\mathsf{PSL}(d,\rr)$-Hitchin representations have been classified by Guichard. Hence, the results of the previous section apply nicely in this setting giving global rigidity results and leading to asymmetric distances in the whole component. This is explained in detail in Subsection \ref{subsec: rigidity hitchin}, where we also treat the case of $\mathsf{PSO}_0(p,p)$, the remaining classical case not covered by Guichard's classification, using recent results by Sambarino \cite{sambarino2020infinitesimal}. In Subsection \ref{subsec: finsler for hitchin} we discuss Finsler norms associated to some special length functionals in the $\mathsf{PSL}(d,\rr)$-Hitchin component, showing that they are non degenerate (this will be a consequence of Corollary \ref{cor: finsler for reps} and results in \cite{BCLS,BCLSSIMPLEROOTS}).

Throughout this section we let $S$ be a closed oriented surface of genus $g\geq2$, and denote by $\Gamma=\pi_1(S)$ its fundamental group. We also let $\g$ be an adjoint, connected, simple real-split Lie group. Apart from exceptional cases, $\g$ is one of the following $$\mathsf{PSL}(d,\rr),\mathsf{PSp}(2r,\rr), \mathsf{SO}_0(p,p+1), \tn{ or } \mathsf{PSO}_0(q,q),$$ \noindent for $q>2$.  Hitchin representations are  $\Pi$-Anosov (c.f. Example \ref{ex: hitchin and positive}). We denote by $\tn{Hit}(S,\g)$ the Hitchin component into $\g$, when $\g=\mathsf{PSL}(d,\rr)$ we also use the special notation $\Hit_d(S)$.

\subsection{Length spectrum rigidity}\label{subsec: rigidity hitchin}

For $\rho\in\Hit(S,\g)$ denote $\mathscr{L}_\rho^*:=(\mathscr{L}^{\Pi}_\rho)^*$ and consider $\varphi \in \bigcap_{{\rho}\in \Hit(S,\g)} \text{int} (\mathscr{L}_\rho^*) \subset \liea_{\Pi}^*=\liea^*$. 

The main goal of this section is to prove the following.

\begin{teo}\label{Thm:rigitity length hitchin}
Let $\g$ be an adjoint, simple, real-split Lie group of classical type. In the case $\g=\mathsf{PSO}_0(p,p)$, assume furthermore $p\neq 4$.
Let $\varphi \in \bigcap_{{\rho}\in \Hit(S,\g)} \tn{int} (\mathscr{L}_\rho^*)$ be so that $\varphi\circ\sigma\neq\varphi$ for every non inner automorphism $\sigma$ of $\g$. If $\rho, \widehat{\rho} \in \Hit(S,\g)$ satisfy $h^\varphi_{\rho}L^{\varphi}_{\rho}=h^\varphi_{\widehat{\rho}}L^{\varphi}_{\widehat{\rho}}$, then $\rho=\widehat{\rho}$.
\end{teo}

Before going into the proof of Theorem \ref{Thm:rigitity length hitchin} we make few remarks and establish the main corollaries of interest.

\begin{rem}
\begin{itemize}
    \item When $\g=\mathsf{PSL}(d,\rr)$, Bridgeman-Canary-Labourie-Sambarino \cite[Corollary 11.8]{BCLS} proved Theorem \ref{Thm:rigitity length hitchin} for the spectral radius length function $\varphi=\lambda_1$. The proof of Theorem \ref{Thm:rigitity length hitchin} follows the same approach.
    \item We aim to define a simple root asymmetric metric on $\Hit(S,\g)$ (Corollary \ref{cor: asymm for hitchin roots} below). As every simple root of $\mathsf{PSO}_0(4,4)$ is fixed by a non inner automorphism, the function $$d_{\tn{Th}}^{\alpha}:\Hit(S,\mathsf{PSO}_0(4,4))\times\Hit(S,\mathsf{PSO}_0(4,4))\to\rr$$ \noindent does not separate points for any simple root $\alpha$. This is the main reason why we exclude the case $\g=\mathsf{PSO}_0(4,4)$ in the statement of Theorem \ref{Thm:rigitity length hitchin}. 
\end{itemize}
\end{rem}

We have the following two consequences of Theorem \ref{Thm:rigitity length hitchin}.

\begin{cor} \label{cor: asymm for hitchin roots}
Let $\g$ be an adjoint, simple, real-split Lie group of classical type. Let $\alpha$ be any simple root of $\sf G$, with the exception of the roots listed in Table \ref{table:1}. Then the function $d_{\tn{Th}}^{\alpha}: \Hit(S,\g) \times \Hit(S,\g)  \to \rr$ given by $$d_{\tn{Th}}^{\alpha}(\rho,\widehat{\rho}):=\log\left(\displaystyle\sup_{[\gamma]\in[\Gamma]}  \frac{ L_{\widehat{\rho}}^{\alpha}(\gamma) }{L_{\rho}^{\alpha}(\gamma)}\right)$$ \noindent defines an asymmetric distance on $\Hit(S,\sf G)$.

\end{cor}

\begin{proof}

By Potrie-Sambarino \cite[Theorem B]{PS} and P.-Sambarino-W. \cite[Theorem 9.9]{PSW1} we have $h^{\alpha}_{\rho}=1$ for all $\rho\in  \Hit(S,\g)$. Since roots as in the statement are not fixed by non inner automorphisms of $\g$, then by Theorems \ref{thm: dth for anosov} and \ref{Thm:rigitity length hitchin} the function $d_{\tn{Th}}^\alpha$ defines a possibly asymmetric metric. 

It remains to show that $d_{\tn{Th}}^\alpha$ is indeed asymmetric. But Thurston \cite[p.5]{ThurstonStretch} exhibits examples of points $\rho,\widehat{\rho}\in \Teich(S)$ for which the distance from $\rho$ to $\widehat{\rho}$ is different from the distance from $\widehat{\rho}$ to $\rho$. Since $\Hit(S,\g)$ contains a copy of $\Teich(S)$, the claim follows.

\end{proof}

\begin{cor} \label{cor: asymm for hitchin spectral}
Let $\g=\mathsf{PSL}(d,\rr)$ and $\varphi=\lambda_1$ be the spectral radius length function. Then the function $d_{\tn{Th}}^{\lambda_1}: \Hit_d(S) \times \Hit_d(S)  \to \rr$ given by $$d_{\tn{Th}}^{\lambda_1}(\rho,\widehat{\rho})=\log\left(\displaystyle\sup_{[\gamma]\in[\Gamma]} \frac{h_{\widehat{\rho}}^{\lambda_1}}{h_{\rho}^{\lambda_1}} \frac{ L_{\widehat{\rho}}^{\lambda_1}(\gamma) }{L_{\rho}^{\lambda_1}(\gamma)}\right)$$ \noindent defines an asymmetric distance on $\Hit_d(S)$.
\end{cor}

\begin{proof}
The action on $\liea$ of the unique non inner automorphism of $\mathsf{PSL}(d,\rr)$ coincides with the opposition involution $\iota$.
When $d>2$ note that $\lambda_1 \neq \lambda_1 \circ \iota$, hence in this case the result follows from Theorems \ref{thm: dth for anosov} and \ref{Thm:rigitity length hitchin}.  If $d=2$, the result follows from Theorem \ref{thm: dth for anosov} and the Length Spectrum Rigidity for hyperbolic surfaces. 
\end{proof}

We now turn to the proof of Theorem \ref{Thm:rigitity length hitchin}. In view of the natural inclusions $$\Hit(S,\mathsf{PSp}(2r,\rr))\subset \Hit_{2r}(S) \tn{ and } \Hit(S,\mathsf{SO}_0(p,p+1))\subset \Hit_{2p+1}(S),$$ \noindent we may assume that $\g$ is either $\mathsf{PSL}(d,\rr)$ or $\mathsf{PSO}_0(p,p)$. We will focus on the case $\g=\mathsf{PSO}_0(p,p)$, the argument for $\g=\mathsf{PSL}(d,\rr)$
is similar (and further, in that case the reader can also compare with \cite[Corollary 11.8]{BCLS}). 

The main step in the proof is to carefully analyse the possible Zariski closures of $\mathsf{PSO}_0(p,p)$-Hitchin representations, and show that they satisfy the hypotheses of Theorem \ref{thm:rigidity}. This is achieved in Corollaries \ref{cor: Z closure hitchin opp simple and center free} and \ref{cor: Z closure0 contains rho} below, as an application of recent work by Sambarino \cite{sambarino2020infinitesimal}.

Let then $p>2$ and consider a principal embedding $\tau:\mathsf{PSL}(2,\rr)\to\mathsf{PSO}_0(p,p)$. Then $\tau$ factors as $$\tau:\mathsf{PSL}(2,\rr)\to\mathsf{SO}_0(p,p-1)\to\mathsf{PSO}_0(p,p),$$ \noindent where the first map is the irreducible representation into $\mathsf{SL}(2p-1,\rr)$, and the second  is induced by the standard embedding stabilizing a non isotropic line $\ell_\tau\subset\rr^{2p}$. We let $\pi_\tau$ be the complementary $(p,p-1)$-hyperplane. Note that $\tau$ lifts to a principal embedding $\widehat{\tau}:\mathsf{PSL}(2,\rr)\to\mathsf{SO}_0(p,p)$. A \textit{Fuchsian} representation is a Hitchin representation into $\mathsf{PSO}_0(p,p)$ (resp. $\mathsf{SO}_0(p,p-1)$) whose image is contained in a conjugate of $\tau(\mathsf{PSL}(2,\rr))$ (resp. $\widehat{\tau}(\mathsf{PSL}(2,\rr))$). The following is well-known (see e.g. \cite[p.25]{sambarino2020infinitesimal} for a proof).

\begin{lema}\label{lem: hitchin PSOpp lifts}
Let $\rho\in\tn{Hit}(S,\mathsf{PSO}_0(p,p))$. Then there exists a representation $\widehat{\rho}:\Gamma\to\mathsf{SO}_0(p,p)$ lifting $\rho$ that may be deformed to a Fuchsian representation.
\end{lema}

Here is another useful lemma. 

\begin{lema}\label{lem: Zclosure hitchin is reductive}
Let $\widehat{\rho}:\Gamma\to\mathsf{SO}_0(p,p)$ be a Hitchin representation. Then the Zariski closure $\g_{\widehat{\rho}}$ of $\widehat{\rho}$ is reductive. 
\end{lema}

\begin{proof}

Suppose by contradiction that $\g_{\widehat{\rho}}$ is not reductive. 
	Then $\g_{\widehat{\rho}}$ is contained in a proper parabolic subgroup of $\mathsf{SO}_0(p,p)$ \cite{BorelTits}. 
	That is, $\widehat{\rho}(\Gamma)$ stabilizes a totally isotropic subspace $W$ of $\rr^{2p}$.
	
	Now the proof reduces to Sambarino \cite[Theorem B]{sambarino2020infinitesimal}. 
	Indeed, let $\lieg_{\widehat{\rho}}^{ss}$ be the semisimple part of the Lie algebra $\lieg_{\widehat{\rho}}$ of $\g_{\widehat{\rho}}$.
	By assumption, we have $\lieg_{\widehat{\rho}}^{ss}\neq \lieg_{\widehat{\rho}}$ and so \cite[Theorem B]{sambarino2020infinitesimal} implies that $\lieg_{\widehat{\rho}}^{ss}$ is either a principal $\mathfrak{sl}_2(\rr)$ or isomorphic to a copy of $\mathfrak{so}(p,p-1)$, stabilizing a non isotropic line $\ell$ and a complementary hyperplane $\pi$. 
	In either case, $\lieg_{\widehat{\rho}}^{ss}$ acts irreducibly on a hyperplane $\pi$ of signature $(p,p-1)$.
	But then $\ell\oplus W$ intersects $\pi$ non trivially and is $\lieg_{\widehat{\rho}}^{ss}$-invariant, a contradiction.

\end{proof}

 For a Hitchin representation $\widehat{\rho}:\Gamma\to \mathsf{SO}_0 (p,p)$, let $\lieg_{\widehat{\rho}}^{ss}$ be the semisimple part of the Lie algebra $\lieg_{\widehat{\rho}}$ of $\g_{\widehat{\rho}}$. By Sambarino \cite[Theorem A]{sambarino2020infinitesimal}, if $p\neq 4$ then $\lieg_{\widehat{\rho}}^{ss}$ is either $\mathfrak{so}(p,p)$, a principal $\mathfrak{sl}_2$, or the image of the standard embedding $\mathfrak{so}(p,p-1)\to\mathfrak{so}(p,p)$. In each case $\lieg_{\widehat{\rho}}^{ss}$ contains, up to conjugation, the Lie subalgebra $\od\widehat{\tau}(\mathfrak{sl}_2)$.

\begin{lema}\label{lem: center of Z closure hitchin opp}
Let $\widehat{\rho}:\Gamma\to\mathsf{SO}_0(p,p)$ be a Hitchin representation. Suppose that $g\in\g_{\widehat{\rho}}$ satisfies $ghg^{-1}=\pm h$ for all $h\in\g_{\widehat{\rho}}$. Then $g\in \{\tn{id},-\tn{id}\}$.
\end{lema}

\begin{proof}
Let $g\in\g_{\widehat{\rho}}$ be as in the statement. Since $\widehat{\tau}(\mathsf{PSL}(2,\rr))\subset(\g_{\widehat{\rho}})_0$, then $g$ centralizes (up to a sign) the principal $\mathsf{PSL}(2,\rr)$, which  factors through $\mathsf{SO}_0(p,p-1)$. 

Now if $h\in\mathsf{PSL}(2,\rr)$ is a hyperbolic element with eigenvalues $\pm\lambda$ (well defined up to $\pm 1$), then $\widehat{\tau}(h)$ acting on $\pi_\tau$ is diagonalizable with eigenvalues $$\lambda^{2(p-1)},\dots,\lambda^2,1,\lambda^{-2},\dots,\lambda^{-2(p-1)}.$$ \noindent Note that these are positive independently on whether we choose $\lambda$ or $-\lambda$ for the eigenvalues of $h$, hence to fix ideas we will assume $\lambda>1$. In particular, all the eigenvalues of $\widehat{\tau}(h)$ are positive. We let $\pi_h$ be the two dimensional plane spanned by $\ell_\tau$ and the eigenline in $\pi_\tau$ of eigenvalue $ 1$, which we denote by $\ell^1_h$. That is, $\pi_h$ is the eigenspace of $\widehat{\tau}(h)$ associated to the eigenvalue $1$. We also let $\ell_h^{i}$ be the eigenline of eigenvalue $i=\lambda^{2(p-1)},\dots,\lambda^2,\lambda^{-2},\dots,\lambda^{-2(p-1)}$.

Observe that actually $g\widehat{\tau}(h)g^{-1}=\widehat{\tau}(h)$. Indeed, otherwise we would have $g\widehat{\tau}(h)g^{-1}=-\widehat{\tau}(h)$ and for $v\in\ell_h^i$ one has $$g\cdot v=\frac{1}{\lambda^i}g\widehat{\tau}(h)\cdot v=-\frac{1}{\lambda^i}\widehat{\tau}(h)g\cdot v.$$ \noindent We would then find a negative eigenvalue of $\widehat{\tau}(h)$, a contradiction. We conclude that $g\widehat{\tau}(h)g^{-1}=\widehat{\tau}(h)$ as claimed.

It follows that $g$ preserves $\ell_h^i$ for all $i$, and also preserves $\pi_h$. We claim that $g$ preserves $\ell_\tau$. Indeed, note that there is some $m\in\mathsf{PSL}(2,\rr)$ so that $\widehat{\tau}(m)\cdot \ell_h^1 \neq \ell_h^1$, as the action of $\widehat{\tau}(\mathsf{PSL}(2,\rr))$ on $\pi_\tau$ is irreducible. Furthermore, $\widehat{\tau}(m)\cdot\ell_h^1$ is different from $\ell_h^i$, as all these lines are isotropic, while $\widehat{\tau}(m)\cdot\ell_h^1$ is not. By what we just proved, $g$ preserves $\pi_{mhm^{-1}}$ and therefore preserves $\pi_{mhm^{-1}}\cap\pi_h=\ell_\tau$. Hence $g\cdot\ell_\tau=\ell_\tau$ and therefore $g\cdot\ell_h^1=\ell_h^1$ for every hyperbolic $h\in\mathsf{PSL}(2,\rr)$.

We conclude that for every hyperbolic $h\in\mathsf{PSL}(2,\rr)$, the element $g$ preserves the projective basis $$\mathcal{B}_h:=\{\ell^{2(p-1)}_h,\dots,\ell^2_h,\ell^1_h,\ell_\tau,\ell^{-2}_h,\dots,\ell^{-2(p-1)}_h\}.$$ \noindent Fix such an $h$. Let $m\in\mathsf{PSL}(2,\rr)$ be so that $\widehat{\tau}(m)\cdot \ell_h^1\notin\mathcal{B}_h$. Then $g$ preserves the elements of the basis $\mathcal{B}_{mhm^{-1}}$ as well, and therefore preserves $2p+1$ lines {in general position} in $\rr^{2p}$. It follows that $g=\mu\tn{id}$ for some $\mu\in\rr$. Since $g\in\mathsf{SO}_0(p,p)$, then $\mu=\pm 1$.
\end{proof}

\begin{cor}\label{cor: Z closure hitchin opp simple and center free}
Assume $p\neq 4$ and let $\rho\in\tn{Hit}(S,\mathsf{PSO}_0(p,p))$. Then the Zariski closure $\g_\rho$ of $\rho$ is simple and center free, and with Lie algebra $\mathfrak{so}(p,p)$, $\mathfrak{so}(p,p-1)$, or a principal $\mathfrak{sl}_2$.
\end{cor}

\begin{proof}
Let $\widehat{\rho}$ be a lift of $\rho$. Then $\g_\rho=\g_{\widehat{\rho}}/\{\pm \tn{id}\}$ and by Lemmas \ref{lem: Zclosure hitchin is reductive} and \ref{lem: center of Z closure hitchin opp}, $\g_\rho$ is reductive and center free. In particular, it is semisimple and by Sambarino \cite[Theorem A]{sambarino2020infinitesimal} the result follows.
\end{proof}

The proof of the following well-known fact can be found in \cite[Corollary 6.2]{sambarino2020infinitesimal} for $\mathsf{PSL}(d,\rr)$-Hitchin representations, but the proof applies in our setting.

\begin{cor}\label{cor: Z closure0 contains rho}
Let $\rho\in\tn{Hit}(S,\mathsf{PSO}_0(p,p))$. Then $\rho(\Gamma)\subset(\g_\rho)_0$.
\end{cor}

We have now completed the analysis of the possible Zariski closures of $\mathsf{PSO}_0(p,p)$-Hitchin representations, and we can prove Theorem \ref{Thm:rigitity length hitchin}.

\begin{proof}[Proof of Theorem \ref{Thm:rigitity length hitchin}]
By Corollaries \ref{cor: Z closure hitchin opp simple and center free} and \ref{cor: Z closure0 contains rho} and Theorem \ref{thm:rigidity} there exists an isomorphism $\sigma:(\g_\rho)_0\to(\g_{\widehat{\rho}})_0$ so that $\sigma\circ\rho=\widehat{\rho}$. In particular, $(\g_\rho)_0\cong(\g_{\widehat{\rho}})_0$ and we have three possibilities. If $(\g_\rho)_0$ is a principal $\mathsf{PSL}(2,\rr)$, then the result follows from Length Spectrum Rigidity in Teichm\"uller space. If $(\g_\rho)_0\cong \mathsf{PSO}_0(p,p-1)$, then the corresponding Dynkin diagram is of type $\mathsf{B}_{p-1}$ and therefore admits no non trivial automorphism. Hence, in that case $\sigma$ is inner as desired. 

Finally, assume $(\g_\rho)_0=\mathsf{PSO}_0(p,p)$ and suppose by contradiction that $\rho\neq\widehat{\rho}$. Hence $\sigma$ is a non internal automorphism. But on the other hand by Theorem \ref{thm:rigidity} we have $\varphi\circ\sigma=\varphi$, contradicting our hypothesis.

\end{proof}

\begin{rem}\label{rem: asymmetric}
A natural length function on $\Hit_d(S)$, specially relevant in the case $d=3$, is the Hilbert length (c.f. Example \ref{list, Lengths}). However, the Hilbert length is not rigid, as the \textit{contragredient} representation $\rho^\star(\gamma):=\leftidx{^t}\rho(\gamma)^{-1}$ of $\rho$ satisfies $h_\rho^{\tn{H}}L_\rho^{\tn{H}}=h_{\rho^\star}^{\tn{H}}L_{\rho^\star}^{\tn{H}}$, but in general one has $\rho^\star\neq\rho$. Hence, $d_{\tn{Th}}^{\tn{H}}(\cdot,\cdot)$ does not separate points of $\Hit_d(S)$. It follows from the proof of Theorem \ref{Thm:rigitity length hitchin} that this is the only possible situation where two different $\mathsf{PSL}(d,\rr)$-Hitchin representations can have the same Hilbert length spectra. Similar comments apply to the simple roots listed in Table \ref{table:1}.
   
\end{rem}

\subsection{Simple root and spectral radius Finsler norms}\label{subsec: finsler for hitchin}
We  now restrict to $\g=\mathsf{PSL} (d,\rr)$. % and recall that $\Hit_d(S)$ denotes a corresponding Hitchin component. 
We list some useful consequences of Corollary \ref{cor: finsler for reps} and \cite{BCLS,BCLSSIMPLEROOTS}. For the first simple root we have the following.

\begin{cor}\label{cor: finsler norm hitchin first root}
Let $\varphi=\alpha_1\in \Pi$ be the first simple root. 
The function on $T \Hit_d(S)$
$$\Vert v\Vert_{\tn{Th}}^{\alpha_1}=\displaystyle\sup_{[\gamma]\in[\Gamma]} \frac{ \od_\rho (L_{\cdot}^{\alpha_1}(\gamma))(v) }{L_{\rho}^{\alpha_1}(\gamma)}$$ \noindent defines a Finsler norm on $\Hit_d(S)$.

\end{cor}

\begin{proof}
By Potrie-Sambarino \cite[Theorem B]{PS} we have $h_\rho^{\alpha_1}=1$ for all $\rho\in\Hit_d(S)$. Hence, thanks to Corollary \ref{cor: finsler for reps} we only have to show that $\Vert v\Vert_{\tn{Th}}^{\alpha_1}=0$ implies $v=0$. But this follows from Corollary \ref{cor: finsler for reps} and \cite[Theorem 1.7]{BCLSSIMPLEROOTS}: the set $\{\od_\rho(L_\cdot^{\alpha_1}(\gamma))\}_{\gamma\in\Gamma}$ generates the cotangent space $T_\rho^*\Hit_d(S)$.
\end{proof}

When $d=2j>2$, it is shown in \cite[Proposition 8.1]{BCLSSIMPLEROOTS} that the middle root pressure quadratic form is degenerate along representations that factor through $\mathsf{PSp}(2j,\rr)$. The proof shows that $\Vert\cdot\Vert_{\tn{Th}}^{\alpha_j}$ is degenerate as well.

With the same argument as in Corollary \ref{cor: finsler norm hitchin first root} (but applying \cite[Lemma 9.8 \& Proposition 10.1]{BCLS} instead of \cite[Theorem 1.7]{BCLSSIMPLEROOTS}), we obtain the following.

\begin{cor}\label{cor: finsler norm hitchin spectral}
Let $\varphi=\lambda_1$ be the spectral radius length function. Then the function
$\Vert\cdot\Vert_{\tn{Th}}^{\lambda_1}: T \Hit_d(S)  \to \rr_{\geq 0}$, taking $v\in T_\rho \Hit_d(S)$ to $$\Vert v\Vert_{\tn{Th}}^{\lambda_1}=\displaystyle\sup_{[\gamma]\in[\Gamma]} \frac{\od_{\rho}(h_{\cdot}^{\lambda_1})(v)L_\rho^{\lambda_1}(\gamma)+h_\rho^{\lambda_1} \od_\rho(L_\cdot^{\lambda_1}(\gamma))(v)}{h_\rho^{\lambda_1} L_\rho^{\lambda_1}(\gamma)}$$ \noindent defines a Finsler norm on $\Hit_d(S)$.\end{cor}

We finish this subsection with a comment on Labourie and Wentworth work \cite{VariationAlongFuchsian}, which explicitly compute the derivative of the spectral radius and simple root length functions at points of the Fuchsian locus $\Teich(S)\subset \Hit_d(S)$, along some special directions. More explicitly, fixing a Riemann surface structure $X_0$ on $S$,  the canonical line bundle $K$ associated to $X_0$ is the $(1,0)$-part of the complexified cotangent bundle $T^*X_0^{\mathbb{C}}=\mathbb{C}\otimes_{\mathbb{R}}T^*X_0$. An \textit{holomorphic $k$-differential} is an holomorphic section of the bundle $K^k$, where the power $k$ is taken with respect to tensor operation. In local holomorphic coordinates $z=x+i y$, an holomorphic $k$-differential can be written as $$q_k=q_k(z) 
\underbrace{dz \otimes \cdots \otimes dz}_{k \text{ times}  } = q_k(z)dz^k,$$ \noindent with $q_k(z)$ holomorphic. Hitchin's seminal work \cite{Hitchin_HitchinComponent} parametrizes $\Hit_d(S)$ by the space of holomorphic differentials over $X_0$. More precisely, there exists a homemomorphism $$\Hit_d(S)\cong \bigoplus\limits_{k=2}^{d} H^{0}(X_0, K^k),$$ \noindent where $H^{0}(X_0, K^k)$ denotes the space of holomorphic $k$-differentials over $X_0$. Given an holomorphic $k$-differential $q_k\in H^0(X_0,K^k)$, one may consider a natural family of Hitchin representations $\{\rho_t\}_{t\geq 0}$, corresponding to $\{tq_k\}_{t\geq 0}\subset H^{0}(X_0, K^k)$ under this parametrization, with $\rho_0$ corresponding to the point $X_0$ in the Teichm\"uller space $\Teich(S)$. Infinitesimally, this gives a vector space isomorphism: $$T_{\rho_0} \Hit_d(S) \cong \bigoplus\limits_{k=2}^{d} H^{0}(X_0, K^k). $$ 

Given a family of Hitchin representations $\{\rho_t\}_{t\geq 0}$ as above, we denote by $v=v(q_k):=\frac{d}{dt}\big|_{t=0} \rho_t \in T_{X_0} \Hit_d(S)$ the corresponding tangent vector. The computation of the derivatives $\od_{\rho_0}(L_\cdot^{\lambda_j}(\gamma))(v)$, for $1\leq j \leq d$, has been carried out by Labourie-Wentworth \cite[Theorem 4.0.2]{VariationAlongFuchsian}, using the above identification and information of $H^{0}(X_0, K^k)$. 
To be more precise, define the function $\textnormal{Re } q_k : T^1X_0 \to \mathbb{R}$ as the real part of the holomorphic differential $q_k$ evaluated on unit tangent vectors. More precisely, $$\textnormal{Re } q_k(x):=\textnormal{Re } \bigg(  q_k|_{p}(w,w ,\cdots, w) \bigg) $$ \noindent for $x=(p,w)\in T^1X_0$.

Let $\phi$ be the geodesic flow on $T^1X_0$. For $\gamma\in\Gamma$, let $l_{\rho_0}(\gamma):=\frac{2}{d-1}L^{\lambda_1}_{\rho_0}(\gamma)$ be the hyperbolic length of the closed geodesic on $X_0$ corresponding to the free homotopy class $[\gamma]$. 

\begin{prop}\label{prop: finsler fuchsian locus hitchin} There exist constants $C_1$ and $C_2$, only depending on $d$ and $k$, such that for any  vector $v=v(q_k)\in T_{X_0} \Hit_d(S)$ as above, $$\Vert v(q_k)\Vert_{\tn{Th}}^{\lambda_1}=C_1\displaystyle\sup_{[\gamma]\in[\Gamma]} \frac{1}{l_{\rho_0}(\gamma)} \int_0^{l_{\rho_0}(\gamma)}   \textnormal{Re } q_k(\phi_s(x)) \od s $$\noindent and $$\Vert v(q_k)\Vert_{\tn{Th}}^{\alpha_1}=C_2\displaystyle\sup_{[\gamma]\in[\Gamma]} \frac{1}{l_{\rho_0}(\gamma)} \int_0^{l_{\rho_0}(\gamma)} \textnormal{Re } q_k(\phi_s(x)) \od s, $$ \noindent where $x=x_\gamma$ is any point on $T^1X_0$ that lies in the periodic orbit corresponding to $\gamma$.
\end{prop}

\begin{proof}
The proof is a simple combination of Definition \ref{dfn: finsler reps} together with \cite[Theorem 4.0.2, Corollary 4.0.5.]{VariationAlongFuchsian}. One also needs the fact that $h_{\rho}^{\lambda_1}\leq 1$ with equality precisely when $\rho$ is Fuchsian, and $h_\cdot^{\alpha_1}\equiv 1$ (by \cite[Theorem B]{PS}).
\end{proof}

\section{Other examples}\label{s.other}

As discussed  in the Introduction in \S\ref{subsec: outlineINTRO}, we need two ingredients to gain a good understanding of the asymmetric metric $d_{\tn{Th}}^\varphi(\cdot,\cdot)$: 
\begin{itemize}
\item A reparametrization of the geodesic flow of $\G$ with periods given by the functional $\varphi$: this is needed to show that $d_{\tn{Th}}^\varphi(\cdot,\cdot)$ is non-negative, degenerating if and only if the renormalized length spectra coincide. Sambarino provides such a reparametrization  whenever $\varphi\in\tn{int}((\cone)^*)$ and $\Theta$ is the set of Anosov roots (see Section \ref{sec: anosov flows and reps}).
\item A good understanding of the Zariski closure and its outer automorphism group for representations belonging to a given class of interests: this is necessary to obtain renormalized length spectrum rigidity.
\end{itemize}

Furthermore on subsets of representations for which the entropy of some functional is constant, one can avoid the renormalization by entropy.%, which might seem geometrically  less natural.

We discuss here further classes in which simultaneous knowledge of some of these aspects can be achieved.

\subsection{Benoist representations}\label{subsec: benoist}

Let $\G$ be a torsion free word hyperbolic group. A \textit{Benoist representation} is a faithful and discrete representation $\rho:\G\to\mathsf{PSL}(d+1,\rr)$ dividing an open, strictly convex set $\Omega_\rho\subset\rr\pp^{d}$ (recall Example \ref{ex: hyperconvex}). We denote by $\Ben_d(\G)\subset\frak X(\G,\mathsf{PSL}(d+1,\rr))$ the space of conjugacy classes of Benoist representations. Koszul \cite{Koszul} showed that $\Ben_d(\G)$ is an open subset of the character variety, and Benoist \cite{BenoistDivIII} showed it is closed. Hence, $\Ben_d(\G)$ is a union of connected components of $\frak X(\G,\mathsf{PSL}(d+1,\rr))$.

As Benoist representations are $\Theta$-Anosov for $\Theta=\{\alpha_1,\alpha_d\}$, both the unstable Jacobian $\tn {J}_{d-1}:=d\omega_1-\omega_{d}=d\lambda_1+\lambda_{d+1}$ and $\tn{H}:=\lambda_{1}-\lambda_{d+1}$ belong to $\tn{int}((\cone)^*)$ for every $\rho\in\Ben_d(\G)$. We focus here on these two functionals since it was proven in  \cite[Corollary 7.1]{PS} that  $\tn J_{d-1}$ has constant entropy, and the Hilbert length function has particular geometric significance as  $L_\rho^\tn{H}(\gamma)$ coincides with the length of the unique Hilbert geodesic in $\rho(\G)\backslash \Omega_\rho$ in the isotopy class corresponding to $[\gamma]$.

\begin{cor} \label{cor: unst Jac for benoist hilbert}

The function $d_{\tn{Th}}^{{\tn J_{d-1}}}: \Ben_d(\G) \times \Ben_d(\G)  \to \rr$ given by 
$$d_{\tn{Th}}^{{\tn J_{d-1}}}(\rho,\widehat{\rho})=\log\left(\displaystyle\sup_{[\gamma]\in[\Gamma]}  \frac{ L_{\widehat{\rho}}^{\tn{J}_{d-1}}(\gamma) }{L_{\rho}^{\tn J_{d-1}}(\gamma)}\right)$$ \noindent defines a (possibly asymmetric) distance on $\Ben_d(\G)$.
\end{cor}

\begin{proof}

Benoist \cite[Th\'eor\`eme 3.6]{BenoistAutomorphismes} showed that if $\rho\in\tn{Ben}_d(\Gamma)$ is not Zariski dense, then $\rho(\Gamma)\subset\mathsf{PSO}(d,1)$. Hence, by Theorems \ref{thm: dth for anosov} and \ref{thm:rigidity}, if $d_{\tn{Th}}(\rho,\widehat{\rho})=0 $ then there exists an isomorphism  $\sigma:(\g_\rho)_0\to(\g_{\widehat{\rho}})_0$ so that $\sigma\circ\rho=\widehat{\rho}$. If $(\g_\rho)_0\cong(\g_{\widehat{\rho}})_0\cong\mathsf{PSO}_0(1,d)$, then the equality $\rho=\widehat{\rho}$ follows from Length Spectrum Rigidity in Teichm\"uller space (when $d=2$), or by Mostow rigidity (when $d>2$). 

On the other hand, if $(\g_\rho)_0\cong(\g_{\widehat{\rho}})_0\cong\mathsf{PSL}(d+1,\rr)$ and $\sigma$ is non inner, it acts non trivially on the Dynkin diagram of type $\mathsf{A}_d$, hence its action on $\liea$ coincides with the opposition involution $\iota$. Since $\tn J_{d-1}$ is not $\iota$-invariant, and has constant entropy by \cite[Corollary 7.1]{PS}, Corollary \ref{cor: distance in zariski dense components} finishes the proof.
\end{proof}
\begin{rem}
The same applies for all $(1,1,p)$-hyperconvex representations $\rho:\G\to\PSL(d,\R)$ of hyperbolic groups having as boundary a $(p-1)$-dimensional sphere (see Example \ref{ex: hyperconvex}): it follows from \cite[Proposition 7.4]{PSW1} that their projective limit set is a $\tn C^1$-sphere, and from \cite[Theorem A]{PSW2} that then the entropy of the unstable Jacobian $\tn J_{p-1}:=p\omega_1-\omega_p$ is constant and equal to 1. If we then denote by $\tn{Hyp}^{\tn Z}(\G)$ the open subset of the character variety consisting of Zariski dense $(1,1,p)$-hyperconvex representations, the function 
$$d_{\tn{Th}}^{\tn{J}_{p-1}}(\rho,\widehat{\rho})=\log\left(\displaystyle\sup_{[\gamma]\in[\Gamma]}  \frac{ L_{\widehat{\rho}}^{\tn{J}_{p-1}}(\gamma) }{L_{\rho}^{\tn J_{p-1}}(\gamma)}\right)$$ \noindent defines a (possibly asymmetric) distance on $\tn{Hyp}^{\tn Z}(\G)$.
\end{rem}

With the same proof as in Corollary \ref{cor: unst Jac for benoist hilbert} we get the following result.

\begin{cor} \label{cor: asymm for benoist hilbert}
 
The function $d_{\tn{Th}}^{\tn{H}}: \Ben_d(\G) \times \Ben_d(\G)  \to \rr$ given by $$d_{\tn{Th}}^{\tn{H}}(\rho,\widehat{\rho})=\log\left(\displaystyle\sup_{[\gamma]\in[\Gamma]} \frac{h_{\widehat{\rho}}^{\tn{H}}}{h_{\rho}^{\tn{H}}} \frac{ L_{\widehat{\rho}}^{\tn{H}}(\gamma) }{L_{\rho}^{\tn{H}}(\gamma)}\right)$$ \noindent is real-valued, non-negative and $d_{\tn{Th}}^{\tn{H}}(\rho,\widehat{\rho})=0$ if and only if 
$\rho=\widehat{\rho}$ or  $\rho=\widehat{\rho}^\star$, where $\rho^\star(\gamma):=\leftidx{^t}\rho(\gamma)^{-1}$ for all $\gamma\in\Gamma$.
\end{cor}

\begin{rem}\label{rem: other functionals in benoist components}
The Hilbert length function $\tn H$ is the only element in $\tn{int}((\cone)^*)$ which is fixed by the opposition involution, and the unstable Jacobian $\tn J_{d-1}$ and its image $\tn J_{d-1}\circ \iota=-d\lambda_{d+1}-\lambda_1$ are the only elements in $\tn{int}((\cone)^*)$ that have constant entropy  on the whole $\Ben_d(\G)$. In particular for all other functionals $\varphi\in\tn{int}((\cone)^*)$, such as for example the spectral radius $\lambda_1$,
$$d_{\tn{Th}}^{\varphi}(\rho,\widehat{\rho}):=\log\left(\displaystyle\sup_{[\gamma]\in[\Gamma]} \frac{h_{\widehat{\rho}}^{\varphi}}{h_{\rho}^{\varphi}} \frac{ L_{\widehat{\rho}}^{\varphi}(\gamma) }{L_{\rho}^{\varphi}(\gamma)}\right)$$ 
\noindent defines a (possibly asymmetric) distance on $\Ben_d(\G)$. In all these cases the renormalization by entropy is, however, necessary.
\end{rem}

\subsection{AdS-quasi-Fuchsian representations}\label{subsec: AdSqf}

Let $q\geq 2$ and $\Gamma$ be the fundamental group of a closed real hyperbolic $q$-dimensional manifold. Denote by $\tn{QF}_q(\Gamma)$ the space of AdS-quasi-Fuchsian representations $\Gamma\to\mathsf{PO}_0(2,q)$, which is a union of connected components of the character variety (recall Example \ref{ex: AdSquasi-fuchsian}). Since representations in $\tn{QF}_q(\Gamma)$ are Anosov with respect to the space of isotropic lines, the Hilbert length functional $\tn{H}=\omega_1-\omega_{q+1}$ belongs to the Anosov-Levi space $\liea_\Theta^*$. This functional is a multiple of the spectral radius functional on $\mathsf{PO}_0(2,q)$.

\begin{cor} \label{cor: asymm for AdSQF hilbert}
If $q>2$,  the function $d_{\tn{Th}}^{\tn{H}}: \tn{QF}_q(\Gamma) \times \tn{QF}_q(\Gamma)  \to \rr$ given by $$d_{\tn{Th}}^{\tn{H}}(\rho,\widehat{\rho})=\log\left(\displaystyle\sup_{[\gamma]\in[\Gamma]} \frac{h_{\widehat{\rho}}^{\tn{H}}}{h_{\rho}^{\tn{H}}} \frac{ L_{\widehat{\rho}}^{\tn{H}}(\gamma) }{L_{\rho}^{\tn{H}}(\gamma)}\right)$$ \noindent defines a (possibly asymmetric) distance on $\tn{QF}_q(\Gamma)$.
\end{cor}

\begin{proof}

For $q>2$ the group $\mathsf{PO}_0(2,q)$ is simple, and the associated root system is of type $\mathsf{B}_2$. In particular, it has no non trivial automorphisms and therefore an automorphism of $\mathsf{PO}_0(2,q)$ is necessarily inner. Corollary \ref{cor: distance in zariski dense components} then proves the result when restricting to Zariski dense AdS-quasi-Fuchsian representations. 

Furthermore Glorieux-Monclair \cite[Proposition 1.4]{GMRegularity} computed the possible Zariski closures of an AdS-quasi-Fuchsian representation: if $\rho$ is not Zariski dense, then it is AdS-Fuchsian. This means that $\rho$ preserves a totally geodesic copy of $\mathbb{H}^q$ inside the Anti-de Sitter space and acts co-compactly on it (c.f. \cite[Remark 1.13]{DGKHpqCC}). Therefore $\rho(\Gamma)\subset\mathsf{PO}(1,q)\subset\mathsf{PO}_0(2,q)$. Hence the Length Spectrum Rigidity of closed real hyperbolic manifolds finishes the proof.

\end{proof}

In the special case $q=2$, the function $d_{\tn{Th}}^{\tn{H}}$ does not separate points. Indeed $\mathsf{PSO}_0(2,2)\cong\mathsf{PSL}(2,\rr)\times\mathsf{PSL}(2,\rr)$ and every representation of the form $$\rho=(\rho^{\tn{L}},\rho^{\tn{R}}):\pi_1(S)\to\mathsf{PSL}(2,\rr)\times\mathsf{PSL}(2,\rr),$$ \noindent where $\rho^\varepsilon$ is a point in Teichm\"uller space for $\varepsilon\in\{\tn{L},\tn{R}\}$, is AdS-quasi-Fuchsian. However, the representation $\widehat{\rho}:=(\rho^{\tn{R}},\rho^{\tn{L}})$ has the same Hilbert length spectrum as $\rho$, but $\rho\neq \widehat{\rho}$ (unless $\rho^{\tn{L}}=\rho^{\tn{R}}$).

\begin{rem}
Since AdS-quasi-Fuchsian representations have Lipschitz limit set, it follows again from \cite[Theorem A]{PSW2} that  the entropy of the unstable Jacobian $\tn J_{q-1}:=q\omega_1-\omega_q$ is constant and equal to 1 on $\tn{QF}_q(\Gamma)$. In particular, the function $d_{\tn{Th}}^{{\tn J_{q-1}}}: \tn{QF}_q(\Gamma) \times \tn{QF}_q(\Gamma)  \to \rr$ given by 
$$d_{\tn{Th}}^{{\tn J_{q-1}}}(\rho,\widehat{\rho}):=\log\left(\displaystyle\sup_{[\gamma]\in[\Gamma]}  \frac{ L_{\widehat{\rho}}^{\tn{J}_{q-1}}(\gamma) }{L_{\rho}^{\tn J_{q-1}}(\gamma)}\right)$$ \noindent is non negative. 

However, in this case the unstable Jacobian doesn't belong to the Levi-Anosov subspace. As a result it is not clear whether a metric Anosov flow with periods $\tn J_{q-1}$ exists allowing us to apply the Thermodynamical Formalism which is at the basis of this work. Thus, we don't know if the condition $d_{\tn{Th}}^{\tn J_{q-1}}(\rho,\widehat{\rho})=0$ leads to an equality between length spectra that allows to conclude that $d_{\tn{Th}}^{\tn J_{q-1}}$ separates points.
\end{rem}

\subsection{Zariski dense  $\Theta$-positive representations in $\mathsf{PO}_0(p,p+1)$}\label{s.Zdense}
Let $2\leq p\leq q$. Let $\G=\pi_1(S)$ be a surface group and $\tn{Pos}_{p,q}(S)$ be the space of $\Theta$-positive representations $\Gamma\to\mathsf{PO}_0(p,q)$ (c.f. Example \ref{ex:positive}).

\begin{cor}\label{cor: asymm for maximal sp4 root}
For $2<p\leq q$ and $j=1,\dots,p-2$ let $\alpha_j$ be the corresponding simple root of ${\sf PO}_0(p,q)$. Let $\Pos^\tn{Z}_{p,q}(\G)\subset\Pos_{p,q}(\G)$ be the subset consisting of Zariski dense representations. Then the function $$d_{\tn{Th}}^{\alpha_j}: \Pos^\tn{Z}_{p,q}(\G) \times \Pos^\tn{Z}_{p,q}(\G)  \to \rr$$ \noindent given by $$d_{\tn{Th}}^{\alpha_j}(\rho,\widehat{\rho})=\log\left(\displaystyle\sup_{[\gamma]\in[\Gamma]}  \frac{ L_{\widehat{\rho}}^{\alpha_j}(\gamma) }{L_{\rho}^{\alpha_j}(\gamma)}\right)$$ \noindent defines a (possibly asymmetric) distance on $\Pos^\tn{Z}_{p,q}(\G)$.
\end{cor}

\begin{proof}

As ${\sf PO}_0(p,q)$ $\Theta$-positive representations are $\Theta$-Anosov for $\Theta=\{\alpha_1,\dots,\alpha_{p-1}\}$ (see \cite{GLW, BeyPSO}), we have $\alpha_j\in\tn{int}((\cone)^*)$ for every $\rho\in\tn{Pos}^{\tn{Z}}_{p,q}(\Gamma)$. Furthermore, $\alpha_j$-entropy is constant on the space of ${\sf PO}_0(p,q)$ positive representations \cite[Corollary 1.7]{PSW2}. Thus to finish the proof it only remains to show that $\alpha_j$-length spectrum rigidity holds on $\tn{Pos}^{\tn{Z}}_{p,q}(\Gamma)$.  

Since $\mathsf{PO}_0(p,q)$ is simple and center free, Theorem \ref{thm:rigidity} guarantees that two representations in $\tn{Pos}^{\tn{Z}}_{p,q}(\Gamma)$ having the same renormalized length spectra differ by an automorphism of $\mathsf{PO}_0(p,q)$. Since the Dynkin diagram associated to the root system of $\mathsf{PO}_0(p,q)$ is of type of type $\mathsf{B}_p$ and admits no non trivial automorphism, the outer automorphism group of $\mathsf{PO}_0(p,q)$ is trivial and this finishes the proof.

\end{proof}

\begin{rem}
The space $\tn{Pos}_{2,3}(\Gamma)$ contains connected components only consisting of Zariski dense representations \cite[Theorem 4.40]{AlessandriniCollier}. More generally, for all $p>2$ the space $\tn{Pos}_{p,p+1}(\Gamma)$ contains smooth connected components.  It is conjectured that these consist only of Zariski dense representations as well (see \cite[Conjecture 1.7]{Collier}), if the conjecture were true, the functions in Corollary \ref{cor: asymm for maximal sp4 root} would define metrics on these connected components.

On the other hand it follows from the classification in \cite{ABCGGO} that for $q\geq p$ all connected components of $\tn{Pos}_{p,q}(S)$, with the exception of the Hitchin component if $p=q$ contain representations with compact centralizer. 
\end{rem}

\appendix

\section{Geodesic currents}\label{appendix: currents}

Bridgeman-Canary-Labourie-Sambarino \cite[p.60]{BCLSSIMPLEROOTS} remarked that the renormalized intersection number of Subsection \ref{subsec: inter and renormalized intersection} can be linked to \textit{Bonahon's intersection number}, in the specific case of geodesic flows associated to points in the Teichm\"uller space of a surface. We explain this in more detail for the reader's convenience. 

Let $S$ be a connected closed orientable surface of genus bigger than one and $\tilde{S}$ be its universal cover. Let $\Gamma$ be the fundamental group of $S$. A \textit{(complete) geodesic} of $\tilde{S}$ is an element of $\bgs$. A \textit{geodesic current} is a Borel, locally finite, $\Gamma$-invariant measure on the space of geodesics of $\tilde{S}$, which is also invariant under the map $(x,y)\mapsto(y,x)$. We let $\mathscr{C}(S)$ be the space of geodesic currents in $S$. An important example of a geodesic current is given by isotopy classes of closed curves in $S$: every such class $\alpha$ defines an element $\delta_{\alpha}\in\mathscr{C}(S)$ by representing $\alpha$ as a conjugacy class $c_\alpha$ in $\Gamma$, and then considering the sum of Dirac masses supported on the axes of elements in $c_\alpha$. Another interesting example is given by \textit{measured geodesic laminations} on $S$ (c.f. Bonahon \cite[p. 153]{BonahonCurrents}).

Bonahon \cite{BonahonCurrents} defined a continuous, bilinear, symmetric pairing $$i:\mathscr{C}(S)\times\mathscr{C}(S)\to\rr_{\geq 0},$$
\noindent called the \textit{intersection number} between geodesic currents. This terminology is motivated by the following property: if $\alpha$ and $\beta$ are isotopy classes of closed curves in $S$, then one has
$$i(\delta_\alpha,\delta_\beta)=\displaystyle\inf_{\alpha'\in\alpha,\beta'\in\beta}\#(\alpha'\cap\beta').$$

Furthermore, Bonahon defines an embedding $$\mathtt{L}:\Teich(S)\hookrightarrow\mathscr{C}(S)$$ \noindent from the Teichm\"uller space $\Teich(S)$ into the space of geodesic currents that can be described as follows. Since every point $\rho\in \teichrep$ is Anosov, we have an equivariant limit map $\xi_\rho:\bg\to\pp(\rr^2)$ and we may pull back the Haar measure on $\pp(\rr^2)\times\pp(\rr^2)\setminus\{(\eta,\eta):\hspace{0,2cm} \eta\in\pp(\rr^2)\}$ under this map. We obtain an element $\mathtt{L}_\rho\in\mathscr{C}(S)$ which is called the \textit{Liouville current} of $\rho$. Furthermore, the Haar measure on $\pp(\rr^2)\times\pp(\rr^2)\setminus\{(\eta,\eta):\hspace{0,2cm} \eta\in\pp(\rr^2)\}$ can be normalized so that for every isotopy class of closed curves $\alpha$ in $S$
\begin{equation}\label{eq: intersection liouville}
  i(\mathtt{L}_\rho,\delta_\alpha)=L_\rho(\alpha),
\end{equation}
\noindent where $L_\rho(\alpha)$ is the length of the unique closed geodesic (for the metric $\rho$) in the isotopy class $\alpha$ (c.f. \cite[Proposition 14]{BonahonCurrents}).

The embedding $\mathtt{L}:\Teich(S)\to\mathscr{C}(S)$ allows us to relate renormalized intersection and Bonahon's intersection. Indeed, pick a base point $\rho_0\in \Teich(S)$ and denote by $S_{\rho_0}$ the underlying hyperbolic surface. The associated geodesic flow $\phi=\phi_{\rho_0}$ is a topologically transitive Anosov flow and admits a strong Markov coding (c.f. Theorem \ref{teo: pollicott coding metric anosov}). Furthermore, the choice of $\rho_0$ induces a homeomorphism between $\mathscr{C}(S)$ and the space $\pphi$. Indeed, the Busemann-Iwasawa cocycle of $\rho_0$ induces an identification between the unit tangent bundle of the Riemannian universal cover of $S_{\rho_0}$ with $$\bgs\times\rr,$$
\noindent in such a way that the action of the (lifted) geodesic flow is given by translation in the $\rr$-coordinate. The identification $\mathscr{C}(S)\cong\pphi$ is defined by associating to a geodesic current $\nu$ the probability measure $m_\nu$ homothetic to the quotient measure of $\nu\otimes\od t$.

 The geodesic flow $\psi=\psi_\rho$ corresponding to another point $\rho\in \Teich(S)$ is H\"older orbit equivalent to $\phi=\phi_{\rho_0}$, and therefore we may think $\psi$ as an element of\footnote{Formally, there is no canonical way of identifying $\psi$ with a specific reparametrization of $\phi$, but just to a Liv\v{s}ic cohomology class (c.f. Liv\v{s}ic's Theorem \ref{teo: livsic}). For simplicity we will ignore this detail in this discussion and think that the choice of $\rho$ induces a specific element $\psi\in\hr$. As it will become clear, the discussion is independent of this arbitrary choice.} $\hr$.

\begin{lema}\label{lem: renormalized and bonahon}
Let $\rho_0$ and $\rho$ be two points in $\Teich(S)$ and take $\nu\in\mathscr{C}(S)$. Then: $$\mathbf{I}_{m_\nu}(\phi,\psi)=\mathbf{J}_{m_\nu}(\phi,\psi)=\frac{i(\nu,\mathtt{L}_\rho)}{i(\nu,\mathtt{L}_{\rho_0})}.$$
\end{lema}

\begin{proof}

The function $\mathbf{J}_{\cdot}(\phi,\psi)$ is continuous on $\pphi$. Similarly, $i(\cdot,\mathtt{L}_{\rho_0}) \tn{ and } i(\cdot,\mathtt{L}_{\rho})$ are continuous on $\mathscr{C}(S)$. Since $\nu\mapsto m_\nu$ is a homeomorphism and multi-curves are dense in $\mathscr{C}(S)$ (see Bonahon \cite[Proposition 2]{BonahonCurrents}), it suffices to prove the statement for $\nu=\delta_\alpha$, where $\alpha$ is any isotopy class of closed curves in $S$. 

Assume then $\nu=\delta_\alpha$. By Equation (\ref{eq: intersection liouville}) we have
$$i(\nu,\mathtt{L}_{\rho_0})=L_{\rho_0}(\alpha) \tn{ and } i(\nu,\mathtt{L}_\rho)=L_{\rho}(\alpha).$$
\noindent On the other hand, it is well known that $h_{\phi}=h_{\psi}=1$ (c.f. Manning \cite{ManningEntropy}), hence $\mathbf{J}_{m_\nu}(\phi,\psi)=\mathbf{I}_{m_\nu}(\phi,\psi)$. Also, $\alpha$ defines a periodic orbit $a_\alpha\in\mathcal{O}$ satisfying $p_{\phi}(a_\alpha)=L_{\rho_0}(\alpha)$ and $p_{\psi}(a_\alpha)=L_\rho(\alpha)$. Since $m_{\delta_\alpha}=\delta_{\phi}(a_\alpha)$, Equation (\ref{eq: integral of reparametrizing over delta in periodic orbit}) completes the proof. 
\end{proof}

One can check that $m^{\tn{BM}}(\phi)=m_{\mathtt{L}_{\rho_0}}$. Hence, combining Lemma \ref{lem: renormalized and bonahon} and Bonahon \cite[Proposition 15]{BonahonCurrents} we have \begin{equation}\label{eq: renormalized BM teich}
\mathbf{J}_{m^{\tn{BM}}(\phi)}(\phi,\psi)=\frac{i(\mathtt{L}_{\rho_0},\mathtt{L}_\rho)}{i(\mathtt{L}_{\rho_0},\mathtt{L}_{\rho_0})}=\frac{i(\mathtt{L}_{\rho_0},\mathtt{L}_\rho)}{\pi^2\vert \chi(S)\vert}.    
\end{equation} \noindent As an interesting consequence, one gets $$\mathbf{J}_{m^{\tn{BM}}(\phi)}(\phi,\psi)=\mathbf{J}_{m^{\tn{BM}}(\psi)}(\psi,\phi)$$\noindent for all $\rho_0,\rho\in \teichrep$. However, if one considers the supremum of all renormalized intersections (rather than just the Bowen-Margulis-renormalized intersection), this symmetry no longer holds: combine Theorem \ref{teo: asymmetric distance flows} with Thurston's example \cite[p.5]{ThurstonStretch}.

Another interesting consequence of Equation (\ref{eq: renormalized BM teich}) is that it recovers a result by Bonahon \cite[p. 156]{BonahonCurrents}. Indeed, combining that equation with
Proposition \ref{prop: BCLS renorm int rigidity} and length spectrum rigidity on $\Teich(S)$, one has $$i(\mathtt{L}_{\rho_0},\mathtt{L}_\rho)\geq \pi^2\vert \chi(S)\vert$$
\noindent for all $\rho_0,\rho\in \teichrep$, with equality if and only if $\rho=\rho_0$.

\bibliography{ref}
\bibliographystyle{alpha}

\end{document}